\documentclass[11pt,a4paper]{article}

\usepackage[T1]{fontenc}
\usepackage[utf8]{inputenc}
\usepackage{amsmath,amssymb,amsthm}
\usepackage{graphicx}
\usepackage{algorithm}
\usepackage{algpseudocode}
\usepackage{booktabs}
\usepackage{longtable}
\usepackage{caption}
\usepackage[margin=2.6cm]{geometry}
\usepackage{setspace}
\usepackage[authoryear,round]{natbib}   %
\usepackage{hyperref}
\usepackage{xcolor}
\usepackage{lmodern}   %
\newcommand{\epivar}{\mu}
\newcommand{\PREF}{}

\hypersetup{colorlinks=true,linkcolor=blue!50!black,citecolor=blue!50!black,
            urlcolor=blue!50!black}

\newtheorem{proposition}{Proposition}
\newtheorem{lemma}{Lemma}

\theoremstyle{definition}
\newtheorem{remark}{Remark}

\newcommand{\wmax}{V}

\title{\bf EndoWake: Modeling Wind by Linear Programming\\
       for Wind Farm Layout Optimization}

\author{%
Martina Fischetti\thanks{Department of Statistics and Operations Research
and Institute of Mathematics (IMUS), University of Seville, Facultad de
Matem\'aticas, Av.~Reina Mercedes s/n, 41012 Sevilla, Spain.
ORCID: 0000-0002-7673-6917.
\emph{Corresponding author}: \texttt{mfischetti@us.es}}
\and
Matteo Fischetti\thanks{Department of Information Engineering, University of
Padova, Via Gradenigo 6/B, 35131 Padova, Italy.
ORCID: 0000-0001-6601-0568}}

\begin{document}
\maketitle

\begin{abstract}
The Wind Farm Layout Optimization problem consists of placing up to a prescribed number of turbines so as to maximize energy
production. One of its main challenges is the wake effect between
turbines, which can strongly affect production and is a cumbersome
function to evaluate and to embed in an optimization framework. We introduce a new endogenous wake model (EndoWake), based on the
idea of treating the wind speed at every grid point as a decision
variable, linked to its upwind neighbors through linear propagation
constraints. This makes it possible to incorporate the wake equations
directly into a Mixed-Integer Linear Programming (MILP) model. We
develop such a MILP model and introduce a new class of
valid ``band'' inequalities, showing that they substantially strengthen
the root bound. Building on this, we design a branch-and-check scheme:
the master problem retains only the site variables together with one
hypograph variable per cell and scenario, integer solutions are
evaluated by an exact linear-time oracle, and the optimality cuts are
problem-specific ``upwind cuts''. Computational results on
a ladder of square grids under a four-sector axial wind rose show
that branch-and-check proves optimality on the non-interpolated EndoWake objective up to
$10\times10$\unskip\ in a fraction of the time of the
monolithic model, and keeps the tighter bound on every open instance,
also on an irregular site with installation costs; on the grids that a
classical interference-matrix model leaves open, the branch-and-check
bound certifies that model's layouts, the better cold-start primal heuristic, with
median gaps of
7.3\%--8.2\%.
\end{abstract}

\noindent\textbf{Keywords:} Discrete optimization; OR in energy; wind
farm layout optimization; wake modeling; branch-and-check.

\section{Introduction}
\label{sec:intro}

The \emph{Wind Farm Layout Optimization} (WFLO) problem asks where to place
up to a prescribed number of wind turbines inside an admissible region so as to
maximize energy production. Turbines interfere: a machine operating in the
wake of another sees a reduced wind speed and produces less power, and since
power grows with the cube of wind speed in the sub-rated regime, even
moderate wake losses translate into substantial revenue losses over the
twenty-five years of a typical concession---millions of euros are involved
in these analyses~\citep{fischetti2020vattenfall,fischetti2023integrated}.
The classical mathematical-programming approaches
discretize the region into candidate sites, evaluate an analytical wake
model offline for every ordered pair of sites, and optimize over the
resulting \emph{interference matrix} of pairwise coefficients
\citep{donovan2005,turner2014,zhang2014,fischetti2016proximity},
multiple-wake deficits included \citep{ulku2019}; a recent
variant keeps a wind-speed variable per site but still feeds it by
precomputed pairwise deficits~\citep{perezrua2023}. We introduce a new
endogenous wake model: the wind speed at
every grid point is a continuous decision variable, linked to its upwind
neighbors by linear propagation constraints, and a turbine reduces
multiplicatively the speed it actually sees; we call this wind-flow
model \emph{EndoWake}. The model has two attractive features: it fits a
Mixed-Integer Linear Programming (MILP) framework directly, and no pairwise
superposition rule has to be chosen, since higher-order interactions follow
from the propagation itself and the flow representation is linear
rather than quadratic in the number of grid points, at a fixed stencil
support and number of scenarios.

Wake models that  let each machine respond to the inflow
shaped by the machines upwind of it are not new
in the engineering literature
\citep{zong2020,bastankhah2021,lanzilao2022,zhang2026}; what is new here is that
the propagation is written \emph{inside} the optimization model, as
linear constraints on decision variables.

Solving the MILP model built on EndoWake \emph{exactly} turned out to be
far from trivial. Finding a feasible solution (an incumbent) is seldom
the difficulty: a greedy construction followed by local search produces
good layouts quickly. Obtaining a good dual bound is. We prove in
Section~\ref{sec:vacuous} that the linear relaxation of the formulation
as written yields the worst bound the instance admits, and that standard
local repairs are not effective. We then introduce a new family of valid
inequalities, the band inequalities, that significantly improve the root
bound, and we develop a decomposition scheme to solve the instances.

To summarize, the main contributions of our work are as follows.

\begin{itemize}
\item \emph{A new, endogenous formulation.} We introduce a MILP model of
wind farm layout in which the flow field is part of the formulation: one
wind-speed variable per cell and scenario, tied to its upwind neighbors
by linear propagation constraints, so that wake aggregation and
higher-order interference follow from the model instead of a
superposition rule fixed in advance, and the formulation grows linearly
rather than quadratically with the number of cells.

\item \emph{The linear relaxation is weak, and local repairs provably fail.}
We construct a uniform fractional point that survives every constraint of
the linear relaxation, so the dual bound
it certifies carries no information. Two local repairs fail as well: a maximally
lifted wake constraint provably leaves the bound where it stands, and the
perspective reformulation of the power block is measured to lower it only
to the wake-free value.

\item \emph{A valid dual bound and the inequalities that repair the root.}
We introduce a family of ``band inequalities'' coupling the layout to the
power variables over a band of cells that the wind direction orders, bounding by
a dynamic program the largest model power that $k$ turbines can produce
inside the band and taking the facets of the upper concave envelope of the resulting function.
In our computational experience, the band inequalities improve
the root bound very significantly, cut both nodes and time on the
largest instance that closes, and lower the median certificate on every instance
left open (Section~\ref{sec:exp-cuts}).

\item \emph{Branch-and-check with upwind cuts, certifying on the exact
scale.} Every certificate above bounds the model objective, which the
interpolated power curve inflates above the exact energy. We therefore
decompose: a master that keeps the site
variables, the packing constraints and one hypograph variable per cell and
scenario, dropping the flow and power blocks altogether, and an oracle
that scores (exactly) the proposed incumbents and returns optimality cuts
to the master. Our optimality cuts are \emph{upwind cuts}: their support
is restricted to the cells upwind of the cell they charge, and their
constant is the oracle value at exactly those cells. On the largest
instance both methods close, the scheme certifies the exact energy in a
fraction of the time the monolithic model needs to certify only its
interpolated surrogate (Section~\ref{sec:exp-bnc}).
\end{itemize}

This paper is methodological: its object is the formulation and how to
solve it exactly. Throughout, ``exact'' is meant with respect to the
model: propagation is a calibrated surrogate of an
engineering wake model, and every optimum, bound and certificate below
is a statement about the formulation as calibrated, not about the
physics it approximates, and holds up to the solver's tolerances.

\paragraph{Companion work.} EndoWake is introduced in this paper.
A companion paper~\citep{companion} takes from it the model,
its maximal-field property, the MILP formulation and the calibration
recipe, and carries them to industrial size by other means---a kernel
distilled from a Reynolds-averaged Navier--Stokes (RANS) flow simulation,
a batched evaluator on graphics processors, and a local search---on
Lillgrund and Horns Rev~1.
Everything that concerns the formulation as an optimization object is
established here: the structure of its linear relaxation, the band
inequalities, the branch-and-check scheme with upwind cuts, and the
certificates; no result here depends on a companion preprint.

The paper is organized as follows. Section~\ref{sec:related} reviews the
related literature, Section~\ref{sec:problem} states the problem in
words and introduces the EndoWake wind model, and Section~\ref{sec:model}
states the MILP model built on it and
proves the structural property---antitonicity of the maximal flow
field---on which everything else rests. Section~\ref{sec:vacuous}
characterizes the linear relaxation and shows that local repairs do not
strengthen it.
Section~\ref{sec:cuts} develops the band inequalities. Section~\ref{sec:bnc} presents the branch-and-check scheme
with upwind cuts. Section~\ref{sec:exp} reports our computational
experience, and Section~\ref{sec:concl} concludes.
The supplementary material gives the proofs omitted from the paper,
the tables for every seed, and additional computational runs.

\section{Related work}
\label{sec:related}

Mathematical programming
approaches to WFLO were initiated by \citet{donovan2005} and \citet{fagerfjall2010}, and systematized by \citet{turner2014}, who compare quadratic-integer and mixed-integer
linear formulations built on a pre-computed interaction matrix;
\citet{zhang2014} extend the comparison to constraint programming and to a
decomposition we return to below, \citet{kuo2016} replace the analytical
deficits by CFD-derived ones inside the same mixed-integer scheme, and
\citet{archer2011} follow the same pattern. All of them
combine pairwise wake deficits, evaluated in free stream by an analytical
model \citep{jensen1983,bastankhah2014}, through a superposition rule fixed
before the layout is known---linear \citep{lissaman1979} or root-sum-square
\citep{katic1986}.

Within mathematical programming, the
state of the art rests on interference-matrix models. \citet{cazzaro2023} optimize layout
and cable routing at industrial sizes, though the layout side is
solved matheuristically; \citet{pedersen2026} cast the integrated problem
as a quota Steiner tree problem with interference and solve it exactly,
with an interference-based splitting strategy, orders of magnitude beyond
a generic MILP solver; and \citet{perezrua2023} embed per-site wind-speed
variables and a stepwise power curve in a MILP solved over adaptive
neighborhoods, with per-iteration gaps but no global certificate;
and the most recent integrated layout-and-cable models
\citep{cao2025,hu2026} keep the same pricing of the wakes, by precomputed
pairwise deficits combined under the root-sum-square rule, paying for the
nonlinearity of that rule with a MINLP. All of
them price the wakes by precomputed pairwise deficits under a fixed
superposition rule, and extend an earlier mathematical-programming
campaign \citep[for an overview, see][]{fischetti2019overview}:
proximity search on interference-matrix models~\citep{fischetti2014proximity,fischetti2016proximity}, joint layout and cable routing with the relaxation
strengthened by combinatorial Benders-like cuts~\citep{fischetti2023integrated}, and industrial deployment at
Vattenfall~\citep{fischetti2020vattenfall}. It is notable that across
this line the layout side
certifies through incumbents and matheuristics; where dual bounds are
attacked head-on~\citep{fischetti2023integrated,pedersen2026}, the means
is a combinatorial strengthening of the relaxation, which is the register
of the present paper.
The metaheuristic line goes further: \citet{cazzaro2022vns} reach
industrial sizes with a variable neighborhood search,  \citet{cazzaro2022multiscale} decompose the design across scales, and \citet{li2025multitype} let the search choose among several turbine types against the levelized cost of energy, all
without ever asking the relaxation for a bound.
The engineering
practice largely works in continuous coordinates and gradients: exact
gradients of the wake model~\citep{guirguis2016}, the boundary-grid
parameterization of \citet{stanley2019}, the Fourier-based FLOWERS model
with analytic gradients~\citep{locascio2024}, and the pseudo-gradient
method of \citet{quaeghebeur2021}; there the question of a certificate
does not arise at all. Closer to our modeling choice, the engineering
literature has moved away from fixed superposition rules: the
momentum-conserving superposition of \citet{zong2020}, the cumulative
wake model of \citet{bastankhah2021} and the recursive wake merging of
\citet{lanzilao2022} all let a wake depend on the inflow the upwind
machines have already shaped, which is the effect our propagation
reproduces by construction; and \citet{king2017} optimize a layout
against a flow field carried as the state of an adjoint computation.
These are evaluators or gradient methods, however: none of these enters a mixed-integer model, and none returns a dual bound.
Making a physical state endogenous to a design problem is, by itself,
the classical move of PDE-constrained optimization, mixed-integer
variants included~\citep{gnegel2021}; what is specific here is the
combination the layout problem affords---binary siting decisions, linear
local propagation, acyclicity along the flow---and the optimization
structure it exposes. To the best of our knowledge, no valid
inequalities coupling the
layout to the flow field have been proposed for WFLO before, which is what
Section~\ref{sec:cuts} supplies.

When
the second stage of a decomposition provides no useful duals, the classical
Benders machinery \citep[for a review, see][]{rahmaniani2017} gives way to logic-based Benders decomposition, in which
the subproblem acts as an oracle and the cuts are derived from
inference~\citep{hooker2003}, the combinatorial Benders cuts of
\citet{codato2006} being the purely combinatorial instance of that
idea. Branch-and-check~\citep{thorsteinsson2001} is
its incremental variant, in which the master is solved once and the oracle
is called at every incumbent; \citet{beck2010} measures when it wins, and
cheap subproblems---ours costs less than a linear program---are its
favorable regime. \citet{forbes2024} give the closest methodological
relative: a logic-based Benders scheme whose subproblem is a
\emph{simulation} evaluating each incumbent, with cuts strengthened under
monotone performance measures. On the optimality side, the integer
L-shaped method~\citep{laporte1993} covers binary first stages with
integer recourse: \citet{angulo2016} improve its alternating strategy, and
\citet{parada2024} disaggregate the cut and strengthen it under monotone
recourse. For WFLO itself, \citet{zhang2014} already proposed a
decomposition of this family: a MIP master over the layout variables, an
exact evaluator of each incumbent under the root-sum-square interference
formula, and two families of logic-based cuts---full-support no-good cuts
carrying a big-$M$, and three-turbine cuts valid by the sublinearity of
the deficits---re-solving the master from scratch at each round, with no
instance closed within an hour.

\section{The problem, and the EndoWake wind model}
\label{sec:problem}

\paragraph{The problem.}
We work in the discrete WFLO setting: we are given a grid of candidate
cells, a cap on the number of (identical) turbines, and a minimum
separation between any two. The objective is to maximize the expected
energy over a finite set of
wind \emph{scenarios}, each prescribing a direction, an undisturbed
(\emph{free-stream}) speed $\wmax_s$ and a probability $p_s$. Wakes
couple the cells: a turbine slows the air downstream, production grows
with the \emph{cube} of the wind speed below the rated regime, and the
value of a position therefore depends on the entire layout, through a
physical field rather than through a distance. Of the power curve
$P(v)$ we use two structural facts. It is
non-decreasing on $[0, v_{\text{cut-out}}]$---nothing below the
\emph{cut-in} speed, a rise roughly like $v^3$, a plateau at the rated
power $P_{\max}$ from the rated speed $v_r$ on---and it is \emph{not}
concave there, being a convex rise followed by a plateau. Monotonicity
is what the standing assumption of Section~\ref{sec:model} secures, and
it is the property the relaxation of Section~\ref{sec:vacuous}
rests on; the failure of concavity is what the model has to work around,
since a maximization cannot use the curve itself as an upper bound.

\paragraph{Our new formulation: EndoWake.}
The standard treatment evaluates, once and offline, the speed deficit that
each candidate position would cause at each other one, constructing the
\emph{interference matrix}, and then combines
those pairwise deficits by a superposition rule fixed before the layout is
known. EndoWake precomputes no interference matrix and fixes no
superposition rule. Instead, a regular grid is superimposed on the site,
and the wind speed is carried as a
\emph{decision variable at every cell}. The cells are grouped into
\emph{propagation planes} orthogonal to the dominant grid axis of the scenario, ordered
from upwind to downwind---if the wind blows from north to south, the
planes are the rows of the grid, taken from the top one down---and the
speed at a cell is bounded by a weighted
average of the speeds at a few cells of the preceding plane, pulled back
toward the free stream by a \emph{recovery} term.

Each cell reads a fixed window of cells on the preceding plane---the
one directly upwind of it, and two on either side of that one, at
transversal offsets $-2,\dots,2$: this window is the \emph{stencil}. The
weights with which the cell averages the window form the \emph{kernel}, a
vector of five non-negative numbers summing to one and symmetric about its
center. Summing to one means that an undisturbed flow stays undisturbed;
being spread over five offsets rather than concentrated on the central one
means that whatever deficit arrives is smeared a little sideways at every
plane, which is how a wake widens as it travels. Two parameters fix how
much of that lateral spread there is, and a third, the recovery rate, fixes
how fast the profile returns toward the free stream. All three are
\emph{input data} and not decisions: the kernel is fitted once, before any
optimization, against a reference wake model, and the optimizer takes it as
given. That
calibration---a weighted least-squares fit against an analytical wake
model, described at the end of Section~\ref{sec:model}---leaves a
residual below $8\%$ of the reference field at the resolution used here.
Applying the same
kernel plane after plane is a discrete convolution along the flow, which
is the closed form the single-turbine field takes on an unbounded grid.

A turbine installed at a cell acts
by multiplying the speed handed to the cells it shadows by a factor
$\rho < 1$. Wake merging, higher-order effects and the recovery are then
consequences of the propagation rather than modeling choices. The price
is that the flow becomes part of the feasible region:
that is the object this paper studies.

\section{The MILP model}
\label{sec:model}

We start with the notation used throughout. As already mentioned, the
site is discretized into a regular grid of cells with spacing
$\Delta$, expressed in rotor diameters $D$; a wind scenario prescribes a
direction, a free-stream speed and a probability (the probabilities
summing to one over the scenario set). We assume that every free-stream
speed satisfies
$\wmax_s \le v_{\text{cut-out}}$, so that the modeled power is
non-decreasing on the whole range $[0,\wmax_s]$. Our instances carry one scenario per
direction, so we
speak of sectors and of scenarios interchangeably. The following table
collects the symbols used throughout.

{\small
\begin{longtable}{p{0.22\textwidth}p{0.70\textwidth}}
\toprule
symbol & meaning \\
\midrule
\endfirsthead
\midrule
symbol & meaning \emph{(continued)} \\
\midrule
\endhead
\midrule
\endfoot
\bottomrule
\endlastfoot
$\mathcal{C}$, $\mathcal{A}$ & grid cells; allocatable cells,
$\mathcal{A} \subseteq \mathcal{C}$ \\
$\mathcal{S}$; $\theta_s$, $\wmax_s$, $p_s$ & wind scenarios; direction,
free-stream speed and probability of scenario $s$ \\
$N_{\min}$, $N_{\max}$; $C_c$ & bounds on the number of machines
installed; installation cost at $c$, in kW of expected power, taken to be zero in the main campaign \\
$\mathcal{Q}$ &  a family of cliques of the \emph{conflict
graph}---one node per allocatable cell, an edge between two cells closer
than the minimum separation---that covers every edge. Ours consists of
the $3 \times 3$ windows with
top-left cell at every cell of the grid, truncated at its border (the
one-cell window at the far corner is dropped): $|\mathcal{Q}| = n^2 - 1$
on an $n \times n$ grid \\
 &  \\
$K = (K_{-2},\dots,K_2)$ & propagation kernel,
$K = \bigl(\tfrac{\beta}{2},\tfrac{\alpha}{2},
1-\alpha-\beta,\tfrac{\alpha}{2},\tfrac{\beta}{2}\bigr)$,
non-negative and summing to one, the same for every sector \\
$\alpha$, $\beta$, $\gamma$ & \leavevmode{}propagation parameters:
transversal spreading onto the nearest and the next-nearest offset, and
recovery toward free stream, with $\gamma \in [0,1]$. They are calibrated once, as described below,
and are the same for every instance \\
 &  \\
$\mathrm{up}_s(c,m)$ & the five upwind neighbors of cell $c$ in the
stencil of scenario $s$, $m = -2,\dots,2$ \\
$\rho = 1-2a \in [0,1)$ & fraction of the wind speed surviving immediately
behind a rotor; $a \in (0,\tfrac12]$ is the calibrated strength of the
lumped source, not the momentum-theory induction factor:
$a = 0.3534$, hence
$\rho = 0.293$ \\
$\mathcal{F}_s(c)$ & downstream footprint of a turbine at $c$: the cells
of the next propagation plane at most $h$ cells off the flow line
through $c$, with $h = \mathrm{round}\bigl(D/(2\Delta) - \tfrac12\bigr)$
rounding ties to even ($h = 0$, the single downstream cell, at
$\Delta = D$) \\
$P$, $\widehat{P}$, $\bar{P}$ & the three power curves used
throughout: the exact power curve $P$ of Section~\ref{sec:problem}; the
piecewise-linear interpolant $\widehat{P}$ the breakpoints below define,
which is the representation the model uses under SOS2 conditions; and the
upper concave envelope $\bar{P}$ of the breakpoints, which is what the
$\lambda$-block of \eqref{eq:mip-power} yields once the SOS2 conditions
are dropped, hence the power curve the linear relaxation sees. All three
are non-decreasing on $[0,\wmax_s]$, and $\bar{P} \ge \widehat{P} \ge P$,
the latter because $P$ is convex between consecutive breakpoints and
vanishes below cut-in, so its chords lie above it \\
$v_r$, $P_{\max}$ & rated speed and rated power of one machine: the point
at which $P$ flattens, and the
value it is capped at. For the turbine of Section~\ref{sec:exp},
$v_r = 10.34$~m/s \\
$(v_b,P_b)$, $b = 1,\dots,B$ & breakpoints of power curve $P$; by
convention they start at $0$, include the cut-in speed and the rated speed $v_r$, and reach
at least $\wmax_s$ \\
\end{longtable}}
\addtocounter{table}{-1}

Formally, a \emph{propagation plane} is the set of cells sharing the same
coordinate along the dominant axis of the scenario. Since every cell
depends only on cells of the preceding plane, the dependency graph is
acyclic for every wind direction---each scenario is assigned the octant
of its dominant axis, and the cells are swept plane by plane along that
axis---and this is what
Proposition~\ref{prop:fixpoint} below rests on. At the resolution
$\Delta = D$ used in all our experiments the footprint
$\mathcal{F}_s(c)$ is the single cell directly downstream of $c$, one rotor diameter wide.

One convention has to be stated, because it is a modeling choice and not
an omission: a cell of the first plane, and a cell whose stencil reaches
past a lateral edge, read the \emph{free-stream} speed $\wmax_s$ for the
neighbors that fall outside $\mathcal{C}$. Outside the grid the flow is
undisturbed, rather than unconstrained. The convention is the only one
available without modeling the surroundings, and it is mildly optimistic
for the cells it touches: a machine at the upwind edge is guaranteed clean
wind by construction, whereas one in the interior has to earn it. When edge effects matter, a frame of flow-only cells around
$\mathcal{A}$ makes the allocatable region interior to the flow domain at
no binary cost; the experiments of Section~\ref{sec:exp} use no frame.

\paragraph{The optimization model.} We use three groups of variables. A
binary $x_c$, $c \in \mathcal{A}$, selects the cells that receive a
turbine. A continuous $w^s_c \in [0,\wmax_s]$, one per cell and scenario,
is the wind speed the model assigns to $c$ under $s$; this is the variable
that distinguishes the construction, since the flow field is decided inside
the optimization rather than tabulated before it. Some statements read
better in terms of the shortfall of that speed below the free stream, the
\emph{deficit} $d^s_c = \wmax_s - w^s_c$; this is a change of variables and
not a further datum, and we use whichever of the two is clearer at the
point of use. A continuous
$z^s_c \in [0,P_{\max}]$, $c \in \mathcal{A}$, is the power credited to
cell $c$ under $s$; the auxiliary weights $\lambda^s_{cb} \ge 0$ carry the
piecewise-linear representation of the power curve. The model reads:
\begin{align}
\max\ & \sum_{s \in \mathcal{S}} p_s \sum_{c \in \mathcal{A}} z^s_c
        \;-\; \sum_{c \in \mathcal{A}} C_c\, x_c
\label{eq:mip-obj}\\[2pt]
\text{s.t.}\quad
& w^s_c \;\le\; \gamma \wmax_s + (1-\gamma)
  \sum_{m=-2}^{2} K_m\, w^s_{\mathrm{up}_s(c,m)},
& & c \in \mathcal{C},\ s \in \mathcal{S},
\label{eq:mip-flow}\\
& w^s_{c'} \;\le\; \rho\, w^s_c + \wmax_s\,(1-x_c),
& & c \in \mathcal{A},\ c' \in \mathcal{F}_s(c),\ s \in \mathcal{S},
\label{eq:mip-wake}\\
& \left.
\begin{aligned}
& \textstyle\sum_b \lambda^s_{cb} = 1, \quad
  \textstyle\sum_b v_b \lambda^s_{cb} = w^s_c, \quad
  \lambda^s_{c\cdot}\ \text{SOS2},\\
& z^s_c \le \textstyle\sum_b P_b \lambda^s_{cb}, \quad
  z^s_c \le P_{\max}\, x_c
\end{aligned}\ \right\}
& & c \in \mathcal{A},\ s \in \mathcal{S},
\label{eq:mip-power}\\
& \left.
\begin{aligned}
& \textstyle\sum_{c \in Q} x_c \;\le\; 1,\ \ Q \in \mathcal{Q}; \qquad
  \\
& N_{\min} \le \textstyle\sum_{c \in \mathcal{A}} x_c \le N_{\max},
\quad x_c \in \{0,1\},\ c \in \mathcal{A}
\end{aligned}\ \right\}
& &
\label{eq:mip-dist}
\end{align}

The objective \eqref{eq:mip-obj} maximizes expected production over the
scenarios net of installation costs. The wind flow rows \eqref{eq:mip-flow} cap
each speed by a convex combination of the free-stream value and of five
speeds one plane upwind: since $K$ is non-negative and sums to one, the
right-hand side can neither create wind nor drive it negative, whatever the
layout.  The wake rows \eqref{eq:mip-wake} state that a
turbine at $c$ scales down, by the factor $\rho$, whatever speed reaches
it; the big-$M$ constant $\wmax_s$ switches the implication off when
$x_c = 0$. Note that the reduction is multiplicative on the \emph{local}
speed: a machine already deep in a wake casts a proportionally weaker wake,
which is precisely the higher-order effect that pairwise coefficients
cannot express. The block \eqref{eq:mip-power} represents the power curve:
under the SOS2 condition---at most two consecutive weights positive, a
device due to \citet{beale1970}---the pair
$(w^s_c, \sum_b P_b\lambda^s_{cb})$ lies on the interpolant
$\widehat{P}$, so the block states $z^s_c \le \widehat{P}(w^s_c)$, while
the final row switches production off where no turbine stands. Finally,
\eqref{eq:mip-dist} encodes the minimum separation as \emph{clique
inequalities}, one per clique of family $\mathcal{Q}$,
and the cardinality bounds close the block. The
aggregated form $\sum_{v \in N(c)} x_v \le |N(c)|\,(1-x_c)$, with $N(c)$
the cells conflicting with $c$, is a weaker encoding. Every
wind
constraint is written with ``$\le$'' rather than ``$=$'': the objective is
non-decreasing in $w$ through $z$, so the inequalities are tight where it
matters, and the field is defined by a maximal point rather than by an
equation system, as we now make precise. The model has $|\mathcal{A}|$
binaries and $O(|\mathcal{S}|\,|\mathcal{C}|)$ continuous variables and
rows: linear in the grid, against the $\Theta(|\mathcal{C}|^2)$ entries of
an interference matrix.

\paragraph{The maximal field, and its antitonicity.} Fix a scenario $s$
and a possibly fractional solution $x \in [0,1]^{\mathcal{A}}$, and let
$T_x$ assign to each
cell the \emph{smallest} of the right-hand sides that \eqref{eq:mip-flow}
and \eqref{eq:mip-wake} impose on it, so that the wind constraints read
$w \le T_x(w)$ and the feasible fields form the set
$\mathcal{W}(x) = \{ w \in [0,\wmax_s]^{\mathcal{C}} :
w \le T_x(w) \}$.
Everything in this paper rests on the following proposition, and above all
on its part~(d); we therefore prove the parts we use.

\begin{proposition}[Maximal field]
\label{prop:fixpoint}
For every $x \in [0,1]^{\mathcal{A}}$ and every scenario $s$:
\begin{enumerate}
\item[(a)] $T_x$ maps $[0,\wmax_s]^{\mathcal{C}}$ into itself, so the
variable bounds are consistent with the constraints;
\item[(b)] $\mathcal{W}(x)$ has a greatest element $w^{*}(x)$, computable
in $O(|\mathcal{C}|)$ time by one forward sweep of the grid in the order
of the propagation planes;
\item[(c)] any objective non-decreasing in $w$ attains its maximum over
$\mathcal{W}(x)$ at $w^{*}(x)$;
\item[(d)] \emph{(antitonicity)} if $x' \ge x$ componentwise, then
$w^{*}(x') \le w^{*}(x)$ componentwise: installing additional turbines can
only lower the maximal field, everywhere.
\end{enumerate}
\end{proposition}

\begin{proof}
Two elementary properties of $T_x$ drive the proof. First, $T_x$ is
\emph{monotone in $w$}: every candidate right-hand side (the flow bound
and each wake bound) has non-negative coefficients on $w$, so each is
non-decreasing in $w$, and so is their minimum. Second, $T_x$ is
\emph{antitone in $x$} for fixed $w$: the layout enters only through the
terms $\wmax_s(1-x_c)$ of the wake bounds, each non-increasing in $x$,
so $x' \ge x$ implies $T_{x'}(w) \le T_x(w)$ pointwise.

(a) If $0 \le w \le \wmax_s$ componentwise, every wake bound is
non-negative and the flow bound lies in $[\gamma\wmax_s, \wmax_s]$,
being a convex combination of values in $[0,\wmax_s]$ blended with
$\wmax_s$; the minimum of the candidates is therefore non-negative, and it
is at most $\wmax_s$ because the flow bound is.

(b) The value $T_x(w)_c$ depends only on entries of $w$ strictly upwind of
$c$---the flow bound reads the preceding plane, and a wake bound at $c$
reads the cell one plane upwind whose footprint contains $c$---so the
dependency graph is acyclic and the planes give a topological order.
Define $w^{*}$ by forward substitution with equality,
$w^{*}_c = T_x(w^{*})_c$, in that order: each value depends only on values
already computed, so $w^{*}$ is well defined, and it lies in
$[0,\wmax_s]$ by (a) and induction along the order. For any
$w \in \mathcal{W}(x)$, the same induction gives
$w_c \le T_x(w)_c \le T_x(w^{*})_c = w^{*}_c$, the middle step by
monotonicity and the inductive hypothesis on the upwind cells; hence
$w^{*}$ is the greatest element.

(c) Immediate from (b).

(d) Induction along the topological order. On the first plane both maximal
fields equal the flow bound at free stream. At a cell $c$, assume
$w^{*}(x') \le w^{*}(x)$ on every cell strictly upwind of $c$. Then
$w^{*}(x')_c = T_{x'}(w^{*}(x'))_c
\le T_{x}(w^{*}(x'))_c
\le T_{x}(w^{*}(x))_c = w^{*}(x)_c$,
the first inequality by antitonicity of $T$ in $x$ at the fixed field
$w^{*}(x')$, and the second by monotonicity of $T_x$ in $w$ together with
the inductive hypothesis, since $T_x(\cdot)_c$ reads only upwind entries.
\end{proof}

Note that part (b) is also an algorithm: evaluating a layout exactly is one
sweep of the grid, five multiply--add operations and one comparison
per cell and scenario, plus one power-curve lookup per installed
cell, which is the
oracle of Section~\ref{sec:bnc} and the evaluator behind every exact-power figure of Section~\ref{sec:exp}. Part (d) is the single structural
hypothesis on which the validity of every cut in this paper rests---the
band inequalities of Section~\ref{sec:cuts} and the
upwind cuts of Section~\ref{sec:bnc} alike. The monotonicity part (c) asks of the objective
holds on the whole modeled range by the standing assumption
$\wmax_s \le v_{\text{cut-out}}$.

\paragraph{Calibration and instances.} The parameters
$(\alpha,\beta,\gamma)$ and the source strength $2a$ are calibrated by
weighted least squares against the
Bastankhah--Port\'e-Agel wake model~\citep{bastankhah2014} at thrust
coefficient $C_T = 8/9$ and wake growth rate $k^* = 0.04$: the single-turbine field of the kernel---the
discrete convolution above, fed by the same lumped source the MILP
injects---is matched to the analytical deficit on the operational window
of $3\,D$--$12\,D$\
downstream, where actual turbine spacings live, over a domain reaching
$20\,D$ downstream and $5\,D$ laterally; the amplitude is recovered in
closed form at each evaluation, the field being linear in the source.
The error measure is the RMS of the residual over the RMS of the
reference on the window. At $\Delta = D$, the resolution of every
experiment below, the residual is 7.7\%, the
fitted values being
reported in the setup of Section~\ref{sec:exp}. Refitted at growth
rates from 0.03\ to
0.075\unskip, the kernel keeps a residual
of 6.7\%--8.2\%\unskip;
on the $8\times8$ and $10\times10$ instances, the optimal layouts returned
are unchanged, up to symmetry, from $k^* = 0.03$
to $0.05$\unskip, while at
$k^* = 0.06$ each adds one turbine and
loses 3.9\%\ and
2.4\%\ under the reference kernel
(2.4\%\ and none with the number of
turbines fixed). Figure~\ref{fig:calib}
shows the calibration as a field. That $\beta$ is driven
to zero is an outcome of the fit, not a modeling choice: the theory
above is stated for the full five-point stencil.
The results below are statements about the model as calibrated, not about
the fidelity of the fit.
The same five-point kernel, distilled from a RANS field rather than
from the analytical wake, is compared in the companion
paper~\citep{companion} with the RANS field it is distilled from, and
that field with the production records of Lillgrund.

\begin{figure}[htbp]
\centering
\includegraphics[width=0.85\textwidth]{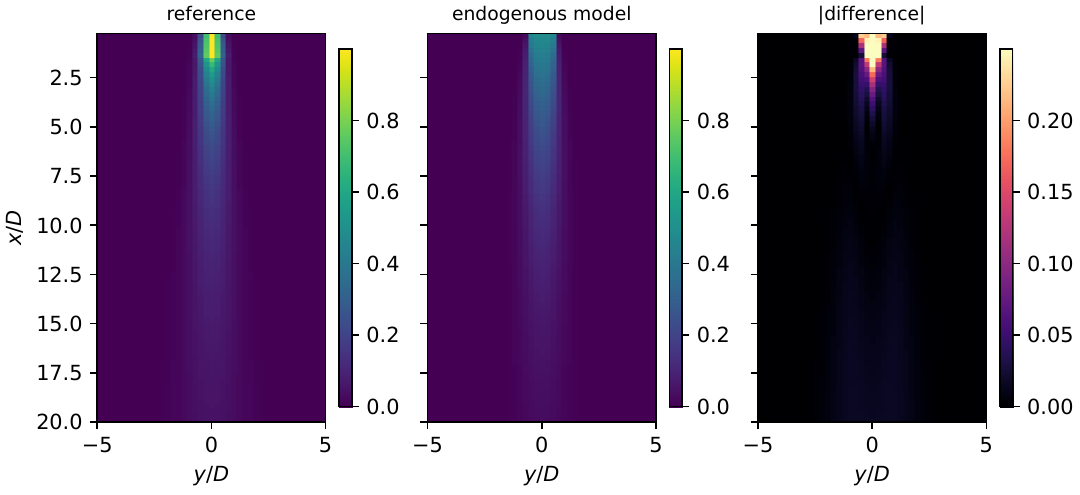}
\caption{The calibration seen as a field, for the northern sector:
reference wake model (left), calibrated EndoWake field (center), and
pointwise absolute difference (right).}
\label{fig:calib}
\end{figure}

\section{A weak relaxation}
\label{sec:vacuous}

In this section we study the strength of the linear relaxation of the
model. The statement below uses the concave envelope $\bar{P}$ of the symbol
table: once the SOS2 conditions are dropped, the $\lambda$-block of
\eqref{eq:mip-power} admits exactly the pairs $(w, z)$ with $z \le
\bar{P}(w)$, so $\bar{P}$ is the power curve the relaxation actually
prices.

\begin{proposition}[Vacuous root bound]
\label{prop:vacuous}
Consider the relaxation of model \eqref{eq:mip-obj}--\eqref{eq:mip-dist}
obtained by replacing $x \in \{0,1\}^{\mathcal{A}}$ with
$x \in [0,1]^{\mathcal{A}}$ and dropping the SOS2 conditions, with no
installation cost. Let $\bar{x} = N_{\max}/|\mathcal{A}|$. If
$\bar{x} \le \min\bigl\{ \min_s \bar{P}(\wmax_s/2)/P_{\max},\;
1/q_{\max},\; 1/2 \bigr\}$,
where $q_{\max}$ is the size of the largest clique of the family
$\mathcal{Q}$ in \eqref{eq:mip-dist},
then the optimal value of the relaxation is exactly $N_{\max} P_{\max}$.
\end{proposition}

\begin{proof}
The value is at most $N_{\max}P_{\max}$, by summing
$z^s_c \le P_{\max}x_c$ over $c$ against $\sum_c x_c \le N_{\max}$. It is
attained by the uniform fractional point $x_c = \bar{x}$,
$w^s_c = \hat{w}_s := \min\{\wmax_s,\,(1-\bar{x})\wmax_s/(1-\rho)\}$,
$z^s_c = \bar{x}P_{\max}$. At that point \eqref{eq:mip-wake} reads
$w^s_{c'} \le \rho\,w^s_c + \wmax_s(1-\bar{x})$, whose largest uniform
solution is exactly $\hat{w}_s$; the power block, with the SOS2 conditions
dropped, admits $z^s_c = \bar{x}P_{\max}$ as soon as
$\bar{x}P_{\max} \le \bar{P}(\hat{w}_s)$, and $\hat{w}_s \ge
(1-\bar{x})\wmax_s \ge \wmax_s/2$ since $\bar{x} \le 1/2$, with $\bar{P}$ nondecreasing, so
the condition suffices; and the
separation rows read $q_{\max}\,\bar{x} \le 1$. Constraint
\eqref{eq:mip-flow} holds because $\sum_m K_m = 1$ and
$\hat{w}_s \le \wmax_s$, and cardinality holds with equality by the choice
of $\bar{x}$.
\end{proof}

Two remarks say where the difficulty is. Note first what
the condition of Proposition~\ref{prop:vacuous} actually asks. All three terms are constants of order
one: for the turbine of Section~\ref{sec:exp} the power ratio is
0.48\unskip, and at the resolution and
separation of our experiments the largest clique has $q_{\max} = 9$ cells---so the
relaxation is vacuous whenever the average turbine density is below one
turbine every nine cells. So is every instance of
the main ladder of Section~\ref{sec:exp} (square grids, no installation cost) that the solver leaves open, where $\bar{x} = (n+2)/n^2 \le 1/9$ from $n = 12$ on. The four smallest grids of the ladder sit above the threshold, and on 3\ of them the packing rows bind before the cardinality one and the root falls short of the trivial value; Section~\ref{sec:exp-cuts} measures what the band inequalities are left to repair there. The
condition is not a technicality that a larger instance escapes; it is the
operating regime.

Note second that $N_{\max}P_{\max}$ is even larger than the wake-free
maximum, the value of a farm whose machines never shadow one another:
$N_{\max}\sum_s p_s \widehat{P}(\wmax_s)$ on the model's own power curve
and $N_{\max}\sum_s p_s P(\wmax_s)$ on the exact one, smaller by
8.2\%\ and
9\%\ respectively for the turbine of
Section~\ref{sec:exp}, whose scenarios all carry $\wmax_s = 10$~m/s. The relaxation therefore discards the
wake interaction \emph{and} the power curve, pricing every fractional
turbine at rated output. A reader would be tempted to blame the loss on
the big-$M$ constant of \eqref{eq:mip-wake}---but $\rho$ does not appear in
that condition at all: shrinking it only lowers the certificate
from $\wmax_s$ to $\hat{w}_s$, which the power curve absorbs. The rest of
this section shows the blame to be misplaced on both counts:
each of the two losses admits
a per-cell repair that is exact, maximal, or both, and neither repair takes
the bound past the wake-free value---nor moves it at all once the
inequalities of Section~\ref{sec:cuts} are in place.

\subsection[Perspective reformulation and lifted wake constraint]{Perspective reformulation and lifted wake constraint}
\label{sec:lifted}

The power loss suggests a standard repair.
The power block has the on/off structure of the \emph{perspective
reformulation}~\citep{frangioni2006,gunluk2010,hijazi2012}, and because the
curve enters piecewise-linearly the ideal per-cell convexification is
itself polyhedral: on our instances, where $\wmax_s < v_r$ keeps
the range sub-rated, it is two linear inequalities,
$z^s_c \le \widehat{P}(\wmax_s)\,x_c$ and $z^s_c \le
\bigl(\widehat{P}(\wmax_s)/\wmax_s\bigr)\, w^s_c$, which cost nothing to impose.
They repair the per-cell link between production and installation, and
that replaces $P_{\max}$ by $\widehat{P}(\wmax_s)$ in the trivial bound:
the bound falls to the wake-free value discussed above and stops there,
on the
4\ instances of the ladder of
Section~\ref{sec:exp} we test it on (relative
discrepancy $9\cdot 10^{-16}$\unskip), and imposed on top of the band
inequalities of Section~\ref{sec:cuts} it leaves every root bound
unchanged, by $5\cdot 10^{-10}$\,kW in the worst
case---solver noise.
What Proposition~\ref{prop:vacuous} certifies is not a per-cell defect that
a tighter indicator link could repair.

The wind variables admit the same treatment, and the outcome is the same.
In an axial sector, at the resolution where the footprint of a
cell is the single cell downstream, and when the weight $\kappa_0$ of a
cell's own field in the bound of its footprint is at least $\rho$, the
wake row \eqref{eq:mip-wake} can be lifted over a clique of mutually
conflicting cells; the supplementary material states and proves the
maximally lifted form. With every wake row replaced by it, the uniform
point of Proposition~\ref{prop:vacuous} retreats to a slightly lower
field and survives, so the relaxation is still worth $N_{\max}P_{\max}$
under a density condition that holds on every instance of
Section~\ref{sec:exp} where Proposition~\ref{prop:vacuous} applies: the largest total lifting
across the campaigns is $A_s = 11.4$\,m/s, the
resulting density threshold is 0.274, and the
densest instance we test has $\bar{x} = 0.222$,
a factor 1.2\ to spare.

Lifting is acting on the wrong \emph{variables}: it lowers the wind of one
cell, which the relaxation absorbs by retreating to a slightly lower
uniform field, while the objective reads the power variables, which
$z^s_c \le P_{\max}x_c$ keeps priced at rated output however small
$\bar{x}$ is. In our computational experience, the bound remains
vacuous even after adding the family of single-source deficit inequalities
$w^s_c \le \wmax_s - d^s_{uc}\,x_u$, with the \emph{single-source
deficit} $d^s_{uc}$ the deficit at $c$ under the maximal field of the
single-turbine layout $x = e_u$, for every source $u$ and every cell $c$
downwind of it, whose validity follows directly from antitonicity
(Proposition~\ref{prop:fixpoint}(d)): they act on the wind of one cell at a
time, and the relaxation absorbs them in the same way. What moves the bound
must couple the layout to the \emph{power} variables, and in aggregate;
that is what the band inequalities of Section~\ref{sec:cuts} do.

\section{Band inequalities}
\label{sec:cuts}

To strengthen the model, we introduce here a family of valid inequalities that couples the layout to
the power variables over sets of cells. It requires no auxiliary
separation problem, and it is generated once per scenario before the
search begins.

\begin{figure}[htbp]
\centering
\includegraphics[width=0.55\textwidth]{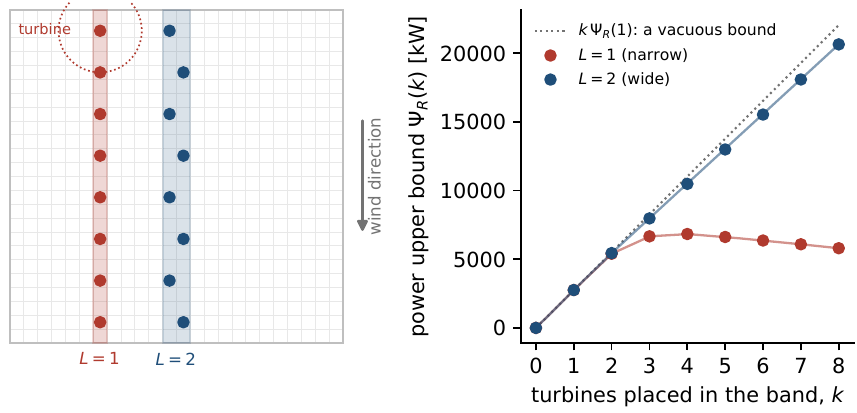}
\caption{Band inequalities on the $24\times24$ instance, north wind,
at the resolution and separation of Section~\ref{sec:exp}. Left: bands of width $L=1$ and $L=2$, each
with the longest chain it can hold. The wide chain
zigzags, so that consecutive
machines sit off each other's axis. Right: the bound $\Psi_R(k)$ and
its upper concave envelope, against the linear bound $k\,\Psi_R(1)$ (dotted).}
\label{fig:bands}
\end{figure}

\paragraph{Band inequalities (BI).}
We first give the rationale of these inequalities, with the help of
Figure~\ref{fig:bands}. Fix a single scenario---say a north wind---and take
a band of consecutive columns of the grid, narrow enough that each row of
the band can hold at most one turbine. Whatever the layout, the turbines
inside the band are then strung along the flow, one per row at most, so
each of them sits downwind of all the others before it. That chain is what
a dynamic program can walk to compute, for each $k$, an upper bound on the
power that $k$ turbines placed in the band could produce. The
resulting function of $k$ is concave-looking but need not be concave, so we
take its piecewise-linear concave envelope; each facet of that envelope is
one linear inequality, bounding the production of the band---the sum of its
$z$ variables---by an affine function of the number of turbines it
holds---the sum of its $x$ variables. The inequality couples the layout to
the power directly, over a set of cells rather than one at a time, which is
what no inequality acting on the wind of a single cell can do
(Section~\ref{sec:lifted}).

To be more specific, recall the \emph{propagation planes} of
Section~\ref{sec:model}---the rows of the grid for a north wind, and in
general the lines of cells orthogonal to the dominant axis. Let a
\emph{band} $R$
be the allocatable cells of a strip of $L$ (say) consecutive cells across the flow, extended over
the whole domain along it, and take the width $L$ below the minimum
separation $\sigma$, expressed in cells. The choice of $L$ is what makes the
construction work:
two cells of the same propagation plane inside $R$ are then at most
$(L-1)\Delta$
apart, closer to each other than the minimum separation allows, so
\emph{at most one of them can carry a turbine}. The turbines of a band are
therefore totally ordered along the flow (each one lies strictly downwind
of the previous) and it is that chain, rather than the two-dimensional
geometry, that the dynamic program below walks. (We generate every width $L = 1, \dots, \sigma - 1$, at every offset
across the flow and for every scenario, which at the separation of
Section~\ref{sec:exp} means widths 1 and 2 at all $n - L + 1$
positions.)

Let $K_R$ be an upper bound on the number of turbines that $R$ can hold,
and $\Psi_R(k)$ an upper bound on the power that $k \le K_R$ turbines
placed inside $R$ can produce. We compute both by dynamic programming over
the cells of $R$ in flow order, with the last installed cell as the state
and, as the transition reward, the power a turbine would produce if the
only deficit it felt were the \emph{floor deficit} of its offset from the
immediate predecessor: the minimum, over all the positions the source can
take in the domain, of the deficit it casts at that along-flow distance
and lateral displacement---formally, with $d^s_{uv}$ the
single-source deficit of Section~\ref{sec:lifted},
\begin{equation}
\underline{\delta}_s(g,d) \;=\; \min\bigl\{\, d^s_{uv} \;:\; u \in
\mathcal{A},\ v \in \mathcal{C},\ v \text{ lies $g$ planes downwind of
$u$ and $d$ cells aside}\,\bigr\},
\label{eq:floor}
\end{equation}
a table filled once per scenario from the $|\mathcal{A}|$ single-source
sweeps, at $O(|\mathcal{A}|\,|\mathcal{C}|)$ time. The minimum is what the lateral boundary
requires: the same offset produces different deficits at different
absolute positions---milder near the edge of the domain, where part of
the wake is lost---and charging anything above the floor could cut off a
boundary-hugging layout. Three deliberate relaxations are thus built in,
and each errs in the safe direction for its own reason. The minimum
separation is enforced between consecutive turbines only, which is a
combinatorial relaxation: it can only enlarge the set of chains the
program walks, so the value it returns can only grow. Each turbine is
charged one deficit, its predecessor's, even though every turbine upwind
contributes to the field it actually faces; this is
Proposition~\ref{prop:fixpoint}(d), since dropping sources can only raise
the maximal field. And that one deficit is charged at its floor, which is
a minimum by construction. The last two relaxations also use that
$\widehat P$ is non-decreasing, which the breakpoints give by construction.
The cost is
$O(|R|^2 N_{\max})$ per band, on top of the deficit table, and the floor deficit depends only on the offset
between two cells, so $\Psi_R$ is a function of the shape of $R$ and not
of its position: one dynamic program per band shape and scenario serves the
whole family.

With $\Psi_R$ in hand, for every facet $j$ of the \emph{upper} concave
envelope of the points $(k,\Psi_R(k))$, $k = 0,\dots,K_R$, the following
\emph{band inequalities} hold true:
\begin{equation}
\sum_{c \in R} z^s_c \;\le\; a_j \sum_{c \in R} x_c + b_j ,
\qquad
\sum_{c \in R} x_c \;\le\; K_R ,
\label{eq:cut-t2}
\end{equation}
where $(a_j,b_j)$ is the line supporting facet $j$. The quantifier is
worth reading twice: one inequality per facet of the hull, and not one per
pair of consecutive points, because $\Psi_R$ need not be concave and the
chords between consecutive points would then fall below it.

The following statement records that $\Psi_R$ is a upper bound,
hence the validity of \eqref{eq:cut-t2}.
\begin{proposition}[Validity of the band inequalities]
\label{prop:bi}
Let $R$ be a band of width $L < \sigma$, fix a scenario $s$, and let
$K_R$ and $\Psi_R$ be computed by the dynamic program with a
non-decreasing power representation $\widetilde{P}$ at least as large as
the one the model imposes on $z^s$ ($\widehat{P}$ for model
\eqref{eq:mip-obj}--\eqref{eq:mip-dist}, the exact $P$ for the master of
Section~\ref{sec:bnc}). Then for every feasible layout $x$ with support
$X$ and every $(w^s, z^s)$ feasible for the model at $x$,
$|X \cap R| \le K_R$ and $\sum_{c \in R} z^s_c \le \Psi_R(|X \cap R|)$;
hence \eqref{eq:cut-t2} holds at every feasible point.
\end{proposition}
\begin{proof}
Let $X \cap R = \{b_1, \dots, b_k\}$, listed along the flow. Two cells of
$R$ on one propagation plane are at most $(L-1)\Delta < \sigma\Delta$
apart, so no two of the $b_i$ share a plane and the order is strict; and
consecutive $b_{i-1}, b_i$ are two turbines of a feasible layout, so
their offset $(g_i, d_i)$ satisfies $g_i^2 + d_i^2 \ge \sigma^2$, that
is, $(b_{i-1}, b_i)$ is a transition the dynamic program allows.
For $i \ge 2$ we have $x \ge e_{b_{i-1}}$ componentwise, so by
Proposition~\ref{prop:fixpoint}(b,d)
\[
w^s_{b_i} \;\le\; w^{*}_{b_i}(x) \;\le\; w^{*}_{b_i}(e_{b_{i-1}})
\;=\; \wmax_s - d^s_{b_{i-1} b_i} \;\le\; \wmax_s -
\underline{\delta}_s(g_i, d_i),
\]
the last step by \eqref{eq:floor}, while $w^s_{b_1} \le \wmax_s$; and
$z^s_c \le \widetilde{P}(w^s_c)$ with $\widetilde{P}$ non-decreasing,
$z^s_c = 0$ off $X$. Summing over $R$,
$\sum_{c \in R} z^s_c \le \widetilde{P}(\wmax_s) + \sum_{i \ge 2}
\widetilde{P}\bigl(\max\{0, \wmax_s - \underline{\delta}_s(g_i,d_i)\}\bigr)$,
which is the reward the dynamic program assigns to the chain
$b_1, \dots, b_k$, hence at most $\Psi_R(k)$. In
particular $\Psi_R(k) > -\infty$, so $k \le K_R$. Finally the upper
envelope lies above every point $(k, \Psi_R(k))$, and $\sum_{c \in R}
x_c = k$.
\end{proof}

Note that the coefficients $(a_j, b_j)$ are floating-point numbers, but
their representation error lies orders of magnitude below the feasibility
tolerance of the solver.

Two details decide validity: the envelope of the points
$(k,\Psi_R(k))$ must be the \emph{upper} hull, the lower one cutting off
feasible layouts; and $\Psi_R$ must be evaluated with a power
representation at least as large as the one the model uses---the SOS2
interpolant $\widehat{P}$ for the monolithic model, the true curve $P$
only in the master of Section~\ref{sec:bnc}, whose power representation
is the oracle itself.

\section{Branch-and-check with upwind cuts}
\label{sec:bnc}

Every certificate the inequalities of Section~\ref{sec:cuts} can produce
bounds the objective of the EndoWake layout \emph{model}, which the piecewise-linear
representation inflates above the exact energy. This section carries the
certification onto the exact energy scale. Installation costs stay in the objectives
below and enter nothing else: the oracle and the cuts see energy alone.
The instrument is decomposition.

\subsection[Two-stage reading and the oracle]{Two-stage reading and the oracle}
\label{sec:twostage}

The multi-scenario model is block-angular. The binary layout variables $x$
are the first stage, common to all scenarios; the flow and power variables
$(w^s,z^s)$ form $|\mathcal{S}|$ independent second-stage blocks, coupled
to the first stage only through the right-hand side of the wake constraints
\eqref{eq:mip-wake} and of the rows $z^s_c \le P_{\max}\, x_c$; and the objective is
$\max_x \bigl( \sum_{s} p_s\,Q_s(x) - \sum_{c \in \mathcal{A}} C_c\,x_c
\bigr)$, with
\begin{equation}
Q_s(x) \;=\; \max \Bigl\{ \textstyle\sum_{c} z^s_c \;:\;
(w^s,z^s) \text{ satisfies \eqref{eq:mip-flow}--\eqref{eq:mip-power} for
the given } x \Bigr\}.
\label{eq:subproblem}
\end{equation}
We use the two-stage language for its algorithmic content, not to claim a
decision-theoretic one: the $w^s$ are variables whose maximal value $w^{*}(x)$, the only one the objective
ever uses, is determined by $x$, as Proposition~\ref{prop:fixpoint} makes precise, and the scenarios
are a known distribution rather than an uncertainty revealed over time.

\begin{proposition}[Subproblem in linear time]
\label{prop:oracle}
For fixed binary $x$ and any scenario $s$, the value $Q_s(x)$ is computable
exactly in $O(|\mathcal{C}|)$ time, under either representation of the
power curve: one forward sweep of Proposition~\ref{prop:fixpoint}(b)
followed by the evaluation of the power function at the installed cells.
\end{proposition}

\begin{proof}
By Proposition~\ref{prop:fixpoint}(b,c), the maximum of
\eqref{eq:subproblem} is attained at the maximal field $w^{*}(x)$, which
the sweep computes; the power block then decouples cell by cell.
\end{proof}

Which decomposition fits depends on the power representation.  With the curve as it is---interpolated under SOS2
conditions in the model, exact in the oracle---the second stage is
non-concave (the exact curve is also discontinuous at cut-in), supports no LP dual
multipliers, and one is in the territory of integer recourse: the integer
L-shaped method of \citet{laporte1993}, or a
logic-based scheme in which the subproblem is an oracle rather than a
source of dual information~\citep{hooker2003}. Our setting is an extreme case of the
latter, and an unusually favorable one: the subproblem is not merely easy
to solve, it is a function evaluation, exact and cheaper than the linear
program it replaces. This is what makes
branch-and-check~\citep{thorsteinsson2001} (solve the master once, while calling
the oracle at every incumbent) the natural algorithm here.

\subsection{The master}
\label{sec:master}

Let $P_s(\cdot)$ denote the \emph{true} power curve of scenario $s$ (so
that $P_s(\wmax_s)$ is the free-stream power of one machine), and let
$p^{\mathrm{ex}}_c(X) = P_s\bigl(w^{*}_c(X)\bigr)$ be the exact power of
cell $c$ under the maximal field of the layout with support
$X \subseteq \mathcal{A}$. The master keeps the site variables, the
packing constraints, the band inequalities, and one hypograph variable per cell and scenario, and
drops everything else:
\begin{align}
\max\ & \sum_{s \in \mathcal{S}} p_s \sum_{c \in \mathcal{A}}
\epivar^s_c \;-\; \sum_{c \in \mathcal{A}} C_c\, x_c \notag\\
\text{s.t.}\quad
& \text{the packing and cardinality constraints \eqref{eq:mip-dist}},
\notag\\
& \text{the band inequalities \eqref{eq:cut-t2}, with $\Psi_R$ built on
the true curve $P_s$ and $\epivar^s_c$ in place of $z^s_c$}, \notag\\
& \epivar^s_c \;\le\; P_s(\wmax_s)\, x_c,
\qquad c \in \mathcal{A},\ s \in \mathcal{S},
\label{eq:master-cap}\\
& \epivar^s_c \;\in\; [0,\, P_s(\wmax_s)],
\qquad c \in \mathcal{A},\ s \in \mathcal{S}. \notag
\end{align}
The wind flow and power blocks are gone, and with them
\eqref{eq:mip-flow}--\eqref{eq:mip-power} and the big-$M$ of
\eqref{eq:mip-wake}: no constraint of the master carries $\wmax_s$ as a
coefficient. Building $\Psi_R$ on the true curve is legitimate here because the model's power representation \emph{is}
the oracle. The band inequalities give the master a root bound on the
exact energy scale; the cuts below drive that bound down to the exact optimum.

\subsection[The oracle and incumbent posting]{The oracle and incumbent posting}
\label{sec:posting}

At every master solution integer with respect to $x$, with support
$X = \{c \in \mathcal{A} : x_c = 1\}$, the oracle scores the layout
exactly: one sweep per
scenario returns $w^{*}(X)$, hence $p^{\mathrm{ex}}_c(X)$ for every
$c \in X$ and the exact energy $\sum_s p_s \sum_{c\in X}
p^{\mathrm{ex}}_c(X)$. Two things are done with that number. First, any
cell whose hypograph exceeds its exact power is cut off by the upwind cut
below. Second, the layout is \emph{posted} to the solver as a candidate incumbent
at its exact value (rather than letting the master's optimistic hypographs
stand for it) so that pruning acts on the exact energy scale, up to
the flagging tolerance, from both sides.

The posted point, $x = \mathbf{1}_X$ with $\epivar^s_c =
p^{\mathrm{ex}}_c(X)$ on $X$, satisfies every constraint the master
holds (Proposition~\ref{prop:bi} with $\widetilde{P} = P$,
Proposition~\ref{prop:cone}), so the solver adopts it whenever it
improves on its incumbent; a rejected post costs nothing, the best exact
energy being tracked outside the solver.

\subsection{The upwind cut}
\label{sec:cone}

Fix a scenario $s$, a cell $c$ of the incumbent support $X$, and any set
$U \subseteq \mathcal{A}\setminus\{c\}$ of allocatable cells---in
practice, cells of $X$ upwind of $c$. Let
$q = P_s\bigl(w^{*}_c(U \cup \{c\})\bigr)$
be the exact power of $c$ when a turbine is installed in each cell of
$U \cup \{c\}$ \emph{and nowhere else}, and impose the \emph{upwind
cut} $\epivar^s_c \le q\, x_c +
\bigl(P_s(\wmax_s) - q\bigr) \sum_{u \in U} (1 - x_u)$.

\begin{proposition}[Validity of the upwind cut]
\label{prop:cone}
Under the standing assumption of Section~\ref{sec:model}, for every
scenario $s$, every cell $c \in \mathcal{A}$, and \emph{every}
set $U \subseteq \mathcal{A}\setminus\{c\}$, the upwind cut is
satisfied by every binary point of the master in which each
$\epivar^s_c$ is bounded by the exact power of $c$ under the layout, which
is what a lazy constraint must respect.
\end{proposition}

\begin{proof}
Let $x$ be binary with support $X$, and note first that
$0 \le q \le P_s(\wmax_s)$: the maximal field never exceeds the free
stream, and $P_s$ is non-decreasing on the modeled range by the standing
assumption of Section~\ref{sec:model}. Three cases. If
$x_c = 0$, then $\epivar^s_c = 0$ because of \eqref{eq:master-cap}, and
the right-hand side of the cut is non-negative. If $x_c = 1$ and
$x_u = 0$ for some $u \in U$, the right-hand side is at least
$q + (P_s(\wmax_s) - q) = P_s(\wmax_s)$, which \eqref{eq:master-cap}
already grants. If $x_c = 1$ and $x_u = 1$ for all $u \in U$, then
$X \supseteq U \cup \{c\}$, and antitonicity
(Proposition~\ref{prop:fixpoint}(d)) gives
$w^{*}_c(X) \le w^{*}_c(U \cup \{c\})$, hence
$\epivar^s_c \le p^{\mathrm{ex}}_c(X) \le q$, which is the right-hand
side. \end{proof}

The upwind cuts are logic-based \emph{optimality} cuts of the strengthened
no-good type. The no-good
cuts of \citet{zhang2014} carry the whole incumbent
in their support and a big-$M$ constant; ours restrict the support to the
upwind cells and replace the big-$M$ by the oracle value at
$U \cup \{c\}$, which antitonicity makes exact. The simulation-based
Benders cuts of \citet{forbes2024} and the disaggregated integer L-shaped
cuts of \citet{parada2024} strengthen optimality cuts under a
\emph{monotone} recourse; our cuts exploit the antitonicity of the potential power at a fixed cell, the
disaggregation is per cell and scenario rather than per scenario, and the
strengthened constant is exact by
Proposition~\ref{prop:fixpoint}(d).

Three points deserve comment. First, the constant $q$ must be
evaluated at $U \cup \{c\}$ and not at the whole incumbent $X$.
Second, separation is by comparison: a cut is called for at the pairs $(s,c)$
whose hypograph value exceeds $p^{\mathrm{ex}}_c(X)$, and the cut
separates whenever that value exceeds $q$. In practice $U$ holds
the cells of $X$ whose single-source deficit at $c$ is positive; a lemma
in the supplementary material shows that $q$ then exceeds
$p^{\mathrm{ex}}_c(X)$ only by the deficits below the code's threshold,
at most $1.1\cdot 10^{-9}$~kW here, so every flagged pair is
separated; and should it not, the callback falls back to $U = X
\setminus \{c\}$, whose constant is $p^{\mathrm{ex}}_c(X)$ by
construction---a case that arose 0\ times in
the 45\ runs of Section~\ref{sec:exp-bnc}
that count it, and never in 4\,414\,254\ earlier
cuts. Third,
restricting the support to the
cells upwind of $c$ is what makes the cut survive: with the whole incumbent in its support, as in
the no-good cuts of \citet{zhang2014}, a cut caps the power of $c$ only
on the layouts that contain all of $X$, whereas
the upwind cut caps it on every layout containing $U \cup \{c\}$.
The restriction is not nominal: over the final layouts of the runs
of Table~\ref{tab:bnc} the support $U$ averages 2.8\ cells
against 12.9\ in $X \setminus \{c\}$.
For the same reason the upwind cut binds more widely than the integer
L-shaped optimality cut of \citet{laporte1993}: that cut
is the full-support upwind cut weakened further
by a term in the cells outside $X$, so it binds at $X$ alone.
When $q = p^{\mathrm{ex}}_c(X)$, as for $U = X \setminus \{c\}$, the
upwind cut is also no weaker on the whole master domain; with a
thresholded support it remains valid, and no dominance is claimed. With $q = P_s(\wmax_s)$ the cut is \eqref{eq:master-cap} itself;
a weaker valid constant only weakens it.

The scheme as a whole---the callback in pseudocode, and a
proof that it adds finitely many cuts and terminates at the exact
optimum (in exact arithmetic)---is given in the supplementary material.

\section{Computational experiments}
\label{sec:exp}

In this section we report our computational experience with the exact
methods of this paper: the size and the difficulty of the monolithic
model, the effect of the band inequalities at the root and inside the
search, and the branch-and-check scheme. As the paper focuses on exact methods, heuristics at industrial
size and the external validation of the model against engineering wake
simulators, beyond the common judge of Section~\ref{sec:exp-mi},
against field data and across atmospheric conditions, are outside its scope.

\paragraph{Setup.} All experiments consider a turbine of the 3\,MW class
with
rotor diameter $D=112$\,m, cut-in $3$\,m/s and cut-out $25$\,m/s. To avoid
proprietary power-curve data we adopt the idealized curve
$P(v) = \min\{\tfrac12 \rho_{air} A C_p v^3, P_{\max}\}$ (the cubic
term converted to kW) between cut-in
and cut-out, $P(v) = 0$ outside, with $C_p = 0.45$,
which rates at $10.34$\,m/s, with
$\rho_{air} = 1.225$\,kg/m$^3$, $A = \pi D^2/4$ and $P_{\max} = 3000$\,kW. The
grid spacing is
$\Delta = 1\,D$, the minimum separation $3\,D$ (three cells), the undisturbed speed
$10$\,m/s, and installation costs are zero in the main campaign ($C_c = 0$). The
propagation parameters are those of the calibration of
Section~\ref{sec:model}:
$\alpha = 0.05324$,
$\beta = 0$ and
$\gamma = 0.11859$, and the lumped source strength
is $2a = 0.707$, hence
$\rho = 0.293$. Scenario sets are uniform: the $|\mathcal{S}|$ sector
directions are equispaced starting from the grid's north, every sector
carries probability $1/|\mathcal{S}|$, and the free-stream speed is the
same in all of them, so no sector exceeds the cut-out speed and the
standing assumption of Section~\ref{sec:model} holds. Every campaign
in the body uses the same four sectors, the axial rose at $0$, $90$, $180$ and $270$
degrees. Four axial sectors
keep every direction aligned with a grid axis, where the stencil is
calibrated rather than interpolated; with eight, four of them would be
oblique at $45$ degrees and the flow block would carry a discretization
error of its own, unrelated to the optimization we are measuring.
The main campaign runs on one ladder of instances, a square
$n \times n$ grid of side $n\,D$ with $N_{\max} = n + 2$ turbines allowed, so that the
density $N_{\max}/n^2$ falls from 22\%
to 4.5\% along the ladder and the minimum separation alone never
fixes the layout; grid and $N_{\max}$ are stated with each table.
We work on this synthetic ladder rather than on a published
benchmark for a reason of method. The discrete benchmarks of the
mathematical-programming line---the $10\times10$ grid of
\citet{mosetti1994} and \citet{grady2005}, with cells of five diameters,
on which \citet{turner2014} and \citet{zhang2014} report their
models---and the continuous case studies of the engineering line---the
IEA Task~37 cases of \citet{baker2019}---are all posed on a pairwise wake model
under a fixed superposition rule. The
synthetic ladder, instead, holds the grid resolution and the minimum separation
fixed and grows the candidate set and the turbine cap together, so that
it measures how the methods scale.
Section~\ref{sec:exp-mi} also runs the comparison one would
expect: an interference-matrix model on the same ladder, under the same
budget, with the returned layouts of both formulations scored by a common
judge.
All
runs of the tables of the paper were performed on a single cluster node,
with an Intel Xeon E3-1220~V2 processor at $3.10$\,GHz---four physical
cores, hyper-threading disabled so that timings are reproducible---and
$14$\,GB of RAM; the models are solved by
Gurobi~13.0~\citep{gurobi}. The power curve enters the monolithic model as
the interpolant $\widehat{P}$ under the SOS2 conditions
of~\eqref{eq:mip-power}, on 6\ segments
up to $v_r$ and a flat one beyond; only the oracle and
the evaluator of the returned layouts use the exact curve.
Every solve is given four threads and a budget of $3600$
seconds;
the monolithic solves are run to a relative optimality tolerance of $10^{-6}$ with Gurobi's highest
\texttt{NumericFocus}; the
branch-and-check master runs at default
tolerances (relative gap $10^{-4}$). Every numerical result reported
here was produced by executing the code of the replication
package~\citep{repo}, which also documents the software environment.

\paragraph{Scope of the campaign.}

The formulation is stated on a grid of candidate cells, and three of
its ingredients are more general than our main testbed requires. The
candidate set may be any subset of the grid: an irregular boundary or an
exclusion zone is a set of cells whose site variable is
fixed to zero, while the flow still propagates through them. Installation
costs $C_c$ enter the objective cell by cell and may differ across the
site. Any wind direction enters as one more scenario with its own
propagation kernel. In our main testbed we deliberately use none of these generalizations:
every instance is a square grid, an axial rose with equiprobable sectors
of equal speed, and zero installation cost. The choice is conservative
on the side that matters here, the dual bound:
the condition of Proposition~\ref{prop:vacuous} depends neither on the
sector directions nor on the kernel coefficients (the sector speeds enter only through the power-curve term), so the bound it describes is
structural, not an artifact of symmetry, and the axial rose is the case
the on-axis calibration covers exactly.
Oblique sectors are a limit of the calibration, not of our
method: the propagation of any direction advances along its dominant
grid axis, so the maximal field (Proposition~\ref{prop:fixpoint}), the
oracle, the band inequalities and the upwind cuts hold for any sector, but the on-axis stencil
of Section~\ref{sec:model} is not calibrated for them (with the axial
kernel off axis, both methods run unchanged).
Unequal sector probabilities are weights on the scenarios, and several
speeds per direction are further scenarios, with the kernel calibrated at
each speed. A heterogeneous inflow, as over complex terrain, would make
the free stream cell-dependent, and a non-uniform grid has no stencil in
the sense of Section~\ref{sec:model}: both are outside the scope of this paper.
Section~\ref{sec:exp-irr} treats an irregular site with installation costs.

\paragraph{Performance variability.} Every run of our computational
campaign reported in the tables below and run to a time limit was repeated five times with different Gurobi random seeds.
Every table of time-limited runs below reports the \emph{median run}
of the five: for each instance, and for each configuration where a table
compares two, the run whose certificate on the exact energy scale is the
median of its group, with the solve time breaking ties. Each entry is
therefore one run, reproducible by rerunning that seed, rather than a
column-wise average; only Tables~\ref{tab:ejor-irr}
and~\ref{tab:ejor-mi} report column-wise medians. The tables for all five seeds are in the
supplementary material. For the
branch-and-check the instances that close do so with every seed, and
on the open ones the spread between the best and the worst seed is
between 0.4\ and
1.3\ percentage points. For the monolithic model it is not: across
seeds the solver gap of one instance spans as much as
131.41\ percentage points, and in the
band-placement runs below the spread of the certificate reaches
17\ points on
$20\times20$\unskip, so on the largest grids whether
the search reaches a good incumbent inside the hour depends on the seed.

\paragraph{Reading the tables.} Column \emph{UB} reports the dual bound of
the model, whose power representation $\widehat{P}$ is an upper
interpolation of the exact curve (it majorizes it); UB is therefore an
upper bound on the exact optimum as well, as we are maximizing.
Column $E$ gives the value of the
returned layout under
the exact (that is, non-interpolated) power curve.  Both columns are expected power in kW: over a year of $8760$ hours, the annual
energy production (AEP) in kWh is $8760$ times the value reported. Column \emph{gap} is the solver's own gap at
termination, so a value above the tolerance means that the time limit was
reached; $n_t$ is the number of turbines the returned layout installs,
which need not saturate $N_{\max}$. Finally, column \emph{cert} is the
conservative certificate $(UB-E)/E$, valid because $UB$ is proved on a model
that overestimates the exact power. Unlike the solver gap, which uses the
incumbent objective recorded by the solver, it is computed on the exact
power of the returned layout; where the instance closes, it measures, up
to the solver tolerance, the interpolation bias on that layout.
On the 18\ seed runs of the
instances of Table~\ref{tab:scaling} that close, this interpolation bias lies between
1.70\%\ and
2.18\%\ of the exact energy. It is the
price of representing the power curve by
6\ segments: the finer the interpolant, the
smaller the bias, at the cost of more SOS2 variables per cell and scenario (runs with the band
inequalities as static rows). With 12 segments the bias on the runs up
to $8\times8$ that close (all but one) falls to
0.31\%--0.36\%,
but $10\times10$ no longer closes within the hour, and on the larger grids
the bound falls by
1.3\%--1.5\%,
most of the bias. As to the branch-and-check of
Section~\ref{sec:bnc}, its master carries no power curve, hence no bias.

\begin{table}[htbp]
\centering\small
\caption{Scalability of the unstrengthened monolithic model, $3600$~s
time limit.}
\label{tab:scaling}
\setlength{\tabcolsep}{5pt}
\begin{tabular}{lrrrrrrrrrrr}
\toprule
grid & $N_{\max}$ & $n_t$ & bin & cont & rows & $t$\,[s] & gap & nodes & UB & $E$\,[kW] & cert \\
\midrule
6$\times$6 & 8 & 4 & 36 & 1440 & 876 & 9.5 & 0.00\% & 3\,591 & 10777 & 10598 & 1.70\% \\
7$\times$7 & 9 & 6 & 49 & 1960 & 1\,197 & 19.6 & 0.00\% & 8\,900 & 14629 & 14347 & 1.97\% \\
8$\times$8 & 10 & 8 & 64 & 2560 & 1\,568 & 173.5 & 0.00\% & 60\,741 & 18391 & 18022 & 2.05\% \\
10$\times$10 & 12 & 11 & 100 & 4000 & 2\,460 & 3471.5 & 0.00\% & 453\,348 & 26248 & 25689 & 2.18\% \\
12$\times$12 & 14 & 14 & 144 & 5760 & 3\,552 & 3600.0 & 16.27\% & 125\,227 & 38738 & 32608 & 18.80\% \\
14$\times$14 & 16 & 16 & 196 & 7840 & 4\,844 & 3600.0 & 17.67\% & 90\,262 & 44617 & 37121 & 20.20\% \\
16$\times$16 & 18 & 18 & 256 & 10240 & 6\,336 & 3600.0 & 18.03\% & 63\,048 & 50212 & 41650 & 20.56\% \\
20$\times$20 & 22 & 22 & 400 & 16000 & 9\,920 & 3600.0 & 24.14\% & 23\,575 & 61215 & 48309 & 26.72\% \\
24$\times$24 & 26 & 19 & 576 & 23040 & 14\,304 & 3600.0 & 77.75\% & 8\,489 & 72373 & 43038 & 68.16\% \\
\bottomrule
\end{tabular}
\end{table}

\paragraph{Performance of the unstrengthened monolithic model.} Table~\ref{tab:scaling} frames
what the EndoWake formulation costs to solve as written. Model size grows as predicted (linear in the cells, built in
a fraction of a second) but solving is another matter: the instances up
to $10\times10$\ close to optimality, and from
$12\times12$\ on the solver reaches the time
limit, with the gap
growing to 77.75\%\ at
24$\times$24\unskip. The cause is not the size of the
model but the weakness of its relaxation, which the rest of this section
isolates.
The measured root bound confirms Proposition~\ref{prop:vacuous} to the
digit on every instance that meets its
condition, from $12\times12$\ on.

\subsection[Band inequalities at the root and in the search]{Band inequalities at the root and in the search}
\label{sec:exp-cuts}

We generate and store the full family of band inequalities in a
preprocessing step, before the search starts; its computational
cost is negligible.

How the family enters the model matters.
Adding every BI as a static constraint in the MILP model is the direct
reading of Section~\ref{sec:cuts}, and it tightens the root as intended (Table~\ref{tab:bounds}), with an effect that lasts: on
$16\times16$, $20\times20$, $24\times24$\ the median dual bound after the hour lies below that of the plain model by
2.7\%\ to
4.6\%\unskip. The static constraints are not
free, however.
On $16\times16$\ the tightened model explores
9\,793\ nodes against the
63\,048\ of the plain one, and good incumbents
stop coming.

We therefore compared three\ ways of keeping the
inequalities without the overhead of static rows, all of them through the
solver's user-cut pool, where a constraint is held outside the model and
inserted only when the relaxation violates it: (a) all inequalities pooled;
(b) only those of the offset bands, with the partitions left as static
constraints; and (c) only those with zero dual multiplier in the root node relaxation, keeping the active ones static. We adopt option (c), the
configuration of the monolithic model in every experiment reported below,
except the additional runs, which keep
them as static rows (on the grids where both run, a bound never looser
than the pooled one).

In the branch-and-check master of Section~\ref{sec:bnc} we keep the
band inequalities static: its incumbent comes from the oracle, and
pooling there only weakens the certificate. In the monolithic model the pool gives the
best median certificate of the three configurations---pooled, static, and no
inequalities at all---on
$16\times16$ and $24\times24$; on $12\times12$\unskip,
where the tree still moves the bound, the static form wins.

\begin{table}[htbp]
\centering\small
\caption{Root relaxation of the MILP, with clique distance constraints and
with and without the valid
inequalities of Section~\ref{sec:cuts}. Column \emph{base} is the
untightened formulation; \emph{BI} adds the band inequalities
\eqref{eq:cut-t2} at every width. Both gaps are $100(\text{bound}/\text{ref.\ LB}-1)$,
where \emph{ref.\ LB}, in the second-to-last column, is a lower bound on the optimum of the
model: the largest of the integer values reached in
Tables~\ref{tab:scaling} and~\ref{tab:bb} and of the exact energies of
the branch-and-check incumbents of Table~\ref{tab:bnc}, over every seed (as
$\widehat{P} \ge P$, the model value of such a layout is at least its
exact energy, so the gap is, if anything, overestimated); a dagger
marks an instance that the reference solve of Table~\ref{tab:scaling} did
not close within its hour, so the gap there is measured against an
incumbent. The last column is the dual bound Gurobi itself reaches on
\emph{base} at the end of that solve.}
\label{tab:bounds}
\begin{tabular}{lrrrrrrr}
\toprule
grid & $N_{\max}$ & base & gap & +BI & gap & ref.\ LB & B\&B bound \\
\midrule
$6\times6$ & 8 & 12000 & 11.3\% & 10820 & 0.4\% & 10777 & 10777 \\
$7\times7$ & 9 & 23983 & 63.9\% & 16284 & 11.3\% & 14629 & 14629 \\
$8\times8$ & 10 & 27000 & 46.8\% & 20593 & 12.0\% & 18391 & 18391 \\
$10\times10$ & 12 & 36000 & 37.2\% & 30097 & 14.7\% & 26248 & 26248 \\
$12\times12$ & 14 & 42000 & 25.1\% & 36832 & 9.7\% & 33564$^{\dagger}$ & 38738 \\
$14\times14$ & 16 & 48000 & 24.8\% & 42274 & 9.9\% & 38456$^{\dagger}$ & 44617 \\
$16\times16$ & 18 & 54000 & 23.9\% & 47838 & 9.8\% & 43584$^{\dagger}$ & 50212 \\
$20\times20$ & 22 & 66000 & 24.4\% & 59100 & 11.4\% & 53036$^{\dagger}$ & 61215 \\
$24\times24$ & 26 & 78000 & 21.8\% & 70393 & 10.0\% & 64017$^{\dagger}$ & 72373 \\
\bottomrule
\end{tabular}
\end{table}

Table~\ref{tab:bounds} reports the root relaxation on
9\ instances, before and after the band inequalities \eqref{eq:cut-t2} of Section~\ref{sec:cuts}, and shows their effectiveness.
On the 5\ instances that the reference
solve of Table~\ref{tab:scaling} leaves open, the root gap falls from
21.8\%--25.1\%\
to 9.7\%--11.4\%\unskip,
and the effect does not fade with size: on
$24\times24$\ it goes from 21.8\%\ to
10.0\%\unskip. On the
4\ instances that close,
the gap the inequalities leave can be larger, up to
14.7\%\unskip: those grids are denser than
the threshold of Proposition~\ref{prop:vacuous}, and a band of width
below the separation says nothing about two cells that conflict
\emph{across} the flow---the band inequalities repair the wakes, not the
packing.

The most telling comparison in the table, however, is with the solver's own
effort. Consider the 5 instances that the reference solve does not close
within its budget. On \emph{every one} of them, the bound Gurobi reaches after
a full hour of branch and bound on the untightened model is weaker than the
root bound with the band inequalities in place. The hour bound exceeds the root bound by
at least
2.8\% of the latter, at a generation
cost of 4.1~s against those $3600$. The
formulation hands the solver a
relaxation with nothing in it, and tens of thousands of nodes are spent
recovering by enumeration what one family of inequalities supplies at the
root---up to 125\,227\ nodes on the instances of
Table~\ref{tab:bb} that do not close.

\begin{table}[htbp]
\centering\small
\caption{The band inequalities inside the search, on the instances of
Table~\ref{tab:scaling}: dual bound, certificate on the exact energy scale
and nodes, without (\emph{none}) and with (\emph{+BI}) them, the latter
pooled as described in Section~\ref{sec:exp-cuts}. Where an instance closes, the
certificate column gives instead, in italics, the time to prove it.
The certificate is $(UB - E)/E$, $UB$ the bound shown and
$E$ the exact expected power of the run's own incumbent.}
\label{tab:bb}
\begin{tabular}{lrrrrrrr}
\toprule
& & \multicolumn{3}{c}{none} & \multicolumn{3}{c}{$+$BI} \\
\cmidrule(lr){3-5}\cmidrule(lr){6-8}
grid & $N_{\max}$ & bound & cert\,/\,$t$ & nodes & bound & cert\,/\,$t$ & nodes \\
\midrule
$6\times6$ & 8 & 10777 & \emph{10\,s} & 3\,591 & 10777 & \emph{28\,s} & 8\,643 \\
$7\times7$ & 9 & 14629 & \emph{20\,s} & 8\,900 & 14629 & \emph{157\,s} & 57\,200 \\
$8\times8$ & 10 & 18391 & \emph{174\,s} & 60\,741 & 18391 & \emph{197\,s} & 51\,986 \\
$10\times10$ & 12 & 26248 & \emph{3472\,s} & 453\,348 & 26248 & \emph{1185\,s} & 152\,822 \\
$12\times12$ & 14 & 38738 & 18.80\% & 125\,227 & 36747 & 13.11\% & 60\,711 \\
$14\times14$ & 16 & 44617 & 20.20\% & 90\,262 & 42274 & 13.88\% & 41\,457 \\
$16\times16$ & 18 & 50212 & 20.56\% & 63\,048 & 47838 & 19.95\% & 29\,086 \\
$20\times20$ & 22 & 61215 & 26.72\% & 23\,575 & 59100 & 25.09\% & 11\,450 \\
$24\times24$ & 26 & 72373 & 68.16\% & 8\,489 & 70393 & 54.02\% & 6\,257 \\
\bottomrule
\end{tabular}
\end{table}

A tighter root bound is worth reporting only if it survives contact with
the search, and Table~\ref{tab:bb} answers that question directly. The
table takes the instances of Table~\ref{tab:scaling} and compares each
under a common $3600$~s budget, with and without the inequalities, charging
the generation time to that budget.
On the instances that close either way, up to
$10\times10$\unskip, the band inequalities take the
largest solve from 453\,348\ nodes to
152\,822\unskip, and bring the wall-clock time down by
a factor of 2.9$\times$\unskip, separation included.
On all 5\ instances
that do not close, the median certificate on the exact energy scale falls by
 3\%\  to
31\%\unskip\ of its value; the bound,
whose comparison does not depend on the incumbent, is lower with the
inequalities on every instance and, seed by seed, on all
25\ open instance--seed pairs, the
certificate on 18\ of them. Note that no instance above $10\times10$\ closes
within the budget, even with the band inequalities in place.

\subsection{Branch-and-check}
\label{sec:exp-bnc}

We report here 45\ runs of the
scheme of Section~\ref{sec:bnc}, because it answers the one question the
rest of the paper cannot: what a certificate on the exact power scale
(without interpolation) costs. The runs use the upwind cut with the
constant of Section~\ref{sec:cone} and the exact incumbent posting of
Section~\ref{sec:posting}.

\begin{table}[htbp]
\centering\small
\caption{Branch-and-check on the instances of Table~\ref{tab:scaling}.
\emph{root} is the master relaxation before any upwind cut, \emph{bound}
the dual bound at the end, \emph{best} the energy of the incumbent under
the exact power curve, and \emph{gap} the difference between the two:
every gap here is on the non-interpolated EndoWake objective. Cut
counts are as the solver reports them, with no deduplication, and the
pool-cut mechanism for the BIs is deactivated.}
\label{tab:bnc}
\begin{tabular}{lrrrrrrrrr}
\toprule
grid & $N_{\max}$ & root & bound & best & $n_t$ & gap & nodes &
cuts & $t$ [s] \\
\midrule
$6\times6$ & 8 & 10643 & 10598 & 10598 & 4 & 0.00\% & 121 & 94 & 0.2 \\
$7\times7$ & 9 & 16022 & 14347 & 14347 & 6 & 0.00\% & 472 & 473 & 0.8 \\
$8\times8$ & 10 & 20273 & 18022 & 18022 & 8 & 0.00\% & 2\,717 & 1\,385 & 2.6 \\
$10\times10$ & 12 & 29604 & 25689 & 25689 & 11 & 0.00\% & 46\,111 & 6\,805 & 77 \\
$12\times12$ & 14 & 36262 & 35411 & 32793 & 14 & 7.98\% & 71\,504 & 43\,860 & 3600 \\
$14\times14$ & 16 & 41474 & 41083 & 37867 & 16 & 8.49\% & 46\,007 & 54\,383 & 3600 \\
$16\times16$ & 18 & 46957 & 46856 & 42698 & 18 & 9.74\% & 45\,879 & 60\,032 & 3600 \\
$20\times20$ & 22 & 58075 & 58075 & 52873 & 22 & 9.84\% & 47\,472 & 78\,538 & 3600 \\
$24\times24$ & 26 & 69234 & 69234 & 63429 & 26 & 9.15\% & 65\,156 & 90\,091 & 3601 \\
\bottomrule
\end{tabular}
\end{table}

Table~\ref{tab:bnc} collects the campaign, one row per instance of the
ladder of Table~\ref{tab:scaling}, on the same machine and under the same
budget as the monolithic runs it is compared against. The scheme closes every instance up to $10\times10$\  to within the tolerances of the
setup above, the largest of them in 77~s
against the 3472~s the monolithic model needs to
close its own model on the same instance
(1185~s with the band inequalities), and where the
monolithic certificate is worth 2.18\%\ on the
exact energy scale the scheme's is zero to the digit. From $12\times12$\ on
no run closes within the hour, so the gaps are what
the budget bought and not what the method can reach: they range from
7.98\%\ to 9.84\%\unskip,
against the
13.11\%--54.02\%\ that the
monolithic model with the band inequalities certifies on the same
instances.
The scheme certifies better on every open instance, its bound living
on the exact energy scale and paying no interpolation bias, and the
primal side goes the same way: on each of them the monolithic incumbent
stays below the scheme's, by as much as
27.95\%\ on
$24\times24$\unskip, when both run without a starting
solution (a common start removes the difference, as shown below). What stalls the scheme on the large grids is the master's
relaxation, not the search: on the largest grids the final bound is the
root bound, and the scheme spends its hour proving, not improving.
On $20\times20$ and $24\times24$ its bound lies
1.7\%--1.8\%
below the monolithic one.
The strength of that relaxation comes from the band inequalities, and
the support of the cut decides where the search can finish: without the
band inequalities in the master, or with the full-support variant of the
cut, $U = X \setminus \{c\}$ as in the no-good cuts
of \citet{zhang2014}, the same
runs close nothing above $8\times8$\unskip.

\paragraph{Primal and dual sides.} The two sides of the certificate
behave differently. The dual side moves at the root and little
afterwards: on the grids left open, an hour of branching moves the bound
by at most 0.2\%\ in the monolithic
model with the band inequalities and
2.3\%\ in the branch-and-check master,
so the root strength of Table~\ref{tab:bounds} decides the certificates
on the large grids. The primal side, where the methods differ most
without a starting solution, is levelled by a common warm
start (Section ``The common warm start'' of the supplementary material):
both methods receive the same heuristic layout as starting solution and
run for one hour, with the same five Gurobi random seeds as the other
runs. The starting solution is improved, when at all, by at most
0.94\%, and what separates the two
certificates is then the dual bound, lower for branch-and-check.

\subsection{Against an interference-matrix model}\label{sec:exp-mi}

The classical alternative to EndoWake prices the wakes before the
layout is known. A pairwise power loss $L^s_{uv}$---the production that a
turbine at $u$, alone and in free stream, takes from a turbine at $v$ in
sector $s$---is computed offline for every ordered pair of cells and every
sector, and the losses are summed linearly
\citep{turner2014,zhang2014,fischetti2016proximity,fischetti2019overview,cazzaro2023}. The
model keeps the site variables and one nonnegative loss variable per cell and sector,
$z_{sv} \ge \sum_{u \ne v} L^s_{uv}\, x_u - M_{sv}(1 - x_v)$ with
$M_{sv} = \sum_u L^s_{uv}$, and maximizes
$\sum_s p_s \sum_v \bigl(P(V_s)\, x_v - z_{sv}\bigr)$ under the same
separation and cardinality rows. It is the standard
linearization of the quadratic objective, with no flow block and no power
curve inside the model. We instantiate it on the ladder of this section
with the pairwise deficits of the Bastankhah--Port\'e-Agel wake our
kernel is calibrated on, the same solver (at default tolerances for this
model), machine, budget, threads and five seeds, and we score the
returned layouts---ours included---with a common judge: the
implementation of the same wake in PyWake~\citep{pywake}, with its
default linear superposition, on the same rose and at the same speed.
As the matrix is built from that deficit, we also check the outcome on
our exact scale and under another deficit model.

Table~\ref{tab:ejor-mi} reports the three costs the construction
incurs and the quality it returns. Preprocessing is negligible at these
sizes, and the model, quadratic in the cells (Section~\ref{sec:model}),
outgrows the monolithic one in nonzeros, counted before presolve, from
$14\times14$\ on. The solver closes
5\ of the
9\ grids within the hour and stops at the limit
from $14\times14$\ on with a gap of
11\%\ to
13\%\ on its own model across seeds. On the grids
that both formulations close, the layouts they return have the same energy
under the judge; on $12\times12$,
which only the interference-matrix model closes, its layout scores
0.16\%\ higher. On the grids that the
interference-matrix model leaves open too, from
$14\times14$\unskip\ on, the judge ranks the
interference-matrix layouts above the branch-and-check incumbents, which
score between $-$2.2\%\ and
$-$1.1\%\ relative to them
(Figure~\ref{fig:layout-ejor} shows the two layouts on the largest grid).
The order is the same on our exact scale, evaluated with the propagator
of Section~\ref{sec:model}, where the difference in favor of the
interference-matrix layouts reaches $+$1.8\%\unskip, and with the deficits of the same wake superposed in
root-sum-square rather than linearly, where the incumbents score between
$-$1.3\%\ and
$-$0.7\%\ relative to them. Under the Gaussian
form~\citep{pedersen2022} of the TurbOPark wake~\citep{nygaard2020}, the
interference-matrix layout stays ahead on
$12\times12$\ (by
2.1\%\unskip), the two are
within one percent on $14\times14$ and $16\times16$, and the
branch-and-check incumbents lead on the two largest grids, by up to
$+$2.9\%\unskip. On the grids that
both formulations close, all three PyWake judges (linear, root-sum-square,
TurbOPark) give their layouts the same score; above them, the ranking
depends on the wake model, and among the runs without a starting solution
the interference-matrix model is the better primal heuristic within the
hour, under its own wake and on our exact scale.
Its bound, however, concerns its pairwise objective and certifies
nothing about the exact energy of Section~\ref{sec:model}. The two constructions are complements rather than
rivals: evaluated by the oracle, the interference-matrix layouts of the
grids that model leaves open, from $14\times14$\
on, are certified by the bound of branch-and-check with gaps from
7.3\%\ to 8.2\%\unskip,
against
8.5\%--9.8\%
for its own incumbents (medians over the seeds; the supplementary
material gives them grid by grid).

\begin{table}[htbp]
\centering\small
\caption{The interference-matrix model (pairwise power losses from the Bastankhah--Port\'e-Agel wake, summed linearly) on the ladder of Section~\ref{sec:exp}: time to compute the deficit matrices, memory of the deficit and loss matrices, size after presolve, total time (matrices, model build and solve) and gap on its own model, and the energy of the returned layout under the common judge (PyWake), against that of the layouts of the monolithic model with band inequalities and of branch-and-check. Each entry is the median over the five seeds; $n_t$ is the number of turbines of the median-judge run.}
\label{tab:ejor-mi}
\setlength{\tabcolsep}{5pt}
\begin{tabular}{lr rrr rr r rrr}
\toprule
 & & \multicolumn{3}{c}{preprocessing, presolved size} & \multicolumn{2}{c}{total (own model)} & & \multicolumn{3}{c}{judge [kW]} \\
\cmidrule(lr){3-5}\cmidrule(lr){6-7}\cmidrule(lr){9-11}
grid & $N_{\max}$ & $t$\,[s] & MB & rows$\times$cols & $t$\,[s] & gap & $n_t$ & MI & mono+BI & B\&C \\
\midrule
6$\times$6 & 8 & 0.02 & 0.1 & 105$\times$118 & 0.2 & 0.00\% & 4 & 10\,689 & 10\,689 & 10\,689 \\
7$\times$7 & 9 & 0.04 & 0.2 & 176$\times$170 & 0.3 & 0.00\% & 6 & 14\,630 & 14\,630 & 14\,630 \\
8$\times$8 & 10 & 0.07 & 0.3 & 274$\times$248 & 0.9 & 0.00\% & 8 & 18\,487 & 18\,487 & 18\,487 \\
10$\times$10 & 12 & 0.18 & 0.6 & 472$\times$411 & 6.6 & 0.00\% & 11 & 26\,381 & 26\,381 & 26\,381 \\
12$\times$12 & 14 & 0.38 & 1.3 & 722$\times$615 & 2\,490.0 & 0.00\% & 14 & 33\,696 & 33\,421 & 33\,642 \\
14$\times$14 & 16 & 0.72 & 2.5 & 1024$\times$859 & 3\,601.7 & 10.75\% & 16 & 38\,966 & 37\,921 & 38\,531 \\
16$\times$16 & 18 & 1.23 & 4.2 & 1404$\times$1169 & 3\,602.0 & 11.79\% & 18 & 43\,938 & 40\,817 & 43\,280 \\
20$\times$20 & 22 & 3.06 & 10.2 & 2256$\times$1845 & 3\,603.9 & 11.82\% & 22 & 53\,898 & 47\,650 & 52\,730 \\
24$\times$24 & 26 & 6.35 & 21.2 & 3365$\times$2717 & 3\,607.6 & 11.27\% & 26 & 63\,999 & 46\,333 & 62\,672 \\
\bottomrule
\end{tabular}
\end{table}

\begin{figure}[htbp]
\centering
\includegraphics[width=0.7\textwidth]{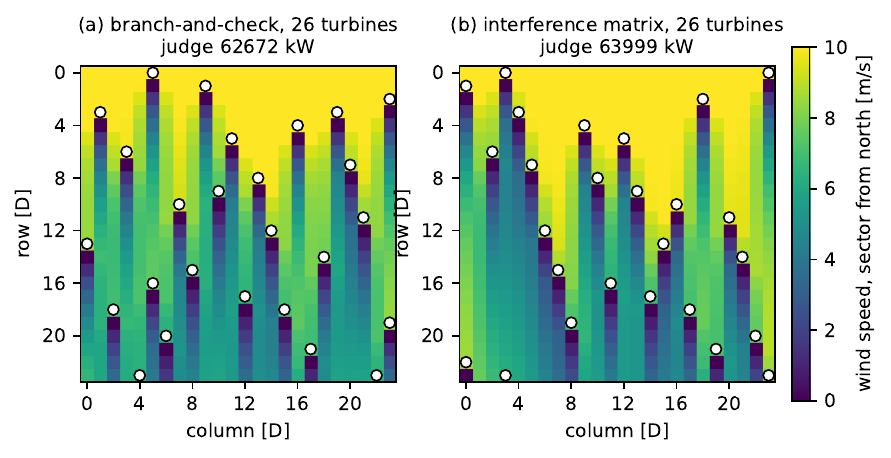}
\caption{Two layouts on the largest grid of the ladder, with the EndoWake
wind field of the sector from the north as the propagator computes it:
(a) branch-and-check, (b) interference matrix, each the median run under
the Bastankhah--Port\'e-Agel judge. Circles are turbines, the color is the wind speed of that
sector under EndoWake; the value above each panel is the expected power
of the layout under the PyWake judge, over the whole rose.}
\label{fig:layout-ejor}
\end{figure}

\subsection{Irregular sites}\label{sec:exp-irr}

Table~\ref{tab:ejor-irr} runs both exact
methods on an irregular site: a wavy boundary, an exclusion zone, and an
installation penalty growing with the distance from a substation, on the
$12\times12$, $16\times16$ and $24\times24$ grids (Section
``An irregular site with installation costs'' of the supplementary
material defines the site). Nothing in the methods
changes, and neither does their ranking: on every grid and seed the bound
of the branch-and-check stays below that of the monolithic model with
static band rows, and so does its gap.

\begin{table}[htbp]
\centering\small
\caption{The irregular site with installation costs: a wavy boundary up to $n/4$ cells deep, an exclusion disk of radius $0.12\,n$, and a distance-based installation penalty, in power-equivalent units, growing linearly with the distance from a substation on the right edge, up to a tenth of the free-stream power of one turbine; in the monolithic model the band inequalities are static rows. $|\mathcal{A}|$ is the number of admissible cells, \emph{root} the master relaxation of the branch-and-check before any upwind cut, \emph{bound} and \emph{best} the dual bound and the exact energy of the incumbent at the time limit, both net of the site cost (medians over five seeds), \emph{gap} their difference relative to \emph{best} (range over the seeds), \emph{cost} the share of the site cost in the gross energy of the incumbent, and the last column the gap of the same method on the square site with zero costs (the static-row runs of the main campaign on the same grids, range over the seeds). $3600$~s; no run closes.}
\label{tab:ejor-irr}
\begin{tabular}{lrlrrrrrr}
\toprule
grid & $|\mathcal{A}|$ & method & root & bound & best & gap & cost & gap, square \\
\midrule
$12\times12$ & 118 & branch-and-check & 33\,967 & 30\,557 & 29\,471 & 3.5--4.0\% & 6.6\% & 7.6--8.1\% \\
 &  & monolithic + BI & -- & 32\,548 & 29\,457 & 10.0--11.2\% & 6.4\% & 11.1--14.8\% \\
\addlinespace
$16\times16$ & 211 & branch-and-check & 44\,588 & 44\,417 & 39\,880 & 10.5--12.4\% & 5.9\% & 9.3--10.1\% \\
 &  & monolithic + BI & -- & 45\,411 & 38\,227 & 16.0--21.6\% & 6.0\% & 24.8--52.7\% \\
\addlinespace
$24\times24$ & 477 & branch-and-check & 65\,900 & 65\,855 & 58\,560 & 12.1--12.5\% & 5.3\% & 8.1--9.4\% \\
 &  & monolithic + BI & -- & 67\,106 & 55\,597 & 17.5--23.0\% & 5.1\% & 34.0--121.5\% \\
\bottomrule
\end{tabular}
\end{table}

\section{Conclusions}
\label{sec:concl}

We have considered the wind farm layout optimization problem and
proposed EndoWake, a wind-flow model in which the wind speed at every
grid cell is a decision variable, linked to its upwind neighbors by
linear propagation constraints, so that the wake equations enter a
mixed-integer linear program directly. On this formulation we have
introduced a family of valid band inequalities and a branch-and-check
scheme, whose master is cut by upwind cuts generated by an exact
linear-time oracle. Computational results on a ladder of square grids
under a four-sector wind rose confirm that the band inequalities
substantially strengthen the root bound, and that branch-and-check
proves optimality on the exact energy up to
$10\times10$\unskip\ in a fraction of the time
of the monolithic model.

The result we regard as the most consequential concerns the linear
relaxation, whose root bound we characterize exactly. Under a condition that the
operating densities of practice satisfy comfortably, the root bound equals
the number of turbines times the rated power of one machine---above even
the wake-free production---whatever the number of wind sectors
(Proposition~\ref{prop:vacuous}),
and the two natural local repairs of Section~\ref{sec:lifted} bring it no
lower than that.

We have closed part of that gap, and the way it closes is itself
informative. What works is aggregation over the cells a wind direction
orders, which is what the band inequalities do: on the instances the reference solve leaves
open, the root bound they give beats the one a state-of-the-art solver
reaches there by an hour of branching.

What ties these results together is the structure of the model. The
multi-scenario formulation is block-angular, and its second stage is
evaluated
\emph{exactly} in $O(|\mathcal{C}|)$ time per scenario
(Proposition~\ref{prop:oracle}), so the subproblem is an oracle rather
than a source of dual information. Branch-and-check is the natural target: the wake
implications are delegated to the oracle. According to
Section~\ref{sec:exp-bnc}, the scheme certifies with no interpolation
bias, and in a fraction of the time. Against a classical
interference-matrix model its advantage is the bound, not the layout: the pairwise
model, run from a cold start, is the better primal heuristic within the hour under its own wake, and branch-and-check
certifies its layouts (Section~\ref{sec:exp-mi}). The sizes closed exactly remain an order of
magnitude below an industrial farm; the bridge is a
local-search heuristic built on the same wake model, in the companion
paper~\citep{companion}.

In our view, future research should address two engineering refinements
of the scheme: management of the lazy cut pool, and a refinement of the
incumbent posting in which a local search improves the layout the oracle
has just scored. A third direction, the dual side beyond the root, is the subject of
follow-up work on a Dantzig--Wolfe reformulation of the
master~\citep{dwcor}. Sectors oblique to the grid  remain open, as their kernel is not calibrated; irregular sites with
costs both exact methods handle unchanged (Section~\ref{sec:exp}).
Finally, the band inequalities are one
instance of a recipe that does not require the zone to be a full-length
strip fixed in advance: any set of cells whose lateral extent stays below
the minimum separation has the chain property the dynamic program needs,
so a separator could carve the zone out of the fractional solution.

\subsection*{Acknowledgments}

We thank the two anonymous reviewers for their comments, which
improved the paper.

\subsection*{Funding}

The work of the first author was supported by the project
PID2024-162616OA-I00 funded by MICIU/AEI/10.13039/501100011033/FEDER, UE,
and by the Ram\'on y Cajal project RYC2023-042623-I funded by
MICIU/AEI/10.13039/501100011033 and FSE+.

\subsection*{CRediT authorship contribution statement}

\textbf{Martina Fischetti}: Conceptualization, Data curation,
Investigation, Methodology, Validation, Writing --- original draft.
\textbf{Matteo Fischetti}: Conceptualization, Formal analysis,
Methodology, Software, Writing --- review and editing.

\subsection*{Declaration of competing interest}

The authors declare that they have no known competing financial interests or
personal relationships that could have appeared to influence the work
reported in this paper.

\subsection*{Data availability}

The replication
package~\citep{repo} contains code, instances, layouts, raw campaign
outputs and the table scripts: every table and every number quoted in
the text can be regenerated from it, and no proprietary data are used.

\subsection*{Declaration of generative AI and AI-assisted technologies in
the manuscript preparation process}

During the preparation of this work the authors used Claude Fable and Claude Opus (Anthropic) as programming assistants for the experimental
code, as adversarial readers of the mathematical statements, and to edit
the manuscript for language and structure; GPT-6 Astra (OpenAI) gave independent adversarial audits. After
using these tools, the authors reviewed and edited the content,
checked every proof independently, and take full responsibility for the
content of the published article.

\setlength{\bibsep}{0pt}

\clearpage
\setcounter{section}{0}\renewcommand{\thesection}{S\arabic{section}}
\setcounter{table}{0}\renewcommand{\thetable}{S\arabic{table}}
\setcounter{figure}{0}\renewcommand{\thefigure}{S\arabic{figure}}
\setcounter{equation}{0}\renewcommand{\theequation}{S\arabic{equation}}
\setcounter{algorithm}{0}\renewcommand{\thealgorithm}{S\arabic{algorithm}}
\setcounter{proposition}{0}\renewcommand{\theproposition}{S\arabic{proposition}}
\setcounter{lemma}{0}\renewcommand{\thelemma}{S\arabic{lemma}}
\setcounter{corollary}{0}\renewcommand{\thecorollary}{S\arabic{corollary}}
\setcounter{remark}{0}\renewcommand{\theremark}{S\arabic{remark}}
\renewcommand{\theHsection}{S\arabic{section}}\renewcommand{\theHtable}{S\arabic{table}}
\renewcommand{\theHfigure}{S\arabic{figure}}\renewcommand{\theHequation}{S\arabic{equation}}
\renewcommand{\theHalgorithm}{S\arabic{algorithm}}\renewcommand{\theHproposition}{S\arabic{proposition}}
\renewcommand{\theHlemma}{S\arabic{lemma}}\renewcommand{\theHcorollary}{S\arabic{corollary}}
\renewcommand{\theHremark}{S\arabic{remark}}
\section*{Supplementary material}
\addcontentsline{toc}{section}{Supplementary material}
\noindent This supplement collects the pseudocode of the band dynamic program
of Section~\ref{sec:cuts} of the paper and of the branch-and-check scheme of its
Section~\ref{sec:bnc}, with the proof that the latter terminates finitely at the
exact optimum (in exact arithmetic; Proposition~\ref{prop:bnc}
states what survives the solver's tolerances); the statement and proof of
the lifted wake constraint of Section~\ref{sec:lifted}; the
floating-point behavior of the band inequalities; the comparison of the
user-cut configurations of Section~\ref{sec:exp-cuts}; the primal and
dual sides of the methods, separated; five further experiments (a common
warm start, larger grids, a finer power curve, an irregular site with
installation costs, eight oblique sectors; in all five the monolithic
model keeps its band inequalities as static rows: where both configurations run, their bound is never looser than the pooled
one, at a cost on the primal side) and the interference-matrix campaign seed by seed;
the calibration of Section~\ref{sec:model} repeated at other wake growth rates; and the tables of the
computational campaign for each of the five solver seeds. It is part of
the replication package.

\section{The band dynamic program}
\label{sm:banddp}

\paragraph{The breakpoints of $\widehat P$.} The interpolant of
Section~\ref{sec:exp} of the paper has $B = 8$ breakpoints:
the origin; the two sides of the jump at the cut-in speed, at the cut-in
less $10^{-3}$~m/s and at the cut-in; three speeds equally spaced between
the cut-in and the rated speed $v_r$; $v_r$ itself; and the cut-out speed.
The 6\ segments of the paper are those up to
$v_r$, the seventh being the flat rated stretch. Between the cut-in and
$v_r$ the exact curve is a cubic, so its chords majorize it, which is
what $\widehat P \ge P$ rests on.

Algorithm~\ref{alg:banddp} spells out the dynamic program of Section~\ref{sec:cuts}
of the paper, which returns the capacity $K_R$ of a band $R$ and the
table $\Psi_R(0), \dots, \Psi_R(K_R)$ from which the band inequalities
\eqref{eq:cut-t2} of the paper are built; the floor deficit is
\eqref{eq:floor} there, and the notation is that of
Section~\ref{sec:cuts}. The program must be run up to $k = N_{\max}$,
not to a smaller horizon: a truncated program returns the horizon itself
as $K_R$, and \eqref{eq:cut-t2} would then forbid a band from holding
as many turbines as it can; running it to $N_{\max}$ costs nothing, since
\eqref{eq:mip-dist} already forbids more.

\begin{algorithm}[htbp]
\caption{The band dynamic program, spelled out for a north wind: the
propagation planes are the rows of the grid, and the flow order is the
row order. In any other direction $\mathrm{row}$ and $\mathrm{col}$ become
the index of the propagation plane and the displacement within it, and
nothing else changes.}
\label{alg:banddp}
\begin{algorithmic}[1]
\Require the allocatable cells of the band $R$, of width $L < \sigma$, so that no
propagation plane of $R$ can carry two turbines; the free-stream speed $V$
and the power representation $\widehat P$, non-decreasing (the true
curve $P$ in the master of Section~\ref{sec:bnc} of the paper); the minimum
separation $\sigma$, in cells; the floor deficit
$\underline{\delta}(g,d)$ of every offset ($g$ planes downwind, $d$ cells
aside), from \eqref{eq:floor} of the paper (finite at every offset the
recursion queries, since the pair $(a,b)$ itself enters the minimum); the global cardinality budget $N_{\max}$, used as the horizon
\Ensure the band capacity $K_R$ and the table $\Psi_R(0),\dots,\Psi_R(K_R)$
\State for $a, b \in R$, write $a \prec b$ when
$g_{ab} = \mathrm{row}(b)-\mathrm{row}(a) \ge 1$
\Statex \hskip\algorithmicindent and
$g_{ab}^2 + d_{ab}^2 \ge \sigma^2$, where
$d_{ab} = \mathrm{col}(b)-\mathrm{col}(a)$
\Comment{$(a,b)$ can be consecutive in a chain}
\State $B_1[b] \gets \widehat P(V)$ for every $b \in R$
\Comment{the first turbine of the chain is charged no wake}
\For{$k = 2, \dots, N_{\max}$}
\For{$b \in R$}
\State $B_k[b] \gets \max\,\bigl\{\, B_{k-1}[a] + \widehat P\bigl(\max\{0,\;
V - \underline{\delta}(g_{ab},d_{ab})\}\bigr) \;:\; a \prec b \,\bigr\}$
\Comment{$\max\emptyset = -\infty$}
\EndFor
\EndFor
\State $\Psi_R(0) \gets 0$; \quad $\Psi_R(k) \gets \max_{b \in R} B_k[b]$
for $k \ge 1$
\State $K_R \gets$ the largest $k$ with $\Psi_R(k) > -\infty$
\end{algorithmic}
\end{algorithm}

\section{The branch-and-check scheme as a whole}
\label{sm:bnc-alg}

Algorithm~\ref{alg:bnc} collects the pieces of Section~\ref{sec:bnc} of the
paper as they are run in its Section~\ref{sec:exp-bnc}, Lemma~\ref{lem:support}
identifies the support of the upwind cut, and
Proposition~\ref{prop:bnc} records what they guarantee together.
Numbers in parentheses refer to the equations of the paper.

\begin{algorithm}[htbp]
\caption{Branch-and-check with upwind cuts.}
\label{alg:bnc}
\begin{algorithmic}[1]
\Require the master of Section~\ref{sec:master} of the paper; the oracle of
its Proposition~\ref{prop:oracle}; a flagging tolerance $\varepsilon$
($10^{-3}$~kW in our runs); installation costs $C_c = 0$ (with
costs, subtract $\sum_{c \in X} C_c$ from $E$ where it is computed)
\Ensure the best layout found $X^\star$ with its exact energy $E^\star$,
and the final dual bound of the master
\State $E^\star \gets -\infty$; hand the master to the solver with lazy
constraints enabled
\Statex \textbf{at every integer candidate} $(x, \epivar)$ of the solver
(incumbent callback):
\State $X \gets \{c : x_c = 1\}$; for every $s$, one sweep returns
$w^{*}(X)$ and $p^{\mathrm{ex}}_c(X)$ for $c \in X$; $E \gets \sum_s p_s
\sum_{c \in X} p^{\mathrm{ex}}_c(X)$
\For{every $(s, c)$ with $c \in X$ and $\epivar^s_c >
p^{\mathrm{ex}}_c(X) + \varepsilon$}
\Comment{separation by comparison}
\State $U \gets \{u \in X \setminus \{c\} : d^s_{uc} > 0\}$
(Lemma~\ref{lem:support}, Remark~\ref{rem:support}); \
$q \gets P_s\bigl(w^{*}_c(U \cup \{c\})\bigr)$
\State \textbf{if} $q > p^{\mathrm{ex}}_c(X) + \varepsilon$
\textbf{then} $U \gets X \setminus \{c\}$, $q \gets p^{\mathrm{ex}}_c(X)$
\Comment{fallback: the same cut with full support} \label{ln:fallback}
\State add the upwind cut of Section~\ref{sec:cone} of the paper with $(U, q)$
as a lazy constraint
\EndFor
\If{$E > E^\star$} $E^\star \gets E$, $X^\star \gets X$, and queue
$\bigl(X, p^{\mathrm{ex}}(X)\bigr)$ for posting
\EndIf
\Statex \textbf{at the next node} (node callback): if a layout is queued,
post $x = \mathbf{1}_X$, $\epivar^s_c = p^{\mathrm{ex}}_c(X)$ as a
heuristic solution (Section~\ref{sec:posting} of the paper), which the solver adopts if it
improves on its incumbent: 308\ of the
422\ layouts posted in the 45 runs behind
Table~\ref{tab:bnc} of the paper were adopted, the rest not improving on the
incumbent the solver held
\State \Return $X^\star$, $E^\star$, and the solver's final dual bound
\end{algorithmic}
\end{algorithm}

\begin{lemma}[Support of the field at a cell]
\label{lem:support}
Fix a scenario $s$ and call \emph{propagation graph} of $s$ the directed
graph on $\mathcal{C}$ with an arc $u \to v$ whenever $u =
\mathrm{up}_s(v,m)$ for an offset $m$ with $(1-\gamma)K_m > 0$ (a
\emph{stencil arc}, row \eqref{eq:mip-flow} of the paper) or $v \in \mathcal{F}_s(u)$ (a
\emph{footprint arc}, row \eqref{eq:mip-wake}). For a cell $c$ let $\mathrm{Anc}_s(c)$ be
the set of allocatable cells $u$ from whose footprint $c$ is reachable:
some directed path, possibly of length zero, leads from a cell of
$\mathcal{F}_s(u)$ to $c$. Then for every layout $X \subseteq \mathcal{A}$
\[
w^{*}_c(X) \;=\; w^{*}_c\bigl((X \cap \mathrm{Anc}_s(c)) \cup \{c\}\bigr).
\]
\end{lemma}
\begin{proof}
Let $X' = (X \cap \mathrm{Anc}_s(c)) \cup \{c\}$, and let $R$ be the set
of cells from which $c$ is reachable in the propagation graph, together
with $c$. We show $w^{*}_v(X) = w^{*}_v(X')$ for every $v \in R$ by
induction along the propagation planes, which order the graph
topologically (proof of Proposition~\ref{prop:fixpoint}(b) of the paper). The sweep sets
$w^{*}_v$ to the minimum of the flow bound \eqref{eq:mip-flow}, which reads $w^{*}_u$
along the stencil arcs $u \to v$ (offsets of weight zero, and cells
outside the grid at free stream, contribute the same to both sides), and
of the wake bounds \eqref{eq:mip-wake} at $v$, which read $x_u$ and $w^{*}_u$ along the
footprint arcs $u \to v$. Every such $u$ lies in $R$, since $u \to v$
and $v$ reaches $c$, so $w^{*}_u(X) = w^{*}_u(X')$ by the inductive
hypothesis; and every footprint tail $u$ lies in $\mathrm{Anc}_s(c)$,
since $v \in \mathcal{F}_s(u)$ reaches $c$, so $u \in X$ if and only if
$u \in X'$ ($u \ne c$, the graph being acyclic). Both bounds therefore
agree at $v$, and so does their minimum. A cell of $R$ without incoming
arcs starts the induction with the free-stream flow bound on both sides,
and $c \in R$ closes it; $x_c$ itself enters no bound at $c$, as
$\mathcal{F}_s(c)$ lies downstream.
\end{proof}

\begin{remark}
\label{rem:support}
Algorithm~\ref{alg:bnc} does not build the graph: it takes for $U$ the
cells $u \in X \setminus \{c\}$ whose single-source deficit $d^s_{uc}$
(Section~\ref{sec:lifted} of the paper) exceeds $10^{-12}$~m/s.
In exact arithmetic this is $X \cap \mathrm{Anc}_s(c)$ whenever every
footprint arc is also a stencil arc. Under the layout $e_u$ the only
active wake bounds are those at the cells of $\mathcal{F}_s(u)$, where
the deficit is $(1-\rho)\wmax_s > 0$, and from there the deficit spreads
through the flow bound \eqref{eq:mip-flow} exactly along the stencil arcs; so $d^s_{uc} >
0$ if and only if $c$ is reachable from $\mathcal{F}_s(u)$ through stencil
arcs alone, and the two sets coincide when every offset of the footprint
carries a positive stencil weight. At $\Delta = D$ the footprint is the
single downstream cell and the condition reads $\kappa_0 = (1-\gamma)K_0
> 0$, which holds here ($\kappa_0 = 0.83$). In
floating point the deficit decays at every transversal step and falls
below the threshold, or to machine zero, a few cells off the wake axis,
so the set the code builds is contained in $X \cap \mathrm{Anc}_s(c)$
rather than equal to it. We measured both facts on the instances of the
paper, at 7\,284\ pairs (scenario, cell): the set
of the code lies inside $\mathrm{Anc}_s(c)$ with
0\ exceptions, and over all pairs
$\mathrm{Anc}_s(c)$ holds 288\,896\ cells more,
all with a deficit at $c$ below the threshold. On
270\ feasible layouts
(14\,408\ checks) the identity of the Lemma holds
with $\mathrm{Anc}_s(c)$ with a largest difference of
0~m/s, and with the set of the code of
$2\cdot 10^{-12}$~m/s, that is $1.1\cdot 10^{-9}$~kW
on the constant $q$. The inclusion is the direction the scheme needs:
$U \cup \{c\} \subseteq X$ gives $q \ge p^{\mathrm{ex}}_c(X)$ by
antitonicity (Proposition~\ref{prop:fixpoint}(d) of the paper), the cut is valid for every
$U$ (Proposition~\ref{prop:cone}), and it separates the flagged candidate as soon as
$q - p^{\mathrm{ex}}_c(X) < \varepsilon$, which the figure above leaves
ample room for; the fallback of Algorithm~\ref{alg:bnc}
(line~\ref{ln:fallback}) to the full support $U = X \setminus \{c\}$,
whose constant is $p^{\mathrm{ex}}_c(X)$ by construction, covers the
remaining case.
\end{remark}

\begin{proposition}[Correctness and finite termination]
\label{prop:bnc}
Let $\mathcal{M} = \{(x, \epivar) : x \text{ feasible for
\eqref{eq:mip-dist}},\ 0 \le \epivar^s_c \le p^{\mathrm{ex}}_c(x)\, x_c\}$,
whose optimal objective is the exact optimum $E^{\mathrm{opt}}$ of the
layout problem under the true power curve, with installation costs
$C_c = 0$ as in the main campaign of the paper (with costs, the objective
is net of the linear cost term, and the proof is unchanged). Algorithm~\ref{alg:bnc} adds
finitely many cuts, every relaxation the solver works on contains
$\mathcal{M}$, and at termination $E^\star \le E^{\mathrm{opt}} \le
\text{bound}$; at termination with proven optimality, in exact
arithmetic with $\varepsilon = 0$ and the solver's tolerances set to
zero, $E^\star = E^{\mathrm{opt}}$.
With $\varepsilon > 0$ and the solver's default tolerances, as in
our runs, $E^\star \le E^{\mathrm{opt}}$ holds as stated, $E^\star$
being the oracle value of a feasible layout, and $E^{\mathrm{opt}} \le
\text{bound}$ holds up to the solver's feasibility and optimality
tolerances, every cut being valid for the optimal point in exact
arithmetic; at any termination, time limit included, $E^{\mathrm{opt}}
- E^\star \le \text{bound} - E^\star$, which, relative to $E^\star$, is
the gap Table~\ref{tab:bnc} of the paper reports. The tolerance $\varepsilon$ enters only the solver's own
incumbent: a candidate the callback does not flag has objective above
its exact energy by at most $\varepsilon$ per installed cell and
scenario, hence, the scenario probabilities summing to one, by at most
$|X|\,\varepsilon \le N_{\max}\,\varepsilon$ ($0.026$~kW on the largest
instance of the main ladder); it decides when the search stops and whether a posted layout
improves on the incumbent, never the certificate.
\end{proposition}
\begin{proof}
Every constraint of the master and every upwind cut is satisfied by
every point of $\mathcal{M}$ (Propositions~\ref{prop:bi} and~\ref{prop:cone} of the paper), so the solver's dual bound is an upper bound on
$E^{\mathrm{opt}}$ throughout, and the optimal point
$(\mathbf{1}_{X^{\mathrm{opt}}}, p^{\mathrm{ex}}(X^{\mathrm{opt}}))$ is
never cut off. Each posted point lies in $\mathcal{M}$, so $E^\star$,
the largest exact value seen, satisfies $E^\star \le E^{\mathrm{opt}}$.
A candidate $(x, \epivar)$ at which no pair is flagged has $\epivar^s_c
\le p^{\mathrm{ex}}_c(X) + \varepsilon$ everywhere (and $\epivar^s_c =
0$ off $X$ by \eqref{eq:master-cap}), so with $\varepsilon = 0$ it lies
in $\mathcal{M}$ and its objective is at most $E(X) \le E^\star$ after
the update: the solver's incumbent value never exceeds $E^\star$, and
the search cannot stop with the optimum unvisited. For finiteness, the
cut added for a flagged $(X, s, c)$ has $(U, q)$ determined by $X$
with $q \le p^{\mathrm{ex}}_c(X) + \varepsilon$: whenever the
computed $q$ exceeds $p^{\mathrm{ex}}_c(X) + \varepsilon$, the fallback
step of Algorithm~\ref{alg:bnc} emits instead the cut with $U = X
\setminus \{c\}$, whose constant is $p^{\mathrm{ex}}_c(X)$ by
construction (in exact arithmetic, and whenever every offset of the
footprint carries positive weight, this never happens: the support $U$
of Algorithm~\ref{alg:bnc} then contains $X \cap \mathrm{Anc}_s(c)$, and
Lemma~\ref{lem:support} applied to $X$ and to $U \cup \{c\}$ gives
$w^{*}_c(U \cup \{c\}) = w^{*}_c(X)$, hence $q = p^{\mathrm{ex}}_c(X)$),
and at $x = \mathbf{1}_X$ it reads $\epivar^s_c \le q$ $\le
p^{\mathrm{ex}}_c(X) + \varepsilon$, so in exact arithmetic the same
triple is never flagged twice. With tolerances, the
solver enforces the cut to within its feasibility tolerance $\tau_F =
10^{-6}$, so a second flag would need $\epivar^s_c > p^{\mathrm{ex}}_c(X)
+ \varepsilon$ with $\epivar^s_c \le q + \tau_F$, impossible as long as
$q - p^{\mathrm{ex}}_c(X) \le \varepsilon - \tau_F$, a condition we
check numerically on the audited layouts, not a theorem: there the excess
never exceeds $1.1\cdot 10^{-9}$~kW
(Remark~\ref{rem:support}), six orders of magnitude inside that margin. Hence, in exact arithmetic, at most
$|\mathcal{A}|\,|\mathcal{S}|$ cuts per layout,
finitely many layouts, and a branch-and-bound over finitely many binaries
with finitely many lazy constraints terminates.
\end{proof}

\paragraph{Checks on the returned layouts.} The master runs at default
tolerances (relative gap $10^{-4}$): on every instance it closes, the
final bound exceeds the incumbent's exact energy by less than
0.01~kW. Each layout returned in
Section~\ref{sec:exp-bnc} of the paper was re-scored outside the
callback by the same propagator, the largest discrepancy from the
reported energy being $0$\,kW---a check
on the bookkeeping rather than on the propagator, the arithmetic being
the same by another route---and each was re-tested against the minimum
separation, which returned 0\ violated
pairs. As an independent check on the propagator, the maximal field of
each returned layout was also recomputed as a linear program with $x$
fixed, an implementation that shares nothing with the sweep; the two
agree to within $5\cdot 10^{-8}$~m/s, and the exact energy to
within $9\cdot 10^{-6}$~kW.

\paragraph{The effect of exact posting.} Posting the incumbent at its
exact value, the component of Section~\ref{sec:posting} of the paper a
reader may take for bookkeeping, pays on both sides of the time limit. On
$10\times10$\unskip, the largest instance of the set
that closes either way, posting proves optimality in
77~s against
96~s---20\%\ less
time for the same certificate. On the instances that do not close it yields
a better incumbent instead, by up to 0.85\%\ on
$16\times16$\unskip. In both cases posting makes the
layout available to the solver at its oracle value at once; the ablation
measures the resulting change in the search.

\paragraph{The band inequalities in the master.} With the band
inequalities removed from the master, the runs of
Section~\ref{sec:exp-bnc} of the paper close nothing above
$8\times8$\unskip, and the gap grows on every one
of the 30\ open instance-seed pairs, by
5.7\ to 10.7\ points at
the median.

\paragraph{The support of the cut.} With the upwind cut at the full
support $U = X \setminus \{c\}$, the full-support variant of Section~\ref{sec:exp-bnc}
of the paper, the runs close nothing above
$8\times8$\unskip, leave the
$10\times10$\ at 7.9\%\
after the hour, and need 5.3$\times$ as many
cuts; where neither support closes, the median gaps stay within
1.0\ points.

\section{The lifted wake constraint}
\label{sm:lifted}

Section~\ref{sec:lifted} of the paper states informally that lifting
the wake rows \eqref{eq:mip-wake} over the cliques of the conflict
graph does not move the vacuous bound of its
Proposition~\ref{prop:vacuous}. This section states and proves the
result. Fix a scenario $s$; for a cell $u \in \mathcal{A}$ and any cell
$v$, the \emph{single-source deficit} $d^s_{uv}$ is the value the deficit
$d^s_v = \wmax_s - w^s_v$ takes under the maximal field of the
single-turbine layout $x = e_u$ (Proposition~\ref{prop:fixpoint}(b) of
the paper), and $\mathcal{W}(x)$ is the set of feasible fields of
Section~\ref{sec:model}.

\begin{proposition}[Lifted wake constraint]
\label{prop:lifted}
Fix a scenario whose direction is aligned with a grid axis, at the
resolution where
$\mathcal{F}_s(c)$ is the single downstream cell $c'$. Let $Q \ni c$ be a
clique of allocatable cells, pairwise closer than the minimum separation,
so that at most one of them can carry a turbine.
\begin{enumerate}
\item[(a)] Let $\kappa_0 = (1-\gamma)K_0$ be the weight with which
$w^s_c$ enters the flow bound of its own footprint cell. If it satisfies
$\kappa_0 \ge \rho$---on our instances,
$0.83$ against $0.293$---then
\begin{equation}
w^s_{c'} \;\le\; \rho\, w^s_c \;+\; (1-\rho)\,\wmax_s\,(1-x_c)
\;-\; \sum_{u \in Q\setminus\{c\}}
\underbrace{\bigl(d^s_{uc'} - \rho\, d^s_{uc}\bigr)}_{\beta^s_u}\, x_u
\label{eq:lifted}
\end{equation}
is valid for the model, for every feasible layout and every
$w \in \mathcal{W}(x)$; the $\beta^s_u$ are non-negative, so
\eqref{eq:lifted} dominates \eqref{eq:mip-wake}; and the choice is
componentwise maximal: if $N_{\min} = 0$ and $N_{\max} \ge 1$,
\eqref{eq:lifted} holds with equality at the
empty layout and at every single-turbine layout $x = e_u$, $u \in Q$, so
no coefficient can be raised and the constant cannot be lowered.
\item[(b)] Let the direction of every scenario be aligned with a grid
axis, let every
constraint \eqref{eq:mip-wake} be replaced by its lifted form
\eqref{eq:lifted}, in the setting and under the hypothesis of (a), the
clique $Q_c \ni c$ of each row chosen arbitrarily, and let $A_s$ be the
largest total lifting
$\sum_{u \in Q_c\setminus\{c\}} \beta^s_u$ of a row.
With $v_r$ the rated speed---equivalently, since the breakpoints include
it, the smallest speed at which $\bar{P}$ attains $P_{\max}$---and
$\bar{x} = N_{\max}/|\mathcal{A}|$ the uniform value of
Proposition~\ref{prop:vacuous}, if
$\bar{x} \le \min\bigl\{1/q_{\max},\;
\min_s \wmax_s/(\wmax_s + v_r + A_s/(1-\rho))\bigr\}$, then the optimal
value of the relaxation of Proposition~\ref{prop:vacuous} is still
exactly $N_{\max}P_{\max}$.
\end{enumerate}
\end{proposition}

\begin{proof}
(a) At most one variable of $Q$ equals one. If $x_c = 1$, row
\eqref{eq:lifted} reads $w^s_{c'} \le \rho w^s_c$, which is
\eqref{eq:mip-wake}. If $x \equiv 0$ on $Q$, the flow bound
\eqref{eq:mip-flow} at $c'$ gives, bounding the lateral parents by
$\wmax_s$, $w^s_{c'} \le (1-\kappa_0)\wmax_s + \kappa_0 w^s_c$, hence
$w^s_{c'} - \rho w^s_c \le (1-\kappa_0)\wmax_s + (\kappa_0 - \rho)w^s_c
\le (1-\rho)\wmax_s$, using $\kappa_0 \ge \rho$ and $w^s_c \le \wmax_s$.
If $x_u = 1$ for the necessarily unique $u \in Q\setminus\{c\}$, then
$x \ge e_u$ componentwise, and Proposition~\ref{prop:fixpoint}(b,d) caps
every cell by the maximal field of the single-turbine layout $e_u$: each
parent $n_m = \mathrm{up}_s(c',m)$, $m = -2,\dots,2$, of $c'$ obeys
$w^s_{n_m} \le \wmax_s - d^s_{un_m}$, and
$w^s_c \le \wmax_s - d^s_{uc}$. Substituting the caps into the flow bound
at $c'$ and collecting the central channel,
$w^s_{c'} - \rho\, w^s_c \le \gamma \wmax_s +
(1-\gamma)\sum_{m \ne 0} K_m\,(\wmax_s - d^s_{un_m})
+ (\kappa_0 - \rho)\, w^s_c$,
and since $\kappa_0 \ge \rho$ the last term is largest at
$w^s_c = \wmax_s - d^s_{uc}$; rearranging, the right-hand side becomes
$(1-\rho)\wmax_s - (1-\gamma)\sum_m K_m\, d^s_{un_m} + \rho\, d^s_{uc}$,
the sum now over all five parents (parents outside the grid carry the
undisturbed speed, hence no deficit). Since $u \ne c$, the cell $c'$ is not
the footprint cell of $u$, so the forward sweep of
Proposition~\ref{prop:fixpoint}(b) computes the maximal field of $e_u$ at
$c'$ from the flow bound with equality, and the sum equals $d^s_{uc'}$:
this is \eqref{eq:lifted}. Non-negativity: the same recursion, keeping
only its central term, gives
$d^s_{uc'} \ge \kappa_0\, d^s_{uc} \ge \rho\, d^s_{uc}$. Maximality: at
the empty layout $w \equiv \wmax_s$ is feasible and \eqref{eq:lifted} is
tight; at $x = e_c$ the maximal field has $w^s_{c'} = \rho\, w^s_c$,
which is what the row reads; at $x = e_u$ it has
$w^s_{c'} = \wmax_s - d^s_{uc'}$ and $w^s_c = \wmax_s - d^s_{uc}$, and
substitution gives equality again. Raising any coefficient, or lowering
the constant, cuts one of these feasible points.

(b) The value is at most $N_{\max}P_{\max}$, exactly as in
Proposition~\ref{prop:vacuous}. For the converse, put $x \equiv \bar{x}$,
$z^s_c = \bar{x}P_{\max}$ and, for each sector, $w^s \equiv \bar{w}_s$
with $\bar{w}_s = \wmax_s - \bar{x}\,(\wmax_s + A_s/(1-\rho))$; the
condition, rearranged, reads $\bar{w}_s \ge \bar{x}\,v_r \ge 0$, and
$\bar{w}_s \le \wmax_s$ since $A_s \ge 0$ by (a). The flow
bound holds at a uniform field, its right-hand side being at least
$\gamma \wmax_s + (1-\gamma)\bar{w}_s \ge \bar{w}_s$. A lifted row
with total lifting $\Sigma \le A_s$ asks
$\bar{w}_s \le \wmax_s - \bar{x}(\wmax_s + \Sigma/(1-\rho))$, which holds
by construction. For the power block, the largest value of
$\sum_b P_b \lambda^s_{cb}$ compatible with the $\lambda$-system at
$\bar{w}_s$ is $\bar{P}(\bar{w}_s)$: if $\bar{w}_s \ge v_r$ this is
$P_{\max} \ge \bar{x}P_{\max}$, and otherwise concavity of $\bar{P}$ with
$\bar{P}(0) \ge 0$ gives
$\bar{P}(\bar{w}_s) \ge P_{\max}\,\bar{w}_s/v_r \ge \bar{x}P_{\max}$, the
last step being again the condition. Separation and cardinality are
handled as in the proof of Proposition~\ref{prop:vacuous}. The point is
feasible and attains $N_{\max}P_{\max}$.
\end{proof}

Part (b) is the one that matters: no choice of $\beta$, maximal or
not, improves the vacuous bound---the uniform point retreats to a
slightly lower field and survives. Nor is the hypothesis restrictive on
the instances of the paper: the largest total lifting across the campaigns of
its Section~\ref{sec:exp} is $A_s = 11.4$\,m/s,
the resulting density threshold is 0.274
(with $v_r = 10.34$~m/s and $\rho =
0.293$),
and the densest instance tested has
$\bar{x} = 0.222$:  the wind term of
the condition holds with a factor 1.2\ to
spare, and the packing term $1/q_{\max}$ is the one of
Proposition~\ref{prop:vacuous}, so part (b) applies exactly where that
proposition does.

\section{The band inequalities in floating point}
\label{sm:fp}

The coefficients $(a_j, b_j)$ of the band inequalities
\eqref{eq:cut-t2} are floating-point numbers, with $b_j =
\Psi_R(k_1) - a_j k_1$ anchored at the left vertex $k_1$ of the facet, so
the facet reproduces $\Psi_R(k_1)$ there to within a single floating-point
rounding (about $10^{-16}$ in relative terms). The
representation error thus lies orders of magnitude below the default
feasibility tolerance of the solver, $10^{-6}$, and certificates are
reported to the kilowatt. In exact rational arithmetic the facets
fall below the points $(k, \Psi_R(k))$ they support by at most
$9.1\times10^{-13}$~kW, never more than one unit in the last place.
On the nine instances of
Table~\ref{tab:scaling} of the paper, under both power representations the paper uses
(the SOS2 interpolant of the monolithic model and the true curve of the
branch-and-check master), band inequalities were evaluated
1\,054\,476\ times on feasible integer layouts---random ones,
and rows and columns packed at the minimum separation at every offset,
the only ones that fill a band to capacity---and none was violated.

\section{The user-cut pool, configuration by configuration}
\label{sm:pool}

\paragraph{Which bands to choose.} The BI family of
Section~\ref{sec:cuts} of the paper spans every width and every offset.
A single band is of little use on its own: the relaxation spreads its
fractional turbines over the neighboring bands and evades it. Offsets
matter less: the bands carrying a nonzero dual multiplier at the root are
only 26\%--52\%\ of
those generated, on all 4\ instances examined
their offset is a multiple of their width, and one partition per
width---77\%\ of the family---reproduces the
root bound 100\%\ exactly; at the root the
width $L = 2$ alone ($L = 1$ on the smallest) closes the whole of the gap
the family closes. We nonetheless generate the full family throughout:
its computational cost is negligible (at most
1044\ inequalities over the four scenarios,
generated once before the search starts), and what the root does not
need the search may.

\paragraph{The pool.} Section~\ref{sec:exp-cuts} of the paper keeps the band inequalities of
the monolithic model in the solver's user-cut pool, except the ones active
at the root (option (c) there), and reports the headline numbers of the
comparison. The comparison in full is the following.
The branch-and-check master of Section~\ref{sec:bnc} keeps its own band
inequalities static, and we verified that this is the right choice there:
the incumbent comes from the oracle rather than from the search, so the
primal difficulty has nothing to gain, while the weaker master loses on
the bound, which is all that is left to gain. Pooling them in the master
costs 1.1\ to 3.0\
percentage points of certificate on the instances that do not close, and
wins on none\ of the
25\ instance-seed pairs. Over the same three grids and
5\ random seeds the pooled monolithic model uses all available turbines
($N_{\max}$) in
14\ runs out of
15\ and gives the best certificate of the
three configurations compared---pooled, static, and no inequalities at all---in 7\ of the
15\ instance-seed pairs; the median
certificate on the exact energy scale is
20.0\%\ against
35.8\%\ with the static constraints, and
20.6\%\ with no inequalities at all, and
54.0\%\ against
77.3\%\ and
68.2\%\ on the largest. However, the pool is not
uniformly better: on $12\times12$\unskip, where the
tree does still move the bound, pooling inequalities outside the model
costs dual strength and the static form wins, 11.8\%\ against
13.1\%\unskip. In other words, the choice pays on
the sizes where the bound has stopped moving.

\section{Primal and dual sides, separated}
\label{sm:pd}

Table~\ref{tab:ejor-pd} separates, for the three exact methods of the
paper and on every grid of the ladder, the dual side (root bound and
final bound of the model each method solves) from the primal side (the
exact energy of the returned layout), as column-wise medians over the
seeds; the paragraph
``Primal and dual sides'' of Section~\ref{sec:exp-bnc} of the paper reads
it. On the grids left open, an hour of branching moves the bound by at
most 0.2\%\ in the monolithic model
with the band inequalities, 7.6\%\
without them, and 2.3\%\ in the
branch-and-check master.

\begin{table}[htbp]
\centering\small
\caption{Primal and dual sides separated, $3600$~s budget, median over the five seeds: the root bound (the linear relaxation of Table~\ref{\PREF tab:bounds} for the monolithic model, the master's root for branch-and-check), the final dual bound after the hour, and the exact energy of the incumbent, for the monolithic model without and with band inequalities and for branch-and-check. Bounds of the monolithic model are on the SOS2 surrogate, those of branch-and-check on the exact energy. Final bounds and incumbents are column-wise medians over the seeds, hence not the median run of Tables~\ref{\PREF tab:scaling} and~\ref{\PREF tab:bb}.}
\label{tab:ejor-pd}
\begin{tabular}{l rrr rrr rrr}
\toprule
 & \multicolumn{3}{c}{monolithic} & \multicolumn{3}{c}{monolithic + BI} & \multicolumn{3}{c}{branch-and-check} \\
\cmidrule(lr){2-4}\cmidrule(lr){5-7}\cmidrule(lr){8-10}
grid & root & final & incumb. & root & final & incumb. & root & final & incumb. \\
\midrule
6$\times$6 & 12\,000 & 10\,777 & 10\,598 & 10\,820 & 10\,777 & 10\,598 & 10\,643 & 10\,598 & 10\,598 \\
7$\times$7 & 23\,983 & 14\,629 & 14\,347 & 16\,284 & 14\,629 & 14\,347 & 16\,022 & 14\,347 & 14\,347 \\
8$\times$8 & 27\,000 & 18\,391 & 18\,022 & 20\,593 & 18\,391 & 18\,022 & 20\,273 & 18\,022 & 18\,022 \\
10$\times$10 & 36\,000 & 26\,248 & 25\,689 & 30\,097 & 26\,248 & 25\,689 & 29\,604 & 25\,689 & 25\,689 \\
12$\times$12 & 42\,000 & 38\,815 & 32\,608 & 36\,832 & 36\,747 & 32\,486 & 36\,262 & 35\,411 & 32\,793 \\
14$\times$14 & 48\,000 & 44\,656 & 37\,121 & 42\,274 & 42\,274 & 37\,123 & 41\,474 & 41\,083 & 37\,867 \\
16$\times$16 & 54\,000 & 50\,128 & 41\,618 & 47\,838 & 47\,838 & 39\,880 & 46\,957 & 46\,856 & 42\,698 \\
20$\times$20 & 66\,000 & 61\,242 & 48\,309 & 59\,100 & 59\,100 & 47\,245 & 58\,075 & 58\,075 & 52\,873 \\
24$\times$24 & 78\,000 & 72\,373 & 43\,038 & 70\,393 & 70\,393 & 45\,704 & 69\,234 & 69\,234 & 63\,429 \\
\bottomrule
\end{tabular}
\end{table}

\section{The common warm start}
\label{sm:ws}

The paragraph ``Primal and dual sides'' of Section~\ref{sec:exp-bnc} of the
paper reports an experiment in which the two exact methods start from the
same heuristic layout. The start is, for each
grid left open, the better of a greedy construction and a pack-and-prune
construction, refined by a local search with relocation, swap and
on--off moves for 300\ s on a laptop
(Intel Core i9-8950HK at $2.9$\,GHz; 301\ s in
all, median over the grids, not charged to the hour of the exact
methods); its
exact energy is $E_0$ in Table~\ref{tab:ejor-ws}. The monolithic model
receives it as a complete MIP start (the layout, the speeds it induces
and the adjacent pair of each power-curve SOS2 set, so that the solver
has only to verify it); branch-and-check receives it as a posted
incumbent with its exact energy, the same mechanism it uses for its own
incumbents. Each method then runs for the usual hour, five seeds, on the
instances of the paper; the cold columns are runs on the same instances
and seeds. In the monolithic model the band inequalities are static rows,
warm and cold (Section~\ref{sec:exp-cuts} of the paper says why). Ten
cold controls rerun under the instrumented code (24$\times$24, both
methods, five seeds) reproduce the cold runs
within 0.00\% in incumbent and bound.

\begin{table}[htbp]
\centering\footnotesize
\caption{The common warm start of Section~\ref{\PREF sec:exp-bnc} of the paper: the two exact methods, on the grids left open, started from the same heuristic layout (greedy construction and local search, five minutes on the laptop, exact energy $E_0$) and run for the usual hour, against cold runs on the same instances and seeds; in the monolithic model the band inequalities are static rows, warm and cold. For each method: the exact energy of the final incumbent, cold and warm (median over the five seeds); the number of seeds whose warm run improves on $E_0$ within the hour; the median time of the first such improvement over those seeds; the final dual bound, cold and warm (median). Bounds of the monolithic model are on the SOS2 surrogate, those of branch-and-check on the exact energy.}
\label{tab:ejor-ws}
\setlength{\tabcolsep}{3.5pt}
\begin{tabular}{l r rrcr rr rrcr rr}
\toprule
 & & \multicolumn{6}{c}{monolithic + BI} & \multicolumn{6}{c}{branch-and-check} \\
\cmidrule(lr){3-8}\cmidrule(lr){9-14}
 & & \multicolumn{2}{c}{incumbent} & & & \multicolumn{2}{c}{bound} & \multicolumn{2}{c}{incumbent} & & & \multicolumn{2}{c}{bound} \\
grid & $E_0$ & cold & warm & impr. & $t_1$\,[s] & cold & warm & cold & warm & impr. & $t_1$\,[s] & cold & warm \\
\midrule
12$\times$12 & 32\,486 & 32\,445 & 32\,486 & 1/5 & 381 & 36\,263 & 36\,130 & 32\,793 & 32\,793 & 4/5 & 1\,742 & 35\,411 & 35\,411 \\
14$\times$14 & 37\,908 & 36\,907 & 37\,908 & 0/5 & -- & 42\,061 & 42\,154 & 37\,867 & 37\,908 & 1/5 & 143 & 41\,083 & 41\,083 \\
16$\times$16 & 42\,636 & 35\,200 & 42\,636 & 0/5 & -- & 47\,807 & 47\,838 & 42\,698 & 42\,636 & 0/5 & -- & 46\,856 & 46\,856 \\
20$\times$20 & 53\,263 & 48\,604 & 53\,263 & 0/5 & -- & 59\,100 & 59\,100 & 52\,873 & 53\,311 & 4/5 & 8 & 58\,075 & 58\,075 \\
24$\times$24 & 63\,990 & 39\,696 & 63\,990 & 0/5 & -- & 70\,393 & 70\,393 & 63\,429 & 63\,990 & 1/5 & 4 & 69\,234 & 69\,234 \\
\bottomrule
\end{tabular}
\end{table}

Table~\ref{tab:ejor-ws} gives the result. The start is already above the
median incumbent of the cold hour in
8\ of the
10\ grid-method pairs. Warm, the monolithic
model with static band rows improves on it in
1\ of
25\ runs and branch-and-check in
10\ of 25,
by at most 0.94\%. A good incumbent
moves the certificate, not the bound: that of the monolithic model moves
by at most 0.56\%, and the warm
incumbents of the two methods end within
0.94\% of each other, so that the bound of
branch-and-check, lower by
1.6--2.5\%,
decides the difference between the certificates.

Figure~\ref{fig:ejor-ws} shows the hour on the largest grid: the exact
energy of the incumbent and the dual bound against time, cold and warm,
on the median seed of each arm. Warm, both incumbents stay at $E_0$, at the reported precision, for
the whole hour, and the warm bounds lie on the cold ones: the start
largely removes the incumbent deficit of the cold runs and leaves the dual
side as it was. The comparison isolates the effect of the start for the
static-row configuration; it does not rank the pool policies under warm
starts.

\begin{figure}[htbp]
\centering
\includegraphics[width=0.95\textwidth]{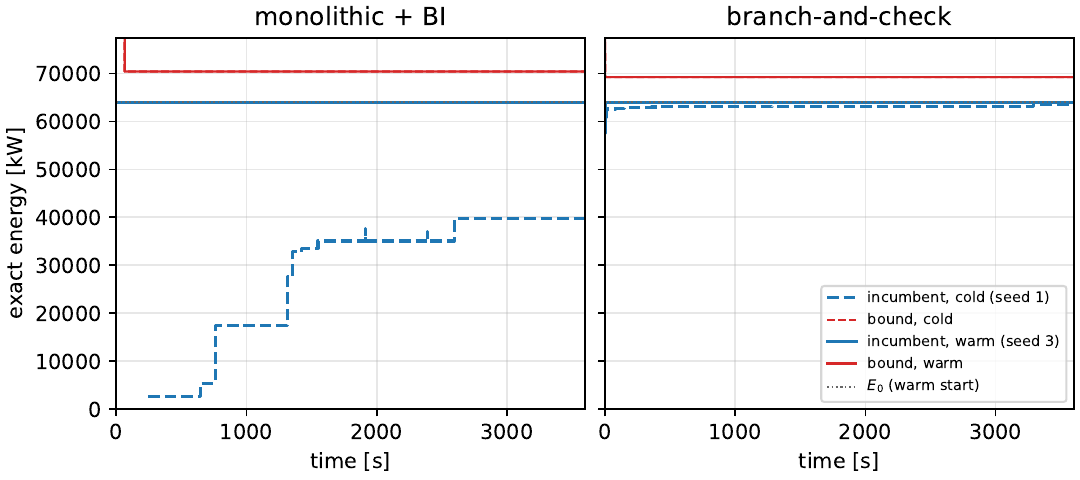}
\caption{The hour on the 24$\times$24 grid, cold (dashed) and warm
(solid), median seed by final incumbent: exact energy of the incumbent
(steps) and dual bound (thin line) against time, for the monolithic model
with band inequalities (left) and branch-and-check (right). The dotted
line is the exact energy $E_0$ of the common warm start. The bound of the
monolithic model is on the SOS2 surrogate.}
\label{fig:ejor-ws}
\end{figure}

\section{Larger grids}
\label{sm:big}

This section extends the common warm start of Section~\ref{sm:ws} to two
grids beyond those of the paper, $32\times32$ and $40\times40$. The instances continue the ladder
of the paper: the axial rose, zero costs and $N_{\max} = n+2$. The farm
therefore gets sparser as the grid grows: the density-equivalent spacing
$\sqrt{n^2/N_{\max}}$ is 4.7,
5.5 and
6.2\ rotor diameters on the
$24\times24$, $32\times32$ and $40\times40$ grids, within the range of
offshore practice. These grids test the size of the model; a last
instance tests the density, on $40\times40$ with $N_{\max} = 100$ (density
6.2\%, density-equivalent spacing
4.0\ diameters, close to those of the
$16\times16$ grid of the paper). The start follows the recipe of
Section~\ref{sm:ws}, with the same budget, and the band inequalities of
the monolithic model are again static rows. Each method then runs for the
usual hour, five seeds, on Intel Xeon E5-2623~v3 machines at $3.00$\,GHz
(four cores): times do not compare with those of the paper, while the
gaps are valid certificates on each instance.

\begin{table}[htbp]
\centering\small
\caption{Larger grids, from the common warm start of Section~\ref{sm:ws}, with the band inequalities of the monolithic model as static rows. $|C| = n^2$ cells, $N_{\max} = n+2$, so the farm gets sparser as the grid grows: \emph{spacing} is $\sqrt{n^2/N_{\max}}$ in rotor diameters. $E_0$ is the exact energy of the start, \emph{best} that of the incumbent at the time limit (median over five seeds), \emph{impr.} the number of seeds that improve on $E_0$, \emph{bound} the dual bound (median), \emph{gap} their difference relative to \emph{best} (range over the seeds; for the monolithic model on the SOS2 surrogate, as the column \emph{cert} of the paper). The $24\times24$ rows are the warm runs of Table~\ref{tab:ejor-ws}; $32\times32$ and $40\times40$ run on Intel Xeon E5-2623 v3 machines, so their times do not compare with those of the paper; the reported gaps remain valid instance-specific certificates. The last two rows keep the $40\times40$ grid and raise $N_{\max}$ to 100, from a warm start built in the same way.}
\label{tab:ejor-big}
\setlength{\tabcolsep}{5pt}
\begin{tabular}{lrrrlrrrrr}
\toprule
grid & $|C|$ & $N_{\max}$ & spacing & method & $E_0$ & best & impr. & bound & gap \\
\midrule
$24\times24$ & 576 & 26 & 4.7 & branch-and-check & 63\,990 & 63\,990 & 1/5 & 69\,234 & 8.19\% \\
 & & &  & monolithic + BI & 63\,990 & 63\,990 & 0/5 & 70\,393 & 10.01\% \\
\addlinespace
$32\times32$ & 1024 & 34 & 5.5 & branch-and-check & 85\,879 & 85\,879 & 0/5 & 91\,444 & 6.48\% \\
 & & &  & monolithic + BI & 85\,879 & 85\,879 & 0/5 & 92\,874 & 8.15\% \\
\addlinespace
$40\times40$ & 1600 & 42 & 6.2 & branch-and-check & 107\,581 & 107\,581 & 0/5 & 113\,442 & 5.45\% \\
 & & &  & monolithic + BI & 107\,581 & 107\,581 & 0/5 & 115\,165 & 7.05\% \\
\addlinespace
$40\times40$ & 1600 & 100 & 4.0 & branch-and-check & 202\,314 & 202\,314 & 0/5 & 257\,793 & 27.42\% \\
 & & &  & monolithic + BI & 202\,314 & 202\,314 & 0/5 & 262\,929 & 29.96\% \\
\bottomrule
\end{tabular}
\end{table}

Table~\ref{tab:ejor-big} gives the result. Neither method improves on the
start in any of the 20\ runs on the
two new grids. The gap of branch-and-check is
6.5\%\ on $32\times32$ and
5.4\%\ on $40\times40$, against
8.1\%\ and
7.1\%\ for the monolithic model:
the ranking of the paper holds on every grid and seed. Both gaps are
smaller than on $24\times24$, which is consistent with the sparser farm,
whose weaker wakes leave less room between the bound and the start.
The denser instance is consistent with this reading from the other side. Neither
method improves on its start either, and the gap rises to
27.4\%\ for branch-and-check and
30.0\%\ for the monolithic
model: on the same grid the gap grows with the density, and at a density
close to that of $16\times16$ it is
27.4\%\ against
9.74\%\ there, which is consistent with a gap
sensitive to both size and density, and the ranking holds. Part of that
gap may be primal: the start is found in
300\ s.

\section{A finer power curve}
\label{sm:seg}

The monolithic model represents the power curve by an SOS2 interpolant
with 6\ segments, and the interpolant
overestimates the curve: on a run that closes, the certificate
\emph{cert} of the paper is this interpolation bias alone. To measure
the effect of a finer interpolant, we reran the monolithic model with the
band inequalities as static rows, with 12 and 24 segments, on the instances of the
paper, seeds 1--3, for one hour on the machines of the paper
(30\ runs; 24 segments up to
$10\times10$ only). The six-segment rows are runs of the same
configuration on the same seeds, from the campaign of the paper.

\begin{table}[htbp]
\centering\small
\caption{The monolithic model with the band inequalities as static rows and a finer power curve, seeds 1--3, one hour. \emph{segments} is the number of linear pieces of the SOS2 interpolant (six in the paper), \emph{closed} the runs that prove optimality, \emph{cert} the range of $(UB-E)/E$, $E$ the exact expected power of the incumbent, over the seeds (on a closed run it is the interpolation bias alone), $t$ the median time of the closed runs and \emph{bound} the median of $UB$.}
\label{tab:ejor-seg}
\begin{tabular}{lrrrrr}
\toprule
grid & segments & closed & cert & $t$ & bound \\
\midrule
$6\times6$ & 6 & 3/3 & 1.70\% & 22\,s & 10\,777 \\
 & 12 & 3/3 & 0.35\% & 14\,s & 10\,634 \\
 & 24 & 3/3 & 0.014\% & 52\,s & 10\,599 \\
\addlinespace
$7\times7$ & 6 & 3/3 & 1.97\% & 19\,s & 14\,629 \\
 & 12 & 3/3 & 0.31\% & 44\,s & 14\,391 \\
 & 24 & 3/3 & 0.027\% & 74\,s & 14\,351 \\
\addlinespace
$8\times8$ & 6 & 3/3 & 2.05\% & 129\,s & 18\,391 \\
 & 12 & 2/3 & 0.36--0.76\% & 668\,s & 18\,087 \\
 & 24 & 3/3 & 0.030\% & 1542\,s & 18\,027 \\
\addlinespace
$10\times10$ & 6 & 3/3 & 2.18\% & 867\,s & 26\,248 \\
 & 12 & 0/3 & 2.50--20.18\% & -- & 26\,956 \\
 & 24 & 0/3 & 3.90--11.44\% & -- & 27\,079 \\
\addlinespace
$16\times16$ & 6 & 0/3 & 35.82--52.68\% & -- & 47\,807 \\
 & 12 & 0/3 & 61.92--74.34\% & -- & 47\,104 \\
\addlinespace
$24\times24$ & 6 & 0/3 & 121.49\% & -- & 70\,393 \\
 & 12 & 0/3 & 32.77--769.00\% & -- & 69\,502 \\
\bottomrule
\end{tabular}
\end{table}

Table~\ref{tab:ejor-seg} gives the result. The bias falls quickly with
the number of segments: on the grids up to $8\times8$ it is
1.70\%--2.05\%\ with
six segments, 0.31\%--0.36\%\ with
12 and 0.014\%--0.030\%\ with
24. On the other hand, the solve time grows: the median time to close $8\times8$
goes from 129\,s to
668\,s and
1542\,s, and $10\times10$, which
closes with six segments on every seed, closes on none with 12 or 24
(\emph{cert} 2.5\%--20.2\%\ and
3.9\%--11.4\%).
On $16\times16$ and $24\times24$ the 12-segment bound is
1.3\%--1.5\%\ below
the six-segment one, while the larger model finds worse incumbents on
every $16\times16$ seed and erratic ones on $24\times24$, so that
\emph{cert} grows on the first and scatters on the second. A finer interpolant is therefore worth its
cost only on the instances that close anyway; branch-and-check, whose
master carries no power curve, has no bias to remove.

\section{An irregular site with installation costs}
\label{sm:irr}

The paragraph ``Scope of the campaign'' of Section~\ref{sec:exp} of the
paper states what the formulation admits beyond the square grids of the
campaign; Section~\ref{sec:exp-irr} and Table~\ref{tab:ejor-irr} of the paper run it on the site
defined here. The site keeps the axial rose
of the paper and changes the rest. A wavy boundary blocks the cells
$(i,j)$ with $j < \mathrm{round}\bigl(n/8\,(1+\sin(2\pi i/n))\bigr)$, a coastline up
to $n/4$ cells deep; a disk of radius $0.12\,n$ centred at
$(0.6\,n, 0.55\,n)$ is an exclusion zone; and every admissible cell $c$
has a distance-based installation penalty, in power-equivalent units, $C_c = \kappa\,P_{\rm free}\,d_c/d_{\max}$, with $d_c$ the
distance from a substation at the middle of the right edge, $d_{\max}$ its
largest value, $P_{\rm free}$ the free-stream power of one turbine, and
$\kappa = 0.1$. The blocked cells leave the candidate set and still carry
the flow; the costs enter the objective as a linear term. Nothing else in
the two methods changes. Both run cold, with the grids, $N_{\max}=n+2$,
seeds, solver settings and machine of the paper, except that the band
inequalities of the monolithic model are static rows, and every returned
layout is checked to avoid the blocked cells.

The site costs 4.9\%--6.6\%
of the gross energy of the incumbents. The two methods keep the ranking
of the square site: on every grid and seed the bound of the
branch-and-check is below that of the monolithic model with static band
rows, and so is its gap. On the
$12\times12$ grid the blocked cells leave
118 candidate cells, and the gap of the
branch-and-check falls to
3.5\%--4.0\%,
against
10.0\%--11.2\%\
for the monolithic model.
On the larger grids it is
10.5\%--12.5\%,
against
16.0\%--23.0\%,
and 9.9\%--11.9\%
when the difference is taken relative to the gross energy: part of the
increase over the square site is the smaller denominator.

\section{Oblique sectors}
\label{sm:obl}

The paragraph ``Scope of the campaign'' of Section~\ref{sec:exp} of the
paper states that the maximal field, the oracle, the band inequalities and
the upwind cuts hold for any sector, while the stencil is calibrated on
axis.
Table~\ref{tab:ejor-obl} runs both exact methods on a rose of eight
equiprobable sectors, $0^\circ, 45^\circ, \dots, 315^\circ$, with the
kernel of the axial sectors applied off axis as it is: at $45^\circ$ the
cross drift of a propagation step is a whole cell. On the diagonals the
two flow components tie, and the code breaks the tie by their
floating-point comparison: the sectors at $45^\circ$ and $225^\circ$
advance along the rows, those at $135^\circ$ and $315^\circ$ along the
columns. The test therefore
measures the methods, not the physics of a diagonal wake. The grids are
$10\times10$, $16\times16$ and $24\times24$ of the paper, with
$N_{\max}=n+2$, five seeds, a cold start, the hour and the machines of
the paper, and the band inequalities of the monolithic model as static
rows.

\begin{table}[htbp]
\centering\small
\caption{Eight equiprobable sectors $0^\circ, 45^\circ, \dots, 315^\circ$, with the kernel of the axial sectors applied off axis (not recalibrated), $N_{\max} = n+2$, cold start, and the band inequalities of the monolithic model as static rows. \emph{closed} counts the seeds that close within the hour, \emph{time} is the median wall-clock time (s), \emph{root} the master relaxation of the branch-and-check before any upwind cut, \emph{bound} and \emph{best} the dual bound and the exact energy of the incumbent (medians over five seeds), \emph{gap} their difference relative to \emph{best} (range over the seeds; for the monolithic model on the SOS2 surrogate), and the last column the gap of the same method under the four axial sectors (static-row runs of the paper's campaign on the same grids). A dash under \emph{best}: no seed finds a layout with a turbine within the hour. $3600$~s.}
\label{tab:ejor-obl}
\setlength{\tabcolsep}{4.5pt}
\begin{tabular}{llrrrrrrr}
\toprule
grid & method & closed & time & root & bound & best & gap & gap, axial \\
\midrule
$10\times10$ & branch-and-check & 4/5 & 2650 & 31\,095 & 26\,665 & 26\,665 & 0.00--0.01\% & 0.00\% \\
 & monolithic + BI & 0/5 & 3600 & -- & 29\,048 & 26\,665 & 8.60--9.66\% & 2.18\% \\
\addlinespace
$16\times16$ & branch-and-check & 0/5 & 3601 & 47\,917 & 47\,867 & 42\,639 & 11.91--13.65\% & 9.26--10.06\% \\
 & monolithic + BI & 0/5 & 3600 & -- & 48\,714 & 39\,354 & 17.56--75.89\% & 24.80--52.68\% \\
\addlinespace
$24\times24$ & branch-and-check & 0/5 & 3601 & 69\,918 & 69\,918 & 62\,503 & 11.60--12.12\% & 8.15--9.43\% \\
 & monolithic + BI & 0/5 & 3600 & -- & 71\,012 & -- & -- & 34.02--121.49\% \\
\bottomrule
\end{tabular}
\end{table}

Branch-and-check closes 4\ of
its 15\ runs, all on $10\times10$,
where the open seed ends at 0.01\%;
the monolithic model closes none of its
15, and on $24\times24$ it finds
no nonempty layout within the hour, so its gap is undefined there. On every grid and seed the bound of
branch-and-check is below that of the monolithic model with static band
rows, and so is its gap, wherever the monolithic model has one. Against the axial
rose its gap grows from
9.26\%--10.06\%
to 11.91\%--13.65\%
on $16\times16$, and from
8.15\%--9.43\%
to 11.60\%--12.12\%
on $24\times24$: the master carries twice the scenarios within the same
hour.

\section{The interference-matrix campaign, seed by seed}
\label{sm:mi-seeds}

Table~\ref{tab:ejor-mi-seeds} lists every run of the interference-matrix
campaign of Section~\ref{sec:exp-mi} of the paper, whose table reports
the median over the five seeds. Its dual behaves as that of the exact
methods does: an hour of branching over millions of nodes moves its bound
by at most 1.6\%\ from the root.
Table~\ref{tab:ejor-judges} gives, grid by grid, the comparison of the
returned layouts under the three judges of the paper, with their
configurations, and, on the grids the interference-matrix model leaves
open, the certificate that the bound of branch-and-check gives its
layouts on the exact EndoWake scale.

On the cost side, counted before presolve, the nonzeros of the
interference-matrix model are 0.5\
times those of the monolithic model on
$6\times6$, exceed them from
$14\times14$\ on, and are
2.1\ times as many on
$24\times24$; the deficit matrices take at most
6.4~s, and with the loss matrices
21.2~MB; the
$12\times12$\ grid closes in
2\,610~s at most. On the quality side,
evaluated with the EndoWake propagator the interference-matrix layouts
stand up to $+$1.8\%\unskip\ above the
branch-and-check incumbents (level with them where both close); with the
deficits of the Bastankhah--Port\'e-Agel wake superposed in root-sum-square
the branch-and-check incumbents of the grids the interference-matrix
model leaves open, from $14\times14$\ on, fall between
$-$1.3\%\ and
$-$0.7\%\ of them; under the
Gaussian TurbOPark wake (PyWake class \texttt{Nygaard\_2022}, with rotor
averaging and squared-sum superposition) they score within one percent
of the interference-matrix layouts on $14\times14$ and $16\times16$ and
up to $+$2.9\%\ above them on the
two largest grids, while the interference-matrix layouts stay ahead by
2.1\%\ on
$12\times12$.

\begin{table}[htbp]
\centering\small
\caption{The interference-matrix model of Section~\ref{\PREF sec:exp-mi}, run by run: solve time and gap on its own model, number of turbines and energy of the returned layout under the common judge, for each of the five Gurobi seeds. The judge is PyWake 2.6.20, Bastankhah--Port\'e-Agel 2014 with $k^* = 0.04$, turbulence intensity $0.06$, effective wind speed, deficit at the rotor centre, linear superposition and $C_T = 8/9$; the same $k^*$ defines the interference matrix.}
\label{tab:ejor-mi-seeds}
\begin{tabular}{lr rrrr}
\toprule
grid & seed & $t$\,[s] & gap & $n_t$ & judge [kW] \\
\midrule
6$\times$6 & 1 & 0.1 & 0.00\% & 4 & 10\,689 \\
6$\times$6 & 2 & 0.1 & 0.00\% & 4 & 10\,689 \\
6$\times$6 & 3 & 0.1 & 0.00\% & 4 & 10\,689 \\
6$\times$6 & 4 & 0.1 & 0.00\% & 4 & 10\,689 \\
6$\times$6 & 5 & 0.1 & 0.00\% & 4 & 10\,689 \\
7$\times$7 & 1 & 0.5 & 0.00\% & 6 & 14\,630 \\
7$\times$7 & 2 & 0.2 & 0.00\% & 6 & 14\,630 \\
7$\times$7 & 3 & 0.2 & 0.00\% & 6 & 14\,630 \\
7$\times$7 & 4 & 0.2 & 0.00\% & 6 & 14\,630 \\
7$\times$7 & 5 & 0.2 & 0.00\% & 6 & 14\,630 \\
8$\times$8 & 1 & 1.0 & 0.00\% & 8 & 18\,487 \\
8$\times$8 & 2 & 1.1 & 0.00\% & 8 & 18\,487 \\
8$\times$8 & 3 & 0.7 & 0.00\% & 8 & 18\,487 \\
8$\times$8 & 4 & 0.8 & 0.00\% & 8 & 18\,487 \\
8$\times$8 & 5 & 0.7 & 0.00\% & 8 & 18\,487 \\
10$\times$10 & 1 & 7.1 & 0.00\% & 11 & 26\,381 \\
10$\times$10 & 2 & 6.2 & 0.00\% & 11 & 26\,381 \\
10$\times$10 & 3 & 6.0 & 0.00\% & 11 & 26\,381 \\
10$\times$10 & 4 & 5.8 & 0.00\% & 11 & 26\,381 \\
10$\times$10 & 5 & 6.9 & 0.00\% & 11 & 26\,381 \\
12$\times$12 & 1 & 2\,489.4 & 0.00\% & 14 & 33\,696 \\
12$\times$12 & 2 & 2\,476.2 & 0.00\% & 14 & 33\,696 \\
12$\times$12 & 3 & 2\,529.3 & 0.01\% & 14 & 33\,696 \\
12$\times$12 & 4 & 1\,961.2 & 0.01\% & 14 & 33\,696 \\
12$\times$12 & 5 & 2\,610.0 & 0.00\% & 14 & 33\,696 \\
14$\times$14 & 1 & 3\,600.8 & 11.43\% & 16 & 38\,760 \\
14$\times$14 & 2 & 3\,600.8 & 11.12\% & 16 & 38\,858 \\
14$\times$14 & 3 & 3\,600.6 & 10.73\% & 16 & 38\,966 \\
14$\times$14 & 4 & 3\,600.6 & 10.75\% & 16 & 38\,966 \\
14$\times$14 & 5 & 3\,601.7 & 10.73\% & 16 & 38\,966 \\
16$\times$16 & 1 & 3\,600.5 & 11.76\% & 18 & 43\,938 \\
16$\times$16 & 2 & 3\,600.5 & 11.79\% & 18 & 43\,938 \\
16$\times$16 & 3 & 3\,600.4 & 11.90\% & 18 & 43\,869 \\
16$\times$16 & 4 & 3\,600.5 & 11.82\% & 18 & 43\,913 \\
16$\times$16 & 5 & 3\,600.5 & 11.70\% & 18 & 43\,963 \\
20$\times$20 & 1 & 3\,600.2 & 12.64\% & 22 & 53\,608 \\
20$\times$20 & 2 & 3\,600.2 & 11.63\% & 22 & 54\,000 \\
20$\times$20 & 3 & 3\,600.2 & 11.58\% & 22 & 54\,025 \\
20$\times$20 & 4 & 3\,600.2 & 11.88\% & 22 & 53\,886 \\
20$\times$20 & 5 & 3\,600.2 & 11.82\% & 22 & 53\,898 \\
24$\times$24 & 1 & 3\,600.2 & 11.32\% & 26 & 63\,977 \\
24$\times$24 & 2 & 3\,600.1 & 10.93\% & 26 & 64\,182 \\
24$\times$24 & 3 & 3\,600.2 & 11.00\% & 26 & 64\,149 \\
24$\times$24 & 4 & 3\,600.1 & 12.45\% & 26 & 63\,414 \\
24$\times$24 & 5 & 3\,600.0 & 11.27\% & 26 & 63\,999 \\
\bottomrule
\end{tabular}
\end{table}

\begin{table}[htbp]
\centering\small
\caption{The three judges of Section~\ref{sec:exp-mi} of the paper: energy of the branch-and-check layouts against the interference-matrix layouts, in percent, as $100\,(\tilde E_{\rm B\&C}/\tilde E_{\rm MI} - 1)$ with $\tilde E$ the median over the five seeds, not a median of paired differences. All three are PyWake 2.6.20 at $10$\,m/s, TI $0.06$, $C_T = 8/9$, $D = 112$\,m, hub height $90$\,m, with the power curve tabulated at steps of $1$\,m/s. BPA: \texttt{BastankhahGaussianDeficit}, $k = 0.04$, effective speed, no rotor averaging, linear or squared-sum superposition. TurbOPark: class \texttt{Nygaard\_2022}, \texttt{TurboGaussianDeficit} with Gaussian-overlap rotor averaging, squared-sum superposition, free-stream reference speed. A dagger marks the grids that both formulations close. The last four columns give, on the grids the interference-matrix model leaves open, the certificate of Section~\ref{sec:exp-mi} of the paper on the exact EndoWake scale: $\tilde U_{\rm B\&C}$, the median over the seeds of the final bound of branch-and-check (Table~\ref{tab:bnc} of the paper gives the median run); $\tilde E^{\rm EW}_{\rm MI}$, the median expected power, in kW, of the interference-matrix layouts evaluated by the oracle; and the certificates $100\,(\tilde U_{\rm B\&C}/\tilde E - 1)$ of those layouts and of the branch-and-check incumbents.}
\label{tab:ejor-judges}
\begin{tabular}{lrrrrrrr}
\toprule
 & \multicolumn{3}{c}{judges} & \multicolumn{4}{c}{EndoWake certificate} \\
\cmidrule(lr){2-4}\cmidrule(lr){5-8}
grid & BPA, linear & BPA, sq.\ sum & TurbOPark & $\tilde U_{\rm B\&C}$ & $\tilde E^{\rm EW}_{\rm MI}$ & MI & B\&C \\
\midrule
6$\times$6$^\dagger$ & 0.00\% & 0.00\% & 0.00\% & -- & -- & -- & -- \\
7$\times$7$^\dagger$ & 0.00\% & 0.00\% & 0.00\% & -- & -- & -- & -- \\
8$\times$8$^\dagger$ & 0.00\% & 0.00\% & 0.00\% & -- & -- & -- & -- \\
10$\times$10$^\dagger$ & 0.00\% & 0.00\% & 0.00\% & -- & -- & -- & -- \\
12$\times$12 & $-$0.16\% & $-$0.20\% & $-$2.09\% & -- & -- & -- & -- \\
14$\times$14 & $-$1.12\% & $-$0.67\% & $+$0.64\% & 41083 & 38290 & 7.3\% & 8.5\% \\
16$\times$16 & $-$1.50\% & $-$0.98\% & $+$0.18\% & 46856 & 43447 & 7.8\% & 9.7\% \\
20$\times$20 & $-$2.17\% & $-$1.32\% & $+$2.07\% & 58075 & 53654 & 8.2\% & 9.8\% \\
24$\times$24 & $-$2.07\% & $-$1.20\% & $+$2.93\% & 69234 & 64129 & 8.0\% & 9.2\% \\
\bottomrule
\end{tabular}
\end{table}

\section{The calibration at other wake growth rates}
\label{sm:calib}

Figure~\ref{fig:calib} of the paper shows the calibration of
Section~\ref{sec:model} as a field.
The reference of the paper is the Bastankhah--Port\'e-Agel model at wake
growth rate $k^* = 0.04$. Table~\ref{tab:ejor-kstar} repeats the same
calibration, with the same window, weights and source, at growth rates
from 0.03\unskip\ to
0.075\unskip, a range that spans offshore
and onshore conditions. The spreading coefficient $\alpha$ (onto the
nearest offset) follows the growth rate, from
0.035\unskip\ to
0.113\unskip; the recovery coefficient
$\gamma$ stays within
0.117--0.123\unskip;
and the residual stays within
6.7\%--8.2\%\unskip.
On the growth rates tested, the residual thus stays in a narrow range,
rising slightly from the lowest rate before levelling off: a different atmospheric
condition changes the three coefficients, which a calibration of a few
seconds provides, and leaves the formulation and the algorithms
unchanged.

Table~\ref{tab:ejor-kst} looks at the decisions rather than at the fit.
The branch-and-check solves the $8\times 8$ and $10\times 10$ instances
of the axial four-sector rose to proven optimality with the kernel
refitted at another growth rate, and the kernel of the paper then
evaluates the resulting layout. For $k^*$ from
0.03\unskip\ to
0.05\unskip\ the optimal layout returned is that of
the paper on both grids, up to a symmetry of the grid. At
$k^* = 0.06$\unskip, where the wake
recovers faster, the optimal layout returned installs one more turbine and loses
3.9\%\unskip\ and
2.4\%\unskip\ under the kernel of the
paper; with the number of turbines fixed to that of the paper, the loss
is 2.4\%\unskip\ on $8\times 8$ and the
layout of the paper is recovered on $10\times 10$.

\begin{table}[htbp]\centering\small
\caption{The calibration of the paper repeated at other wake growth rates $k^*$ of the reference model: fitted parameters and weighted relative residual on the operational window, cell $\Delta = D$. The fitted $\beta$ is zero to four decimals in every row. $^\dagger$The value of the paper.}
\label{tab:ejor-kstar}
\begin{tabular}{lrrrr}
\toprule
$k^*$ & $\alpha$ & $\gamma$ & amplitude & residual \\
\midrule
0.03 & 0.0355 & 0.1174 & 0.843 & 6.7\% \\
0.04$^\dagger$ & 0.0532 & 0.1186 & 0.707 & 7.7\% \\
0.05 & 0.0708 & 0.1194 & 0.614 & 8.2\% \\
0.06 & 0.0881 & 0.1205 & 0.548 & 8.2\% \\
0.075 & 0.1135 & 0.1231 & 0.477 & 7.9\% \\
\bottomrule
\end{tabular}
\end{table}

\begin{table}[htbp]\centering\small
\caption{Layout decisions under a miscalibrated wake growth rate. The model is refitted at the $k^*$ of the row and solved to proven optimality; the optimal layout is then evaluated with the kernel of the paper. Loss: relative to the optimum of the paper on the same grid; a zero loss marks a layout that coincides with it up to a symmetry of the grid. $^\dagger$The value of the paper. $^\ddagger$Number of turbines fixed to that of the paper.}
\label{tab:ejor-kst}
\begin{tabular}{lrrrr}
\toprule
 & \multicolumn{2}{c}{$8\times 8$} & \multicolumn{2}{c}{$10\times 10$} \\
\cmidrule(lr){2-3}\cmidrule(lr){4-5}
$k^*$ & turbines & loss & turbines & loss \\
\midrule
0.03 & 8 & 0.0\% & 11 & 0.0\% \\
0.04$^\dagger$ & 8 & 0.0\% & 11 & 0.0\% \\
0.05 & 8 & 0.0\% & 11 & 0.0\% \\
0.06 & 9 & 3.9\% & 12 & 2.4\% \\
0.06$^\ddagger$ & 8 & 2.4\% & 11 & 0.0\% \\
\bottomrule
\end{tabular}
\end{table}

\clearpage
\section{The campaign tables, seed by seed}
\label{sm:seeds}

Every run of Tables~\ref{tab:scaling}, \ref{tab:bb}
and~\ref{tab:bnc} of the paper was repeated five
times with different Gurobi random seeds (\texttt{Seed} $= 1, \dots, 5$);
the paper reports, for each instance and for each configuration where a
table compares two, the median run of the five---the run
whose certificate is the median there, with the solve time breaking the
tie where all five close. Tables~\ref{tab:scaling} and~\ref{tab:bnc}
thus report one reproducible run per grid rather than a column-wise
average; Table~\ref{tab:bb} selects that run separately for each
configuration, so the two entries for a grid may come from different
seeds. This
section reproduces those three tables and the incumbent-posting
comparison (Section~\ref{sm:bnc-alg}) for each seed, with the same scripts and the
same formatting, the five seeds one below the other, and closes with the
dispersion across seeds per instance; the row the paper reports is one of
the five shown here. Table~\ref{tab:bounds} of the paper (the
root-relaxation study) is deterministic and is not repeated here.

\subsection{Scalability of the monolithic model}
\begin{center}\small
\setlength{\tabcolsep}{4.5pt}
\begin{longtable}{lrrrrrrrrrrr}
\caption{Table~\ref{tab:scaling} of the paper (scalability of the monolithic model, $3600$~s time limit), for each Gurobi seed.}\\
\toprule
grid & $N_{\max}$ & $n_t$ & bin & cont & rows & $t$\,[s] & gap & nodes & UB & $E$\,[kW] & cert \\
\midrule
\endfirsthead
\toprule
grid & $N_{\max}$ & $n_t$ & bin & cont & rows & $t$\,[s] & gap & nodes & UB & $E$\,[kW] & cert \\
\midrule
\endhead
\multicolumn{12}{l}{\emph{Seed 1}} \\
6$\times$6 & 8 & 4 & 36 & 1440 & 876 & 7.4 & 0.00\% & 3\,711 & 10777 & 10598 & 1.70\% \\
7$\times$7 & 9 & 6 & 49 & 1960 & 1\,197 & 17.7 & 0.00\% & 8\,874 & 14629 & 14347 & 1.97\% \\
8$\times$8 & 10 & 8 & 64 & 2560 & 1\,568 & 105.8 & 0.00\% & 32\,235 & 18391 & 18022 & 2.05\% \\
10$\times$10 & 12 & 11 & 100 & 4000 & 2\,460 & 3600.0 & 1.72\% & 385\,718 & 26699 & 25689 & 3.93\% \\
12$\times$12 & 14 & 14 & 144 & 5760 & 3\,552 & 3600.0 & 15.64\% & 136\,811 & 38815 & 32915 & 17.93\% \\
14$\times$14 & 16 & 16 & 196 & 7840 & 4\,844 & 3600.0 & 19.22\% & 96\,626 & 44778 & 36760 & 21.81\% \\
16$\times$16 & 18 & 18 & 256 & 10240 & 6\,336 & 3600.0 & 19.28\% & 59\,933 & 50117 & 41134 & 21.84\% \\
20$\times$20 & 22 & 22 & 400 & 16000 & 9\,920 & 3600.0 & 24.73\% & 28\,291 & 61245 & 48042 & 27.48\% \\
24$\times$24 & 26 & 14 & 576 & 23040 & 14\,304 & 3600.0 & 131.39\% & 7\,434 & 72365 & 33164 & 118.21\% \\
\midrule
\multicolumn{12}{l}{\emph{Seed 2}} \\
6$\times$6 & 8 & 4 & 36 & 1440 & 876 & 12.4 & 0.00\% & 10\,316 & 10777 & 10598 & 1.70\% \\
7$\times$7 & 9 & 6 & 49 & 1960 & 1\,197 & 21.4 & 0.00\% & 9\,771 & 14629 & 14347 & 1.97\% \\
8$\times$8 & 10 & 8 & 64 & 2560 & 1\,568 & 226.3 & 0.00\% & 89\,723 & 18391 & 18022 & 2.05\% \\
10$\times$10 & 12 & 11 & 100 & 4000 & 2\,460 & 2841.1 & 0.00\% & 319\,260 & 26248 & 25689 & 2.18\% \\
12$\times$12 & 14 & 14 & 144 & 5760 & 3\,552 & 3600.0 & 19.71\% & 123\,296 & 38958 & 31852 & 22.31\% \\
14$\times$14 & 16 & 16 & 196 & 7840 & 4\,844 & 3600.0 & 20.41\% & 81\,107 & 44656 & 36264 & 23.14\% \\
16$\times$16 & 18 & 18 & 256 & 10240 & 6\,336 & 3600.0 & 16.20\% & 66\,450 & 50128 & 42223 & 18.72\% \\
20$\times$20 & 22 & 22 & 400 & 16000 & 9\,920 & 3600.0 & 30.13\% & 29\,681 & 61242 & 46043 & 33.01\% \\
24$\times$24 & 26 & 22 & 576 & 23040 & 14\,304 & 3600.0 & 39.52\% & 8\,520 & 72461 & 50847 & 42.51\% \\
\midrule
\multicolumn{12}{l}{\emph{Seed 3}} \\
6$\times$6 & 8 & 4 & 36 & 1440 & 876 & 20.4 & 0.00\% & 19\,564 & 10777 & 10598 & 1.70\% \\
7$\times$7 & 9 & 6 & 49 & 1960 & 1\,197 & 17.3 & 0.00\% & 7\,676 & 14629 & 14347 & 1.97\% \\
8$\times$8 & 10 & 8 & 64 & 2560 & 1\,568 & 225.5 & 0.00\% & 73\,356 & 18391 & 18022 & 2.05\% \\
10$\times$10 & 12 & 11 & 100 & 4000 & 2\,460 & 3600.0 & 0.33\% & 402\,008 & 26336 & 25689 & 2.52\% \\
12$\times$12 & 14 & 14 & 144 & 5760 & 3\,552 & 3600.0 & 15.62\% & 128\,429 & 38808 & 32915 & 17.91\% \\
14$\times$14 & 16 & 16 & 196 & 7840 & 4\,844 & 3600.0 & 16.76\% & 88\,023 & 44714 & 37491 & 19.27\% \\
16$\times$16 & 18 & 18 & 256 & 10240 & 6\,336 & 3600.0 & 19.11\% & 68\,457 & 50282 & 41312 & 21.71\% \\
20$\times$20 & 22 & 22 & 400 & 16000 & 9\,920 & 3600.0 & 17.12\% & 25\,494 & 61472 & 51335 & 19.75\% \\
24$\times$24 & 26 & 14 & 576 & 23040 & 14\,304 & 3600.0 & 170.93\% & 7\,682 & 72391 & 33298 & 117.41\% \\
\midrule
\multicolumn{12}{l}{\emph{Seed 4}} \\
6$\times$6 & 8 & 4 & 36 & 1440 & 876 & 9.5 & 0.00\% & 3\,591 & 10777 & 10598 & 1.70\% \\
7$\times$7 & 9 & 6 & 49 & 1960 & 1\,197 & 19.6 & 0.00\% & 8\,900 & 14629 & 14347 & 1.97\% \\
8$\times$8 & 10 & 8 & 64 & 2560 & 1\,568 & 127.6 & 0.00\% & 38\,707 & 18391 & 18022 & 2.05\% \\
10$\times$10 & 12 & 11 & 100 & 4000 & 2\,460 & 3111.8 & 0.00\% & 381\,434 & 26248 & 25689 & 2.18\% \\
12$\times$12 & 14 & 14 & 144 & 5760 & 3\,552 & 3600.0 & 16.27\% & 125\,227 & 38738 & 32608 & 18.80\% \\
14$\times$14 & 16 & 16 & 196 & 7840 & 4\,844 & 3600.0 & 17.49\% & 91\,839 & 44619 & 37173 & 20.03\% \\
16$\times$16 & 18 & 18 & 256 & 10240 & 6\,336 & 3600.0 & 17.79\% & 57\,915 & 50074 & 41618 & 20.32\% \\
20$\times$20 & 22 & 22 & 400 & 16000 & 9\,920 & 3600.0 & 18.00\% & 25\,208 & 61241 & 50923 & 20.26\% \\
24$\times$24 & 26 & 19 & 576 & 23040 & 14\,304 & 3600.0 & 77.75\% & 8\,489 & 72373 & 43038 & 68.16\% \\
\midrule
\multicolumn{12}{l}{\emph{Seed 5}} \\
6$\times$6 & 8 & 4 & 36 & 1440 & 876 & 7.0 & 0.00\% & 3\,483 & 10777 & 10598 & 1.70\% \\
7$\times$7 & 9 & 6 & 49 & 1960 & 1\,197 & 21.6 & 0.00\% & 8\,951 & 14629 & 14347 & 1.97\% \\
8$\times$8 & 10 & 8 & 64 & 2560 & 1\,568 & 173.5 & 0.00\% & 60\,741 & 18391 & 18022 & 2.05\% \\
10$\times$10 & 12 & 11 & 100 & 4000 & 2\,460 & 3471.5 & 0.00\% & 453\,348 & 26248 & 25689 & 2.18\% \\
12$\times$12 & 14 & 14 & 144 & 5760 & 3\,552 & 3600.0 & 17.12\% & 106\,557 & 38859 & 32438 & 19.79\% \\
14$\times$14 & 16 & 16 & 196 & 7840 & 4\,844 & 3600.0 & 17.67\% & 90\,262 & 44617 & 37121 & 20.20\% \\
16$\times$16 & 18 & 18 & 256 & 10240 & 6\,336 & 3600.0 & 18.03\% & 63\,048 & 50212 & 41650 & 20.56\% \\
20$\times$20 & 22 & 22 & 400 & 16000 & 9\,920 & 3600.0 & 24.14\% & 23\,575 & 61215 & 48309 & 26.72\% \\
24$\times$24 & 26 & 22 & 576 & 23040 & 14\,304 & 3600.0 & 40.98\% & 7\,775 & 72373 & 50324 & 43.81\% \\
\bottomrule
\end{longtable}
\end{center}

\clearpage
\subsection{The band inequalities inside the search}
\begin{center}\small
\begin{longtable}{lrrrrrrr}
\caption{Table~\ref{tab:bb} of the paper (the band inequalities inside the search: dual bound, certificate on the exact energy scale and nodes without and with them; where an instance closes, the certificate column gives the time to prove it, in italics), for each Gurobi seed.}\\
\toprule
& & \multicolumn{3}{c}{none} & \multicolumn{3}{c}{$+$BI} \\
\cmidrule(lr){3-5}\cmidrule(lr){6-8}
grid & $N_{\max}$ & bound & cert\,/\,$t$ & nodes & bound & cert\,/\,$t$ & nodes \\
\midrule
\endfirsthead
\toprule
& & \multicolumn{3}{c}{none} & \multicolumn{3}{c}{$+$BI} \\
\cmidrule(lr){3-5}\cmidrule(lr){6-8}
grid & $N_{\max}$ & bound & cert\,/\,$t$ & nodes & bound & cert\,/\,$t$ & nodes \\
\midrule
\endhead
\multicolumn{8}{l}{\emph{Seed 1}} \\
$6\times6$ & 8 & 10777 & \emph{7\,s} & 3\,711 & 10777 & \emph{9\,s} & 3\,157 \\
$7\times7$ & 9 & 14629 & \emph{18\,s} & 8\,874 & 14629 & \emph{127\,s} & 47\,673 \\
$8\times8$ & 10 & 18391 & \emph{106\,s} & 32\,235 & 18391 & \emph{185\,s} & 49\,953 \\
$10\times10$ & 12 & 26699 & 3.93\% & 385\,718 & 26248 & \emph{1060\,s} & 129\,886 \\
$12\times12$ & 14 & 38815 & 17.93\% & 136\,811 & 36725 & 12.61\% & 59\,448 \\
$14\times14$ & 16 & 44778 & 21.81\% & 96\,626 & 42274 & 14.44\% & 36\,584 \\
$16\times16$ & 18 & 50117 & 21.84\% & 59\,933 & 47838 & 19.95\% & 29\,086 \\
$20\times20$ & 22 & 61245 & 27.48\% & 28\,291 & 59100 & 21.05\% & 12\,825 \\
$24\times24$ & 26 & 72365 & 118.21\% & 7\,434 & 70393 & 59.55\% & 6\,974 \\
\midrule
\multicolumn{8}{l}{\emph{Seed 2}} \\
$6\times6$ & 8 & 10777 & \emph{12\,s} & 10\,316 & 10777 & \emph{29\,s} & 5\,281 \\
$7\times7$ & 9 & 14629 & \emph{22\,s} & 9\,771 & 14629 & \emph{157\,s} & 57\,200 \\
$8\times8$ & 10 & 18391 & \emph{226\,s} & 89\,723 & 18391 & \emph{197\,s} & 51\,986 \\
$10\times10$ & 12 & 26248 & \emph{2841\,s} & 319\,260 & 26248 & \emph{1132\,s} & 135\,081 \\
$12\times12$ & 14 & 38958 & 22.31\% & 123\,296 & 36776 & 13.29\% & 50\,540 \\
$14\times14$ & 16 & 44656 & 23.14\% & 81\,107 & 42274 & 14.70\% & 40\,174 \\
$16\times16$ & 18 & 50128 & 18.72\% & 66\,450 & 47838 & 12.18\% & 27\,659 \\
$20\times20$ & 22 & 61242 & 33.01\% & 29\,681 & 59100 & 24.88\% & 11\,234 \\
$24\times24$ & 26 & 72461 & 42.51\% & 8\,520 & 70393 & 46.55\% & 6\,308 \\
\midrule
\multicolumn{8}{l}{\emph{Seed 3}} \\
$6\times6$ & 8 & 10777 & \emph{21\,s} & 19\,564 & 10777 & \emph{28\,s} & 8\,643 \\
$7\times7$ & 9 & 14629 & \emph{17\,s} & 7\,676 & 14629 & \emph{171\,s} & 58\,521 \\
$8\times8$ & 10 & 18391 & \emph{226\,s} & 73\,356 & 18391 & \emph{379\,s} & 120\,993 \\
$10\times10$ & 12 & 26336 & 2.52\% & 402\,008 & 26248 & \emph{1185\,s} & 152\,822 \\
$12\times12$ & 14 & 38808 & 17.91\% & 128\,429 & 36781 & 13.94\% & 59\,711 \\
$14\times14$ & 16 & 44714 & 19.27\% & 88\,023 & 42274 & 13.22\% & 29\,728 \\
$16\times16$ & 18 & 50282 & 21.71\% & 68\,457 & 47838 & 26.59\% & 30\,381 \\
$20\times20$ & 22 & 61472 & 19.75\% & 25\,494 & 59100 & 37.58\% & 9\,712 \\
$24\times24$ & 26 & 72391 & 117.41\% & 7\,682 & 70393 & 54.02\% & 6\,257 \\
\midrule
\multicolumn{8}{l}{\emph{Seed 4}} \\
$6\times6$ & 8 & 10777 & \emph{10\,s} & 3\,591 & 10777 & \emph{13\,s} & 4\,488 \\
$7\times7$ & 9 & 14629 & \emph{20\,s} & 8\,900 & 14629 & \emph{129\,s} & 46\,392 \\
$8\times8$ & 10 & 18391 & \emph{128\,s} & 38\,707 & 18391 & \emph{218\,s} & 60\,544 \\
$10\times10$ & 12 & 26248 & \emph{3112\,s} & 381\,434 & 26248 & \emph{2251\,s} & 252\,261 \\
$12\times12$ & 14 & 38738 & 18.80\% & 125\,227 & 36677 & 12.90\% & 42\,824 \\
$14\times14$ & 16 & 44619 & 20.03\% & 91\,839 & 42274 & 13.88\% & 41\,457 \\
$16\times16$ & 18 & 50074 & 20.32\% & 57\,915 & 47838 & 13.19\% & 22\,695 \\
$20\times20$ & 22 & 61241 & 20.26\% & 25\,208 & 59100 & 25.09\% & 11\,450 \\
$24\times24$ & 26 & 72373 & 68.16\% & 8\,489 & 70393 & 56.73\% & 6\,266 \\
\midrule
\multicolumn{8}{l}{\emph{Seed 5}} \\
$6\times6$ & 8 & 10777 & \emph{7\,s} & 3\,483 & 10777 & \emph{52\,s} & 16\,302 \\
$7\times7$ & 9 & 14629 & \emph{22\,s} & 8\,951 & 14629 & \emph{297\,s} & 120\,690 \\
$8\times8$ & 10 & 18391 & \emph{174\,s} & 60\,741 & 18391 & \emph{151\,s} & 43\,105 \\
$10\times10$ & 12 & 26248 & \emph{3472\,s} & 453\,348 & 26248 & \emph{1520\,s} & 236\,754 \\
$12\times12$ & 14 & 38859 & 19.79\% & 106\,557 & 36747 & 13.11\% & 60\,711 \\
$14\times14$ & 16 & 44617 & 20.20\% & 90\,262 & 42274 & 12.27\% & 40\,231 \\
$16\times16$ & 18 & 50212 & 20.56\% & 63\,048 & 47838 & 24.82\% & 19\,843 \\
$20\times20$ & 22 & 61215 & 26.72\% & 23\,575 & 59100 & 27.48\% & 11\,694 \\
$24\times24$ & 26 & 72373 & 43.81\% & 7\,775 & 70393 & 45.46\% & 6\,314 \\
\bottomrule
\end{longtable}
\end{center}

\clearpage
\subsection{Branch-and-check}
\begin{center}\small
\begin{longtable}{lrrrrrrrrr}
\caption{Table~\ref{tab:bnc} of the paper (branch-and-check, one row per instance), for each Gurobi seed.}\\
\toprule
grid & $N_{\max}$ & root & bound & best & $n_t$ & gap & nodes & cuts & $t$ [s] \\
\midrule
\endfirsthead
\toprule
grid & $N_{\max}$ & root & bound & best & $n_t$ & gap & nodes & cuts & $t$ [s] \\
\midrule
\endhead
\multicolumn{10}{l}{\emph{Seed 1}} \\
$6\times6$ & 8 & 10643 & 10598 & 10598 & 4 & 0.00\% & 120 & 82 & 0.2 \\
$7\times7$ & 9 & 16022 & 14347 & 14347 & 6 & 0.00\% & 483 & 483 & 0.9 \\
$8\times8$ & 10 & 20273 & 18022 & 18022 & 8 & 0.00\% & 2\,870 & 1\,475 & 3.0 \\
$10\times10$ & 12 & 29604 & 25689 & 25689 & 11 & 0.00\% & 48\,090 & 7\,113 & 78 \\
$12\times12$ & 14 & 36262 & 35411 & 32793 & 14 & 7.98\% & 71\,504 & 43\,860 & 3600 \\
$14\times14$ & 16 & 41474 & 41083 & 37876 & 16 & 8.47\% & 49\,962 & 56\,246 & 3600 \\
$16\times16$ & 18 & 46957 & 46856 & 42573 & 18 & 10.06\% & 46\,908 & 58\,729 & 3601 \\
$20\times20$ & 22 & 58075 & 58075 & 53036 & 22 & 9.50\% & 46\,207 & 71\,240 & 3600 \\
$24\times24$ & 26 & 69234 & 69234 & 63429 & 26 & 9.15\% & 65\,156 & 90\,091 & 3601 \\
\midrule
\multicolumn{10}{l}{\emph{Seed 2}} \\
$6\times6$ & 8 & 10643 & 10598 & 10598 & 4 & 0.00\% & 97 & 87 & 0.2 \\
$7\times7$ & 9 & 16022 & 14347 & 14347 & 6 & 0.00\% & 545 & 506 & 0.9 \\
$8\times8$ & 10 & 20273 & 18022 & 18022 & 8 & 0.00\% & 2\,742 & 1\,388 & 2.6 \\
$10\times10$ & 12 & 29604 & 25689 & 25689 & 11 & 0.00\% & 46\,200 & 7\,069 & 76 \\
$12\times12$ & 14 & 36262 & 35411 & 32915 & 14 & 7.59\% & 78\,727 & 42\,021 & 3600 \\
$14\times14$ & 16 & 41474 & 41083 & 37743 & 16 & 8.85\% & 48\,328 & 56\,895 & 3600 \\
$16\times16$ & 18 & 46957 & 46856 & 42591 & 18 & 10.02\% & 49\,466 & 59\,560 & 3600 \\
$20\times20$ & 22 & 58075 & 58075 & 52866 & 22 & 9.85\% & 44\,195 & 75\,387 & 3601 \\
$24\times24$ & 26 & 69234 & 69234 & 64017 & 26 & 8.15\% & 39\,987 & 90\,467 & 3601 \\
\midrule
\multicolumn{10}{l}{\emph{Seed 3}} \\
$6\times6$ & 8 & 10643 & 10598 & 10598 & 4 & 0.00\% & 121 & 94 & 0.2 \\
$7\times7$ & 9 & 16022 & 14347 & 14347 & 6 & 0.00\% & 466 & 434 & 0.8 \\
$8\times8$ & 10 & 20273 & 18022 & 18022 & 8 & 0.00\% & 2\,717 & 1\,385 & 2.6 \\
$10\times10$ & 12 & 29604 & 25689 & 25689 & 11 & 0.00\% & 44\,626 & 7\,294 & 75 \\
$12\times12$ & 14 & 36262 & 35411 & 32753 & 14 & 8.12\% & 72\,384 & 43\,077 & 3600 \\
$14\times14$ & 16 & 41474 & 41083 & 37878 & 16 & 8.46\% & 55\,764 & 59\,182 & 3600 \\
$16\times16$ & 18 & 46957 & 46856 & 42885 & 18 & 9.26\% & 48\,846 & 65\,055 & 3601 \\
$20\times20$ & 22 & 58075 & 58075 & 52873 & 22 & 9.84\% & 47\,472 & 78\,538 & 3600 \\
$24\times24$ & 26 & 69234 & 69234 & 63321 & 26 & 9.34\% & 38\,439 & 83\,043 & 3600 \\
\midrule
\multicolumn{10}{l}{\emph{Seed 4}} \\
$6\times6$ & 8 & 10643 & 10598 & 10598 & 4 & 0.00\% & 97 & 87 & 0.2 \\
$7\times7$ & 9 & 16022 & 14347 & 14347 & 6 & 0.00\% & 472 & 473 & 0.8 \\
$8\times8$ & 10 & 20273 & 18022 & 18022 & 8 & 0.00\% & 2\,644 & 1\,365 & 2.5 \\
$10\times10$ & 12 & 29604 & 25689 & 25689 & 11 & 0.00\% & 45\,091 & 7\,089 & 79 \\
$12\times12$ & 14 & 36262 & 35411 & 32793 & 14 & 7.98\% & 73\,839 & 40\,567 & 3600 \\
$14\times14$ & 16 & 41474 & 41083 & 37867 & 16 & 8.49\% & 46\,007 & 54\,383 & 3600 \\
$16\times16$ & 18 & 46957 & 46856 & 42884 & 18 & 9.26\% & 44\,468 & 58\,999 & 3600 \\
$20\times20$ & 22 & 58075 & 58075 & 52660 & 22 & 10.28\% & 64\,288 & 74\,564 & 3601 \\
$24\times24$ & 26 & 69234 & 69234 & 63268 & 26 & 9.43\% & 68\,376 & 88\,749 & 3601 \\
\midrule
\multicolumn{10}{l}{\emph{Seed 5}} \\
$6\times6$ & 8 & 10643 & 10598 & 10598 & 4 & 0.00\% & 109 & 88 & 0.2 \\
$7\times7$ & 9 & 16022 & 14347 & 14347 & 6 & 0.00\% & 532 & 414 & 0.7 \\
$8\times8$ & 10 & 20273 & 18022 & 18022 & 8 & 0.00\% & 2\,655 & 1\,364 & 2.6 \\
$10\times10$ & 12 & 29604 & 25689 & 25689 & 11 & 0.00\% & 46\,111 & 6\,805 & 77 \\
$12\times12$ & 14 & 36262 & 35411 & 32793 & 14 & 7.98\% & 83\,394 & 42\,728 & 3600 \\
$14\times14$ & 16 & 41474 & 41083 & 37812 & 16 & 8.65\% & 49\,993 & 57\,126 & 3600 \\
$16\times16$ & 18 & 46957 & 46856 & 42698 & 18 & 9.74\% & 45\,879 & 60\,032 & 3600 \\
$20\times20$ & 22 & 58075 & 58075 & 52954 & 22 & 9.67\% & 51\,415 & 79\,893 & 3601 \\
$24\times24$ & 26 & 69234 & 69234 & 63811 & 26 & 8.50\% & 42\,933 & 90\,563 & 3600 \\
\bottomrule
\end{longtable}
\end{center}

\clearpage
\subsection{Incumbent posting}
\begin{center}\small\setlength{\tabcolsep}{4pt}
\begin{longtable}{lrrrrrrrrrrr}
\caption{The incumbent-posting comparison of Section~\ref{sm:bnc-alg}, for each Gurobi seed: on the left the exact incumbent is tracked outside the solver, on the right it is posted; $3600$~s time limit, and a dagger marks a run stopped by it. Posting wins 19 of the 25 (instance, seed) pairs, ties 5 and loses 1, where a win is a shorter time to prove optimality when both runs close, and a lower final gap otherwise; differences below $0.5$~s or $0.01$ gap points count as ties.}\\
\toprule
& & \multicolumn{5}{c}{tracked outside} & \multicolumn{5}{c}{posted} \\
\cmidrule(lr){3-7}\cmidrule(lr){8-12}
grid & seed & bound & exact & gap & $t$\,[s] & nodes & bound & exact & gap & $t$\,[s] & nodes \\
\midrule
\endfirsthead
\toprule
& & \multicolumn{5}{c}{tracked outside} & \multicolumn{5}{c}{posted} \\
\cmidrule(lr){3-7}\cmidrule(lr){8-12}
grid & seed & bound & exact & gap & $t$\,[s] & nodes & bound & exact & gap & $t$\,[s] & nodes \\
\midrule
\endhead
$8\times8$ & 1 & 18022 & 18022 & 0.00\% & 3.0 & 2\,839 & 18022 & 18022 & 0.00\% & 2.8 & 2\,870 \\
$8\times8$ & 2 & 18022 & 18022 & 0.00\% & 3.0 & 2\,839 & 18022 & 18022 & 0.00\% & 2.6 & 2\,742 \\
$8\times8$ & 3 & 18022 & 18022 & 0.00\% & 3.5 & 3\,084 & 18022 & 18022 & 0.00\% & 2.6 & 2\,717 \\
$8\times8$ & 4 & 18022 & 18022 & 0.00\% & 2.9 & 2\,689 & 18022 & 18022 & 0.00\% & 2.5 & 2\,644 \\
$8\times8$ & 5 & 18022 & 18022 & 0.00\% & 2.9 & 2\,741 & 18022 & 18022 & 0.00\% & 2.8 & 2\,655 \\
\midrule
$10\times10$ & 1 & 25689 & 25689 & 0.00\% & 80 & 46\,379 & 25689 & 25689 & 0.00\% & 77 & 48\,090 \\
$10\times10$ & 2 & 25689 & 25689 & 0.00\% & 81 & 45\,269 & 25689 & 25689 & 0.00\% & 76 & 46\,200 \\
$10\times10$ & 3 & 25689 & 25689 & 0.00\% & 142 & 56\,133 & 25689 & 25689 & 0.00\% & 76 & 44\,626 \\
$10\times10$ & 4 & 25689 & 25689 & 0.00\% & 96 & 48\,034 & 25689 & 25689 & 0.00\% & 79 & 45\,091 \\
$10\times10$ & 5 & 25689 & 25689 & 0.00\% & 136 & 52\,345 & 25689 & 25689 & 0.00\% & 77 & 46\,111 \\
\midrule
$12\times12$ & 1 & 35411 & 32793 & 7.98\%$^{\dagger}$ & 3600 & 65\,379 & 35411 & 32793 & 7.98\%$^{\dagger}$ & 3600 & 71\,610 \\
$12\times12$ & 2 & 35411 & 32475 & 9.04\%$^{\dagger}$ & 3600 & 62\,711 & 35411 & 32915 & 7.59\%$^{\dagger}$ & 3600 & 78\,779 \\
$12\times12$ & 3 & 35411 & 32491 & 8.99\%$^{\dagger}$ & 3600 & 78\,088 & 35411 & 32753 & 8.12\%$^{\dagger}$ & 3600 & 72\,314 \\
$12\times12$ & 4 & 35411 & 32751 & 8.12\%$^{\dagger}$ & 3600 & 68\,214 & 35411 & 32793 & 7.98\%$^{\dagger}$ & 3600 & 73\,910 \\
$12\times12$ & 5 & 35411 & 32690 & 8.32\%$^{\dagger}$ & 3600 & 62\,261 & 35411 & 32793 & 7.98\%$^{\dagger}$ & 3600 & 83\,530 \\
\midrule
$14\times14$ & 1 & 41083 & 37835 & 8.58\%$^{\dagger}$ & 3600 & 55\,750 & 41083 & 37876 & 8.47\%$^{\dagger}$ & 3600 & 49\,962 \\
$14\times14$ & 2 & 41083 & 37497 & 9.56\%$^{\dagger}$ & 3600 & 47\,498 & 41083 & 37743 & 8.85\%$^{\dagger}$ & 3600 & 48\,239 \\
$14\times14$ & 3 & 41083 & 37650 & 9.12\%$^{\dagger}$ & 3600 & 51\,650 & 41083 & 37878 & 8.46\%$^{\dagger}$ & 3600 & 55\,730 \\
$14\times14$ & 4 & 41083 & 37990 & 8.14\%$^{\dagger}$ & 3600 & 48\,258 & 41083 & 37867 & 8.49\%$^{\dagger}$ & 3600 & 46\,007 \\
$14\times14$ & 5 & 41083 & 37449 & 9.70\%$^{\dagger}$ & 3600 & 55\,776 & 41083 & 37812 & 8.65\%$^{\dagger}$ & 3600 & 49\,993 \\
\midrule
$16\times16$ & 1 & 46856 & 42339 & 10.67\%$^{\dagger}$ & 3601 & 47\,844 & 46856 & 42573 & 10.06\%$^{\dagger}$ & 3601 & 46\,948 \\
$16\times16$ & 2 & 46856 & 42268 & 10.86\%$^{\dagger}$ & 3600 & 51\,194 & 46856 & 42591 & 10.02\%$^{\dagger}$ & 3600 & 49\,457 \\
$16\times16$ & 3 & 46856 & 42222 & 10.97\%$^{\dagger}$ & 3600 & 52\,001 & 46856 & 42885 & 9.26\%$^{\dagger}$ & 3601 & 48\,846 \\
$16\times16$ & 4 & 46856 & 42418 & 10.46\%$^{\dagger}$ & 3600 & 49\,148 & 46856 & 42884 & 9.26\%$^{\dagger}$ & 3600 & 44\,433 \\
$16\times16$ & 5 & 46856 & 42405 & 10.50\%$^{\dagger}$ & 3600 & 46\,431 & 46856 & 42698 & 9.74\%$^{\dagger}$ & 3600 & 45\,914 \\
\bottomrule
\end{longtable}
\end{center}

\clearpage
\subsection{Dispersion across seeds}
\begin{table}[htbp]
\centering\small
\caption{Dispersion across the five seeds, per instance. For the monolithic
model without (\emph{none}) and with (\emph{$+$BI}) the band inequalities,
and for the branch-and-check (\emph{B\&C}), the entry is the range over
the five seeds of the certificate on the exact energy scale where the
instance does not close (of the final gap, for the branch-and-check,
whose bound already lives on that scale), and the
range of the time to prove optimality where it does; the last column is
the spread of the branch-and-check incumbent, $(\max - \min)/\min$ of its
exact energy over the seeds. A single value means that the five seeds
agree to the digits shown; a lower bound means that on at least one seed
the run ended with no incumbent, so its gap is unbounded.}
\begin{tabular}{lrrrr}
\toprule
grid & none & $+$BI & B\&C gap\,/\,$t$ & B\&C incumbent \\
\midrule
$6\times6$ & 7--21\,s & 9--52\,s & 0\,s & 0.00\% \\
$7\times7$ & 17--22\,s & 127--297\,s & 1\,s & 0.00\% \\
$8\times8$ & 106--226\,s & 151--379\,s & 2--3\,s & 0.00\% \\
$10\times10$ & 2.18--3.93\% & 1060--2251\,s & 75--79\,s & 0.00\% \\
$12\times12$ & 17.91--22.31\% & 12.61--13.94\% & 7.59--8.12\% & 0.49\% \\
$14\times14$ & 19.27--23.14\% & 12.27--14.70\% & 8.46--8.85\% & 0.36\% \\
$16\times16$ & 18.72--21.84\% & 12.18--26.59\% & 9.26--10.06\% & 0.73\% \\
$20\times20$ & 19.75--33.01\% & 21.05--37.58\% & 9.50--10.28\% & 0.71\% \\
$24\times24$ & 42.51--118.21\% & 45.46--59.55\% & 8.15--9.43\% & 1.18\% \\
\bottomrule
\end{tabular}
\end{table}


\begin{thebibliography}{99}

\bibitem[Angulo et~al.(2016)]{angulo2016}
Angulo, G., Ahmed, S., \& Dey, S.~S. (2016).
Improving the integer L-shaped method.
{\em INFORMS Journal on Computing}, 28(3), 483--499.

\bibitem[Archer et~al.(2011)]{archer2011}
Archer, R., Nates, G., Donovan, S., \& Waterer, H. (2011).
Wind turbine interference in a wind farm layout optimization mixed integer
linear programming model.
{\em Wind Engineering}, 35(2), 165--175.

\bibitem[Baker et~al.(2019)]{baker2019}
Baker, N.~F., Stanley, A.~P.~J., Thomas, J.~J., Ning, A., \& Dykes, K.
(2019).
Best practices for wake model and optimization algorithm selection in
wind farm layout optimization.
In {\em AIAA Scitech 2019 Forum} (AIAA 2019-0540). San Diego, CA.

\bibitem[Bastankhah and Port\'e-Agel(2014)]{bastankhah2014}
Bastankhah, M., \& Port\'e-Agel, F. (2014).
A new analytical model for wind-turbine wakes.
{\em Renewable Energy}, 70, 116--123.

\bibitem[Bastankhah et~al.(2021)]{bastankhah2021}
Bastankhah, M., Welch, B.~L., Mart\'inez-Tossas, L.~A., King, J., \&
Fleming, P. (2021).
Analytical solution for the cumulative wake of wind turbines in wind
farms.
{\em Journal of Fluid Mechanics}, 911, A53.

\bibitem[Beale and Tomlin(1970)]{beale1970}
Beale, E.~M.~L., \& Tomlin, J.~A. (1970).
Special facilities in a general mathematical programming system for
non-convex problems using ordered sets of variables.
In J.~Lawrence (Ed.), {\em Proceedings of the Fifth International
Conference on Operational Research} (pp. 447--454). Tavistock Publications.

\bibitem[Beck(2010)]{beck2010}
Beck, J.~C. (2010).
Checking-up on branch-and-check.
In {\em Principles and Practice of Constraint Programming -- CP 2010},
Lecture Notes in Computer Science 6308 (pp.~84--98). Springer.

\bibitem[Cao et~al.(2025)]{cao2025}
Cao, Z., Chen, L., Li, K., Huang, G., \& Liu, R. (2025).
A mathematical programming approach for joint optimization of wind farm
layout and cable routing based on a three-dimensional Gaussian wake
model.
{\em Wind Energy}, 28(2), e2960.

\bibitem[Cazzaro et~al.(2023)]{cazzaro2023}
Cazzaro, D., Koza, D.~F., \& Pisinger, D. (2023).
Combined layout and cable optimization of offshore wind farms.
{\em European Journal of Operational Research}, 311(1), 301--315.

\bibitem[Cazzaro and Pisinger(2022)]{cazzaro2022vns}
Cazzaro, D., \& Pisinger, D. (2022).
Variable neighborhood search for large offshore wind farm layout
optimization.
{\em Computers \& Operations Research}, 138, Article 105588.

\bibitem[Cazzaro et~al.(2022)]{cazzaro2022multiscale}
Cazzaro, D., Trivella, A., Corman, F., \& Pisinger, D. (2022).
Multi-scale optimization of the design of offshore wind farms.
{\em Applied Energy}, 314, Article 118830.

\bibitem[Codato and Fischetti(2006)]{codato2006}
Codato, G., \& Fischetti, M. (2006).
Combinatorial Benders' cuts for mixed-integer linear programming.
{\em Operations Research}, 54(4), 756--766.

\bibitem[Donovan(2005)]{donovan2005}
Donovan, S. (2005).
Wind farm optimization.
In {\em Proceedings of the 40th Annual ORSNZ Conference}.

\bibitem[Fagerfj\"all(2010)]{fagerfjall2010}
Fagerfj\"all, P. (2010).
{\em Optimizing wind farm layout---more bang for the buck using
mixed integer linear programming} [Master's thesis]. Chalmers University
of Technology.

\bibitem[Fischetti and Fischetti(2023)]{fischetti2023integrated}
Fischetti, Martina, \& Fischetti, Matteo (2023).
Integrated layout and cable routing in wind farm optimal design.
{\em Management Science}, 69(4), 2147--2164.

\bibitem[Fischetti and Fischetti(2026a)]{companion}
Fischetti, Martina, \& Fischetti, Matteo (2026a).
GPU-accelerated wind farm layout search with a distilled endogenous wake
model (EndoWake). Preprint, arXiv:2610.07844 [math.OC], version~3.

\bibitem[Fischetti and Fischetti(2026b)]{repo}
Fischetti, Martina, \& Fischetti, Matteo (2026b).
{\em Replication package for the present paper: model, exact evaluator,
instances, layouts, campaign outputs, seeds and table generators}
[Data set], version 2.0.2. Zenodo. \url{https://doi.org/10.5281/zenodo.22017423}.

\bibitem[Fischetti and Fischetti(2026c)]{dwcor}
Fischetti, Martina, \& Fischetti, Matteo (2026c).
A Dantzig--Wolfe reformulation of the EndoWake model for optimal wind
farm layouts. Preprint, arXiv:2610.39503 [math.OC].

\bibitem[Fischetti et~al.(2020)]{fischetti2020vattenfall}
Fischetti, M., Kristoffersen, J.~R., Hjort, T., Monaci, M., \&
Pisinger, D. (2020).
Vattenfall optimizes offshore wind farm design.
{\em INFORMS Journal on Applied Analytics}, 50(1), 80--94.

\bibitem[Fischetti and Monaci(2014)]{fischetti2014proximity}
Fischetti, M., \& Monaci, M. (2014).
Proximity search for 0--1 mixed-integer convex programming.
{\em Journal of Heuristics}, 20(6), 709--731.

\bibitem[Fischetti and Monaci(2016)]{fischetti2016proximity}
Fischetti, M., \& Monaci, M. (2016).
Proximity search heuristics for wind farm optimal layout.
{\em Journal of Heuristics}, 22(4), 459--474.

\bibitem[Fischetti and Pisinger(2019)]{fischetti2019overview}
Fischetti, M., \& Pisinger, D. (2019).
Mathematical optimization and algorithms for offshore wind farm design:
an overview.
{\em Business \& Information Systems Engineering}, 61(4), 469--485.

\bibitem[Forbes et~al.(2024)]{forbes2024}
Forbes, M., Harris, M., Jansen, M., van der Schoot, F., \& Taimre, T.
(2024).
Combining optimisation and simulation using logic-based Benders
decomposition.
{\em European Journal of Operational Research}, 312(3), 840--854.

\bibitem[Frangioni and Gentile(2006)]{frangioni2006}
Frangioni, A., \& Gentile, C. (2006).
Perspective cuts for a class of convex 0--1 mixed integer programs.
{\em Mathematical Programming}, 106(2), 225--236.

\bibitem[Gnegel et~al.(2021)]{gnegel2021}
Gnegel, F., F\"ugenschuh, A., Hagel, M., Leyffer, S., \& Stiemer, M.
(2021).
A solution framework for linear PDE-constrained mixed-integer problems.
{\em Mathematical Programming}, 188(2), 695--728.

\bibitem[Grady et~al.(2005)]{grady2005}
Grady, S.~A., Hussaini, M.~Y., \& Abdullah, M.~M. (2005).
Placement of wind turbines using genetic algorithms.
{\em Renewable Energy}, 30(2), 259--270.

\bibitem[Guirguis et~al.(2016)]{guirguis2016}
Guirguis, D., Romero, D.~A., \& Amon, C.~H. (2016).
Toward efficient optimization of wind farm layouts: utilizing exact
gradient information.
{\em Applied Energy}, 179, 110--123.

\bibitem[G\"unl\"uk and Linderoth(2010)]{gunluk2010}
G\"unl\"uk, O., \& Linderoth, J. (2010).
Perspective reformulations of mixed integer nonlinear programs with
indicator variables.
{\em Mathematical Programming}, 124(1--2), 183--205.

\bibitem[Gurobi Optimization, LLC(2026)]{gurobi}
Gurobi Optimization, LLC (2026).
{\em Gurobi Optimizer Reference Manual} (Version 13.0).
\texttt{https://www.gurobi.com/}

\bibitem[Hijazi et~al.(2012)]{hijazi2012}
Hijazi, H., Bonami, P., Cornu\'ejols, G., \& Ouorou, A. (2012).
Mixed-integer nonlinear programs featuring ``on/off'' constraints.
{\em Computational Optimization and Applications}, 52(2), 537--558.

\bibitem[Hooker and Ottosson(2003)]{hooker2003}
Hooker, J.~N., \& Ottosson, G. (2003).
Logic-based Benders decomposition.
{\em Mathematical Programming}, 96, 33--60.

\bibitem[Hu et~al.(2026)]{hu2026}
Hu, J., Xu, J., Meng, W., Yang, Q., \& Grossmann, I.~E. (2026).
Integrated optimization of layout and cable routing for wind farms: From
MINLP to MIQCP formulation and bilevel decomposition method.
{\em Computers \& Chemical Engineering}, 214, Article 109814.

\bibitem[Jensen(1983)]{jensen1983}
Jensen, N.~O. (1983).
{\em A note on wind generator interaction} (Technical Report
Ris\o-M-2411). Ris\o\ National Laboratory.

\bibitem[Katic et~al.(1986)]{katic1986}
Katic, I., H\o jstrup, J., \& Jensen, N.~O. (1986).
A simple model for cluster efficiency.

In W.~Palz \& E.~Sesto (Eds.), {\em EWEC'86: Proceedings of the
European Wind Energy Association Conference and Exhibition, Rome, 6--8
October 1986} (Vol.~1, pp.~407--410). Rome: A.~Raguzzi (published 1987).

\bibitem[King et~al.(2017)]{king2017}
King, R.~N., Dykes, K., Graf, P., \& Hamlington, P.~E. (2017).
Optimization of wind plant layouts using an adjoint approach.
{\em Wind Energy Science}, 2(1), 115--131.

\bibitem[Kuo et~al.(2016)]{kuo2016}
Kuo, J.~Y.~J., Romero, D.~A., Beck, J.~C., \& Amon, C.~H. (2016).
Wind farm layout optimization on complex terrains---integrating a CFD wake
model with mixed-integer programming.
{\em Applied Energy}, 178, 404--414.

\bibitem[Lanzilao and Meyers(2022)]{lanzilao2022}
Lanzilao, L., \& Meyers, J. (2022).
A new wake-merging method for wind-farm power prediction in the presence
of heterogeneous background velocity fields.
{\em Wind Energy}, 25(2), 237--259.

\bibitem[Laporte and Louveaux(1993)]{laporte1993}
Laporte, G., \& Louveaux, F.~V. (1993).
The integer L-shaped method for stochastic integer programs with complete
recourse.
{\em Operations Research Letters}, 13(3), 133--142.

\bibitem[Li et~al.(2026)]{li2025multitype}
Li, T., Yang, B., Yang, Q., Zhang, M., Hu, W., \& Guo, K. (2026).
Wind farm layout optimization with multi-type wind turbines for minimizing
levelized cost of energy.
{\em Renewable Energy}, 270, 124386.

\bibitem[Lissaman(1979)]{lissaman1979}
Lissaman, P.~B.~S. (1979).
Energy effectiveness of arbitrary arrays of wind turbines.
{\em Journal of Energy}, 3(6), 323--328.

\bibitem[LoCascio et~al.(2024)]{locascio2024}
LoCascio, M.~J., Bay, C.~J., Mart\'inez-Tossas, L.~A., Thomas, J.~J.,
\& Gorl\'e, C. (2024).
FLOWERS AEP: an analytical model for wind farm layout optimization.
{\em Wind Energy}, 27, e2954.

\bibitem[Mosetti et~al.(1994)]{mosetti1994}
Mosetti, G., Poloni, C., \& Diviacco, B. (1994).
Optimization of wind turbine positioning in large windfarms by means of a
genetic algorithm.
{\em Journal of Wind Engineering and Industrial Aerodynamics}, 51(1),
105--116.

\bibitem[Nygaard et~al.(2020)]{nygaard2020}
Nygaard, N.~G., Steen, S.~T., Poulsen, L., \& Pedersen, J.~G. (2020).
Modelling cluster wakes and wind farm blockage.
{\em Journal of Physics: Conference Series}, 1618, 062072.

\bibitem[Parada et~al.(2024)]{parada2024}
Parada, L., Legault, R., C\^ot\'e, J.-F., \& Gendreau, M. (2024).
A disaggregated integer L-shaped method for stochastic vehicle routing
problems with monotonic recourse.
{\em European Journal of Operational Research}, 318(2), 520--533.

\bibitem[Pedersen et~al.(2022)]{pedersen2022}
Pedersen, J.~G., Svensson, E., Poulsen, L., \& Nygaard, N.~G. (2022).
Turbulence Optimized Park model with Gaussian wake profile.
{\em Journal of Physics: Conference Series}, 2265, 022063.

\bibitem[Pedersen et~al.(2026a)]{pedersen2026}
Pedersen, J.~W., Lindner, N., Rehfeldt, D., \& Koch, T. (2026a).
Integrated wind farm design: optimizing turbine placement and cable
routing with wake effects.
{\em OR Spectrum}. \url{https://doi.org/10.1007/s00291-026-00862-1}.

\bibitem[Pedersen et~al.(2026b)]{pywake}
Pedersen, M.~M., Meyer Forsting, A., van der Laan, P., Riva, R., Alcayaga
Rom\'an, L.~A., Criado Risco, J., et al. (2026b).
{\em PyWake: an open-source wind farm simulation tool, version 2.6.20}
[Software]. DTU Wind.
\url{https://gitlab.windenergy.dtu.dk/TOPFARM/PyWake}

\bibitem[P\'erez-R\'ua et~al.(2023)]{perezrua2023}
P\'erez-R\'ua, J.-A., Stolpe, M., \& Cutululis, N.~A. (2023).
A neighborhood search integer programming approach for wind farm layout
optimization.
{\em Wind Energy Science}, 8(9), 1453--1473.

\bibitem[Quaeghebeur et~al.(2021)]{quaeghebeur2021}
Quaeghebeur, E., Bos, R., \& Zaaijer, M.~B. (2021).
Wind farm layout optimization using pseudo-gradients.
{\em Wind Energy Science}, 6(3), 815--839.

\bibitem[Rahmaniani et~al.(2017)]{rahmaniani2017}
Rahmaniani, R., Crainic, T.~G., Gendreau, M., \& Rei, W. (2017).
The Benders decomposition algorithm: A literature review.
{\em European Journal of Operational Research}, 259(3), 801--817.

\bibitem[Stanley and Ning(2019)]{stanley2019}
Stanley, A.~P.~J., \& Ning, A. (2019).
Massive simplification of the wind farm layout optimization problem.
{\em Wind Energy Science}, 4(4), 663--676.

\bibitem[Thorsteinsson(2001)]{thorsteinsson2001}
Thorsteinsson, E.~S. (2001).
Branch-and-check: a hybrid framework integrating mixed integer programming
and constraint logic programming.
In T.~Walsh (Ed.), {\em Principles and Practice of Constraint
Programming---CP 2001} (Lecture Notes in Computer Science, Vol. 2239,
pp. 16--30). Springer.

\bibitem[Turner et~al.(2014)]{turner2014}
Turner, S.~D.~O., Romero, D.~A., Zhang, P.~Y., Amon, C.~H., \&
Chan, T.~C.~Y. (2014).
A new mathematical programming approach to optimize wind farm layouts.
{\em Renewable Energy}, 63, 674--680.

\bibitem[Ulku and Alabas-Uslu(2019)]{ulku2019}
Ulku, I., \& Alabas-Uslu, C. (2019).
A new mathematical programming approach to wind farm layout problem under
multiple wake effects.
{\em Renewable Energy}, 136, 1190--1201.

\bibitem[Zhang et~al.(2014)]{zhang2014}
Zhang, P.~Y., Romero, D.~A., Beck, J.~C., \& Amon, C.~H. (2014).
Solving wind farm layout optimization with mixed integer programs and
constraint programs.
{\em EURO Journal on Computational Optimization}, 2, 195--219.

\bibitem[Zhang et~al.(2026)]{zhang2026}
Zhang, Z., He, J., Liu, Y., Li, Q., \& Kareem, A. (2026).
A new physically based model for cumulative wakes and application to
wind farms.
{\em Energy Conversion and Management}, 361, 121593.

\bibitem[Zong and Port\'e-Agel(2020)]{zong2020}
Zong, H., \& Port\'e-Agel, F. (2020).
A momentum-conserving wake superposition method for wind farm power
prediction.
{\em Journal of Fluid Mechanics}, 889, A8.

\end{thebibliography}
\end{document}